\documentclass[a4paper,reqno]{amsart}
\usepackage{amsmath,amsfonts,amssymb,amsthm}
\usepackage{txfonts}
\usepackage{eucal}
\usepackage{graphicx}
\usepackage{amscd}
\usepackage[all]{xy}
\usepackage{quiver}
\usepackage{verbatim}
\usepackage{mathrsfs}
\usepackage{latexsym}
\usepackage{xspace}

\usepackage{epsfig}
\usepackage{float}

\usepackage{color}
\usepackage{fancybox}
\usepackage{colordvi}
\usepackage{multicol}
\usepackage[colorlinks, linkcolor=blue!72,anchorcolor=orange,
citecolor=red,urlcolor=Emerald, bookmarksopen,bookmarksdepth=2]{hyperref}
\usepackage{url}
\usepackage{mathtools}

\newcommand{\xdashrightarrow}[2][]{%
  \mathrel{\tikz[baseline=-0.5ex] \draw[-{Stealth[scale=1]}, dashed, dash pattern=on 3pt off 2pt] (0,0) -- node[above=-1pt] {$\scriptstyle #2$} node[below=-1pt] {$\scriptstyle #1$} (20pt,0);}%
}

\newcommand{\deleted}[1]{}
\newcommand{\delete}[1]{}
\newcommand{\mynotes}[1]{}
\newcommand\notes[1]{}

\newtheorem{theorem}{Theorem}[section]
\newtheorem{lemma}[theorem]{Lemma}
\newtheorem{corollary}[theorem]{Corollary}

\newtheorem{prop}[theorem]{Proposition}
\theoremstyle{definition}
\newtheorem{definition}[theorem]{Definition}
\newtheorem{remark}[theorem]{Remark}
\newtheorem{construction}[theorem]{Construction}
\newtheorem{exam}[theorem]{Example}

\deleted{\newtheorem{theorem}{Theorem}[section]
	\newtheorem{prop-def}{Proposition-Definition}[section]
	\newtheorem{coro-def}{Corollary-Definition}[section]

}
\newtheorem{prop-def}[theorem]{Proposition-Definition}
\numberwithin{equation}{section}

\def\mod{\operatorname{mod}}
\def\proj{\operatorname{proj}}
\def\add{\operatorname{add}}
\def\ker{\operatorname{Ker}}

\def\im{\operatorname{Im}}
\def\End{\operatorname{End}}
\def\Hom{\operatorname{Hom}}
\def\Pn{\mathbb{P}}
\def\tPn{\widetilde{\mathbb{P}}}
\def\Omegan{\Omega}
\def\tGamma{\widetilde{\Gamma}}

\def\k{\mathbb{K}}
\def\e{\mathbb{E}}
\def\c{\mathcal{C}}

\def\t{\mathcal{T}}
\def\f{\mathcal{F}}
\def\m{\mathcal{M}}
\def\p{\mathcal{P}}
\def\n{\mathcal{N}}
\def\X{\mathcal{X}}
\def\Y{\mathcal{Y}}
\def\M{\widetilde{\mathcal{M}}}

\def\Ext{\operatorname{Ext}}

\def\nprt{\mathrm{pr}_{\mathcal{T}}^{n+1}}

\def\id{\operatorname{Id}} 

\newcommand{\pr}[3]{\operatorname{pr}_{#1}^{#2}(#3)}
\newcommand{\pj}[1]{\mathcal{P}^{\leq{#1}}(\c)}
\newcommand{\cplx}[2]{C^{\bullet}(#1;#2)}
\newcommand{\cmap}[3]{\id_{#1}(#2,#3)}
\def\tu{\widetilde{u}}
\def\tv{\widetilde{v}}
\def\tw{\widetilde{w}}
\def\tz{\widetilde{z}}
\def\tf{\widetilde{f}}
\def\tg{\widetilde{g}}
\def\th{\widetilde{h}}
\def\tf{\widetilde{f}}
\def\tF{\widetilde{F}}
\def\tbeta{\widetilde{\beta}}
\def\talpha{\widetilde{\alpha}}
\def\tgamma{\widetilde{\gamma}}
\def\tdelta{\widetilde{\delta}}
\def\tDelta{\widetilde{\Delta}}
\def\ti{\widetilde{i}}
\def\tp{\widetilde{p}}

\title{From $(n+1)$-term subcategories to $(n+1)$-term complexes}

\author{Zhaotai Zhang}
\address{Department of Mathematical Sciences, 
Tsinghua University, 
100084 Beijing, 
China}
\email{ztzhangmath@163.com}
	
\author{Yu Zhou}
\address{School of Mathematical Sciences,
Beijing Normal University,
100875 Beijing,
China}
\email{yuzhoumath@gmail.com}
	
\author{Bin Zhu}
\address{Department of Mathematical Sciences, 
Tsinghua University, 
100084 Beijing, 
China}
\email{zhu-b@mail.tsinghua.edu.cn}

\dedicatory{Dedicated to the memory of Idun Reiten}


\thanks{The work was supported by National Natural Science Foundation of China (Grants No. 12271279;\  12371034).}

\begin{document}

\begin{abstract}
    Let $n$ be a positive integer. Given an $n$-rigid subcategory $\mathcal{M}$ of an algebraic triangulated category $\mathcal{T}$, we explicitly construct an extriangulated functor from the $(n+1)$-term subcategory of $\mathcal{T}$ generated by $\mathcal{M}$ to the full subcategory of $(n+1)$-term complexes in the bounded homotopy category $K^b(\mathcal{M})$, which restricts to the identity on $\mathcal{M}$. In the broader context of reduced $(n-1)$-Auslander extriangulated categories, we provide necessary and sufficient conditions for such a functor to be full, in which case it induces an equivalence of extriangulated categories modulo a certain ideal. Furthermore, we establish a mutation-compatible bijection between the silting subcategories of these categories. Finally, we apply these results to $n$-cluster tilting subcategories and $n$-cluster tilting objects in $(n+1)$-Calabi-Yau categories.
\end{abstract}
	
\maketitle
	
\tableofcontents

\section{Introduction}

As a natural extension of classical tilting theory in the representation theory of algebras, cluster tilting theory arose from the additive categorification of Fomin-Zelevinsky's cluster algebras \cite{FZ}. The key notions in this categorification are cluster tilting objects (or subcategories) and their mutations. These concepts were first introduced in cluster categories by Buan-Marsh-Reineke-Reiten-Todorov~\cite{BMRRT} (see \cite{CCS} for cluster categories of type $A_n$), and subsequently investigated in more general $2$-Calabi-Yau triangulated categories by Keller \cite{K1,K2}, Iyama-Yoshino \cite{IY}, and Gei{\ss}-Leclerc-Schr\"oer \cite{GLS}. One of the central problems in cluster tilting theory is to understand the relationships between the combinatorial mutation structure of cluster tilting objects in a $2$-Calabi-Yau triangulated category, the cluster algebras associated with their quivers, and the representation-theoretic properties of their endomorphism algebras \cite{AIR,KR07,KR08,R}.

To approach these relationships, a fundamental tool is the $\tau$-tilting theory introduced in \cite{AIR}. Let $\k$ be a field, and let $T$ be a cluster tilting object in a $\k$-linear Hom-finite Krull-Schmidt $2$-Calabi-Yau triangulated category $\t$. Let $\Lambda=\End_{\t}(T)$ be the endomorphism algebra of $T$, and let $\mod\Lambda$ be the category of finitely generated right $\Lambda$-modules. The functor $\Hom_{\t}(T,-)\colon \t\rightarrow\mod\Lambda$ induces a mutation-compatible bijection between the set of (basic) cluster tilting objects in $\t$ and the set of (basic) support $\tau$-tilting modules over $\Lambda$.  Notably, this connection is part of a broader algebraic property: for any finite-dimensional $\k$-algebra (and thus for $\Lambda$), the $0$-th cohomology functor $H^{0}(-)\colon K^{[-1,0]}(\proj\Lambda)\rightarrow\mod\Lambda$ induces a mutation-compatible bijection between the set of (basic) $2$-term silting objects in $K^{b}(\proj\Lambda)$ and the set of (basic) support $\tau$-tilting modules over $\Lambda$. Here, $K^{[-1,0]}(\proj\Lambda)$ denotes the full subcategory of the homotopy category $K^b(\proj\Lambda)$ consisting of complexes isomorphic to $2$-term complexes of the form $P_{-1}\rightarrow P_0$, where $P_{-1}, P_0\in\proj\Lambda$, and $\proj\Lambda$ is the subcategory of $\mod\Lambda$ consisting of finitely generated projective modules.
	
Recent developments have generalized this correspondence by dropping the $2$-Calabi-Yau assumption and weakening the cluster tilting condition. Let $\t$ be an arbitrary $\k$-linear Hom-finite Krull-Schmidt triangulated category with suspension functor $\Sigma$, and $T$ a rigid object in $\t$ with endomorphism algebra $\Lambda = \End_{\t}(T)$. By introducing the $2$-term subcategory $\c_T = (\add T) * (\add \Sigma T)$, the functor $\Hom_{\t}(T,-)\colon \c_{T} \rightarrow \mod\Lambda$ \cite{IY, KZ} induces a mutation-compatible bijection between the set of (basic) maximal relative rigid objects in $\c_T$ and the set of (basic) support $\tau$-tilting modules over $\Lambda$ \cite{YZ,FGL}. It is worth noting that this generalization naturally recovers the aforementioned setting: if $T$ is cluster tilting, then $\c_T=\t$, and when $\t$ is $2$-Calabi-Yau, the maximal relative rigid objects coincide precisely with the cluster tilting objects \cite{YZ,FGL}.
	
Recently, Yang \cite{Yang} explicitly constructed a $\k$-linear functor $\mathbb{P}$ from the $2$-term subcategory $\c_T$ to $K^{[-1,0]}(\proj\Lambda)$ in the case where the triangulated category $\t$ is algebraic (i.e., the stable category of a Frobenius exact category). This functor completes the following commutative triangle:
\[\begin{tikzcd}
	\c_{T} \arrow[rrd, "{\Hom_\t(T,-)}"'] \arrow[rr, "\mathbb{P}"] &  & {K^{[-1,0]}(\proj\Lambda)} \arrow[d, "H^{0}(-)"] \\
	&  & \mod\Lambda.                                   
\end{tikzcd}\]
We also refer to \cite{Chen1,FGPPP} for related results.

The main aim of this paper is to generalize Yang's functor to $(n+1)$-term subcategories. Let $\t$ be a triangulated category with suspension functor $\Sigma$. Let $n$ be a positive integer, and let $\m$ be an $n$-rigid subcategory of $\t$; that is, $\Hom_{\t}(\m,\Sigma^{i} \m)=0$ for all $1\leq i\leq n$. We assume that $\m$ is closed under direct summands. The $(n+1)$-term subcategory 
\[\pr{\t}{n+1}{\m}:=(\m)*(\Sigma \m)*\cdots*(\Sigma^{n}\m)\] 
admits an extriangulated structure (see Proposition~\ref{prop:extri}). Let $K^{[-n,0]}(\m)$ be the full subcategory of the bounded homotopy category $K^b(\m)$ consisting of objects isomorphic to complexes concentrated in degrees $-n$ to $0$. This category inherits a natural extriangulated structure from the triangulated structure of $K^b(\m)$. Under the assumption that $\t$ is algebraic, we construct an extriangulated functor (see Definition~\ref{defn:extri fun}) between these two extriangulated categories.

\begin{theorem}[Theorem~\ref{preserve}]\label{thm:exist}
    Suppose $\t$ is an algebraic triangulated category. Then there exists an extriangulated functor
    \[(\Pn,\Omegan)\colon\pr{\t}{n+1}{\m}\rightarrow K^{[-n,0]}(\m),\]
    whose restriction to $\m$ is the identity; that is,  $\Pn|_\m=\id_\m$.
\end{theorem}

The extriangulated category $\pr{\t}{n+1}{\m}$ is reduced $(n-1)$-Auslander (see Proposition~\ref{prop:reduced}), with its projective objects being precisely those in $\m$.  Conversely, as shown in \cite[Theorem~3.7]{Chen2} every algebraic reduced $(n-1)$-Auslander extriangulated category arises in this way up to extriangulated equivalence.  Consequently, given any algebraic reduced $(n-1)$-Auslander extriangulated category $(\c,\e)$, Theorem~\ref{thm:exist} yields an extriangulated functor from $\c$ to $K^{[-n,0]}(\m)$ that restricts to the identity functor on $\m$, where $\m$ denotes the full subcategory of projective objects in $\c$. We provide several equivalent conditions for such a functor to be full, in which case it induces an extriangle equivalence from an ideal quotient of $\c$ to $K^{[-n,0]}(\m)$.

\begin{theorem}[Theorems~\ref{finalequiv}, \ref{newcor} and \ref{silt_corres}]\label{thm:main}
    Let $(\c,\e)$ be a reduced $(n-1)$-Auslander extriangulated category, and let $\m$ denote its full subcategory of projective objects. Let
    $$(\Pn,\Omegan)\colon\c\to K^{[-n,0]}(\m)$$
    be an extriangulated functor whose restriction to $\m$ is the identity ($\Pn|_\m=\id_\m$). Then the following are equivalent:
    \begin{enumerate}
        \item[(i)] The functor $\Pn$ is full.
        \item[(ii)] The map $\Omegan^{X,Y}$ is injective for all objects $X,Y\in\c$.
        \item[(iii)] $\Hom_{\c}(\Sigma^{i}\m,\m)=0$ for all $1\leq i\leq n-1$.
    \end{enumerate}
    If any of these equivalent conditions holds, then $\Pn$ is dense, and $\Omegan$ is a natural isomorphism, and this extriangulated functor induces an extriangle equivalence: 
    \[\c/[\Sigma^n \m,\m]\rightarrow K^{[-n,0]}(\m),\]
    where $[\Sigma^n \m,\m]$ denotes the ideal of $\c$ consisting of morphisms factoring through a morphism from an object in $\Sigma^{n}\m$ to an object in $\m$. 
    
    Moreover, under these same conditions, the functor $\Pn$ induces a bijection between the set of silting subcategories of $\c$ and the set of silting subcategories of 
	$K^{[-n,0]}(\m)$ (that is, the set of $(n+1)$-term silting subcategories of $K^b(\m)$), which is compatible with mutation.	
\end{theorem}

Applying this theorem to the case $\c=\pr{\t}{n+1}{\m}$, we obtain corresponding results (see Theorem~\ref{thm:algtri}) for the functor constructed in Theorem~\ref{thm:exist}. In particular, for $n=1$, setting $\m = \add T$, reduces the category $\pr{\t}{2}{\m}$ to the $2$-term subcategory $\c_T$. In this setting, condition (ii) is automatically satisfied, and thus our result recovers \cite[Theorem~1.1]{Yang}.

As a further application, we generalize \cite[Theorem~1.2]{Yang} to arbitrary $n$. One one hand, Corollary~\ref{cor:application} extends part (a) to the setting where $\m$ is an $n$-cluster tilting subcategory. On the other hand, Corollary~\ref{cor:application2} generalizes part (b) to the case where $\t$ is $(n+1)$-Calabi-Yau and $\m$ is the additive hull of an $n$-cluster tilting object, using a formulation adapted to the higher-dimensional setting to bypass characterizations specific to the $n=1$ case.
	
\begin{corollary}[Corollary~\ref{cor:cluster-tilting}]\label{cor:application}
    Let $\t$ be an algebraic triangulated category with an $n$-cluster tilting subcategory $\m$. Suppose that $\Hom_{\t}(\Sigma^{i}\m,\m)=0$ for all $1\leq i\leq n-1$. Then there exists an extriangulated functor
	\[
		(\Pn,\Omegan)\colon \t\rightarrow K^{[-n,0]}(\m),
	\]
	which induces an extriangle equivalence
	\[
		\t/[\Sigma^n\m,\m]\xrightarrow{\sim}K^{[-n,0]}(\m). 
	\]
	Moreover, the functor $\Pn$ is an equivalence if and only if $\Sigma^{n+1}\m=\m$. In this case, $K^{[-n,0]}(\m)$ admits a triangulated structure. If $\t$ is $\k$-linear and admits a Serre functor $\mathbb{S}$, then $\Pn$ is an equivalence if and only if $\m=\mathbb{S}\m$, which is also equivalent to $\m$ being self-injective. 
\end{corollary}

\begin{corollary}[Corollary~\ref{cor:cto}]\label{cor:application2}
    Let $\t$ be a Hom-finite Krull-Schmidt $(n+1)$-Calabi-Yau algebraic triangulated category, and let $M$ be an $n$-cluster tilting object in $\t$ satisfying $\Hom_{\t}(\Sigma^i M,M)=0$ for all $1\leq i\leq n-1$. Then the functor $\Pn$ induces a bijection between the isoclasses of basic $2$-term $n$-cluster tilting objects in $\t$ and the isoclasses of basic $2$-term silting complexes in $K^b(\proj\End_\t(M))$, which is compatible with mutation. Here, an object $X$ in $\t$ is called $2$-term if $X\in(\add M)*\add(\Sigma M)$.

    Assume in addition that $\Hom_\t(\Sigma^n M,M)=0$. Then the above bijection restricts to a bijection between the isoclasses of basic $2$-term $n$-cluster tilting objects $X$ in $\t$ satisfying $\Hom_\t(\Sigma X,X)=0$ and the isoclasses of basic $2$-term tilting complexes in $K^b(\proj\End_\t(M))$. Moreover, for any $2$-term $n$-cluster tilting object $X$ in $\t$, $\Hom_\t(\Sigma X, X)=0$ if and only if $\Hom_\t(\Sigma^i X,X)=0$ for all $1\leq i\leq n$.
\end{corollary}

At the final stage of the preparation of this paper, we learned of a forthcoming work by Lior Silberberg \cite{S}, who has independently obtained some related results.
	
The paper is organized as follows. In Section~\ref{section2}, we recall some basic notions of extriangulated categories and study the two extriangulated categories $\pr{\t}{n+1}{\m}$ and $K^{[-n,0]}(\m)$ associated with an $n$-rigid subcategory $\m$ of a triangulated category $\t$. In Section~\ref{section3}, assuming that $\t$ is algebraic, we construct a functor $\Pn\colon\pr{\t}{n+1}{\m}\rightarrow K^{[-n,0]}(\m)$. In Section~\ref{sec:extri}, we prove that $\Pn$ can be lifted to an extriangulated functor. In Section~\ref{section5}, we deduce our main result, Theorem~\ref{thm:main}. Finally, in Section~\ref{section6}, we apply this result to the settings of $n$-cluster tilting subcategories and $n$-cluster tilting objects in $(n+1)$-Calabi-Yau categories.

\subsection*{Convention}

Throughout this paper, all categories are additive. All subcategories are assumed to be full, additive, and closed under isomorphisms. For a category $\c$, we write $X\in\c$ for an object in $\c$, and $\m\subseteq\c$ for a subcategory of $\c$. For a subcategory $\X\subseteq\c$, we denote by $[\X]$ the ideal of $\c$ consisting of morphisms factoring through an object in $\X$. For two subcategories $\X,\Y\subseteq\c$, we denote by $[\X,\Y]$ the ideal of $\c$ consisting of morphisms factoring through a morphism from an object in $\X$ to an object in $\Y$. For a class $\X$ of objects in $\c$, let $\add\X$ denote the subcategory of $\c$ whose objects are direct summands of finite direct sums of objects in $\X$. If $\X=\{X\}$, we simply write $\add X$ instead of $\add\{X\}$. The composition of morphisms $f\colon X \rightarrow Y$ and $g\colon Y \rightarrow Z$ is denoted by $gf$ (or $g \circ f$). For any two subcategories $\X$ and $\Y$ of a triangulated category $\t$, we denote by $\X*\Y$ the subcategory of $\t$ consisting of the objects $Z$ such that there exists a triangle $X\to Z\to Y\to \Sigma X$ with $X\in\X$ and $Y\in\Y$.

\section{Preliminaries}\label{section2}

In this section, we first recall some basic notions of extriangulated categories and then focus on two extriangulated categories arising from a certain subcategory of a triangulated category.

\subsection{Basic notions of extriangulated categories}

An \emph{extriangulated} category is a triple $(\mathcal{C},\mathbb{E},\mathfrak{s})$, where 
\begin{itemize}
    \item $\c$ is an additive category,
    \item $\e\colon\c^{\mathrm{op}}\times\c\rightarrow\mathrm{Ab}$ is a biadditive functor taking values in the category $\mathrm{Ab}$ of abelian groups, and
    \item $\mathfrak{s}$ is an additive realization sending each element $\xi\in\e(X,Y)$ to an equivalence class of sequences of morphisms $Y\to Z\to X$.
\end{itemize}
This triple is subject to certain axioms; see \cite[Definition 2.12]{NP} or \cite[Appendix A]{GNP2}.

For any $w\in\e(X,Y)$ and any sequence $Y\xrightarrow{u} Z\xrightarrow{v} X$ in $\mathfrak{s}(w)$, the sequence
\[Y\xrightarrow{u} Z\xrightarrow{v} X \xdashrightarrow{w}\]
is called an \emph{extriangle}. The group $\e(X,Y)$ is called the \emph{extension group} from $X$ to $Y$. When the additive realization $\mathfrak{s}$ is clear from the context, we simply denote the extriangulated category by $(\c,\e)$.

Let $w\in\e(X,Y)$. For $f\in\Hom_{\c}(X',X)$ and $g\in\Hom_{\c}(Y,Y')$, we denote $\e(f,Y)(w)$ and $\e(X,g)(w)$ by $f^{*}(w)$ and $g_{*}(w)$, respectively. Moreover, by Yoneda's lemma, the element $w$ of $\e(X,Y)$ induces natural transformations $w_{\#}:\Hom_{\c}(-,X)\rightarrow\e(-,Y)$ and $w^{\#}:\Hom_{\c}(Y,-)\rightarrow\e(X,-)$. For any $C\in\c$, the components $(w_{\#})_{C}$ and $(w^{\#})_{C}$ are given by:
\[
(w_{\#})_{C}: \Hom_{\c}(C,X)\rightarrow\e(C,Y);\quad f\mapsto f^{*}(w),
\]
\[
(w^{\#})_{C}:\Hom_{\c}(Y,C)\rightarrow\e(X,C);\quad g\mapsto g_{*}(w). 
\]


\begin{prop}[{\cite[Proposition A.8]{GNP2}}]\label{def:exactseq}
Let $(\c,\e)$ be an extriangulated category and $Y\xrightarrow{u}Z\xrightarrow{v}X\xdashrightarrow{w}$ be an extriangle in $\c$. Then the following sequences of natural transformations are exact:
\[
\Hom_{\c}(X,-)\xrightarrow{-\circ v}\Hom_{\c}(Z,-)\xrightarrow{-\circ u}\Hom_{\c}(Y,-)\xrightarrow{w^{\#}}\e(X,-)\xrightarrow{v^{*}}\e(Z,-)\xrightarrow{u^{*}}\e(Y,-),
\]
\[
\Hom_{\c}(-,Y)\xrightarrow{u\circ -}\Hom_{\c}(-,Z)\xrightarrow{v\circ -}\Hom_{\c}(-,X)\xrightarrow{w_{\#}}\e(-,Y)\xrightarrow{u_{*}}\e(-,Z)\xrightarrow{v_{*}}\e(-,X).
\]
\end{prop}

\begin{definition}\label{defn:extri fun}
    Let $(\c,\e)$ and $(\c',\e')$ be two extriangulated categories. An \emph{extriangulated functor} from $(\c,\e)$ to $(\c',\e')$ is a pair $(\mathbb{P},\Omega)$, where $\mathbb{P}$ is an additive functor from $\c$ to $\c'$, and $\Omega=\{\Omega^{X,Y}\mid X,Y\in\c\}$ is a natural transformation
    \[\Omega\colon\e(-,-)\to \e'(\mathbb{P}^{\mathrm{op}}(-),\mathbb{P}(-))\]
    between functors $\c^{\mathrm{op}}\times\c\rightarrow\mathrm{Ab}$, such that if
    \[Y\xrightarrow{u} Z\xrightarrow{v} X \xdashrightarrow{w}\]
    is an extriangle in $\c$, then 
    \[\mathbb{P}(Y)\xrightarrow{\mathbb{P}(u)} \mathbb{P}(Z)\xrightarrow{\mathbb{P}(v)} \mathbb{P}(X) \xdashrightarrow{\Omega^{X,Y}(w)}\]
    is an extriangle in $\c'$.

    An extriangulated functor $(\mathbb{P},\Omega)$ is called an \emph{extriangle equivalence} if $\mathbb{P}$ is an equivalence of additive categories and $\Omega$ is a natural isomorphism.
\end{definition}

Note that while our definition of extriangulated functors coincides with that in \cite[Definition~2.32]{BS}, our definition of extriangle equivalences differs slightly, as we additionally require $\Omega$ to be a natural isomorphism; see also \cite[Definition~4.9 and Proposition~4.11]{BHSS}.

\begin{remark}\label{rmk:detectextri}
    Let $(\mathbb{P},\Omega)\colon(\c,\e)\to(\c',\e')$ be an extriangulated functor such that $\mathbb{P}$ is dense and each $\Omega^{X,Y}\colon\e(X,Y)\to \e'(\mathbb{P}(X),\mathbb{P}(Y))$ is surjective. Then $\mathbb{P}$ reflects extriangles as follows. Let \begin{equation}\label{eq:extri2}
        Y'\xrightarrow{u'} Z'\xrightarrow{v'} X'\xdashrightarrow{w'}
    \end{equation} 
    be an extriangle in $(\c',\e')$. Since $\mathbb{P}$ is dense, there exist objects $X, Y \in \c$ and isomorphisms $a\colon \mathbb{P}(Y) \to Y'$ and $c\colon \mathbb{P}(X) \to X'$. Consider the element $\e'(c, a^{-1})(w') \in \e'(\mathbb{P}(X), \mathbb{P}(Y))$. Since the map $\Omega^{X,Y}$ is surjective, there exists an element $w \in \e(X,Y)$ such that 
    \[ \Omega^{X,Y}(w) = \e'(c, a^{-1})(w'). \]
    Equivalently, this relation can be written as
    \begin{equation}\label{eq:equiv}
        \e'(\mathbb{P}(X),a)(\Omega^{X,Y}(w)) = \e'(c,Y')(w').
    \end{equation}
    Let $Y\xrightarrow{u} Z\xrightarrow{v} X \xdashrightarrow{w}$ be an extriangle realizing $w$ in $(\c,\e)$. Since $(\mathbb{P},\Omega)$ is an extriangulated functor, by definition, the induced sequence
    \[ \mathbb{P}(Y)\xrightarrow{\mathbb{P}(u)} \mathbb{P}(Z)\xrightarrow{\mathbb{P}(v)} \mathbb{P}(X) \xdashrightarrow{\Omega^{X,Y}(w)} \]
    is an extriangle in $(\c',\e')$. By \cite[Definition 2.9]{NP}, the equality \eqref{eq:equiv} guarantees the existence of a morphism $b\colon \mathbb{P}(Z) \to Z'$ such that the following diagram commutes:
    \[\xymatrix{
    \mathbb{P}(Y)\ar[r]^{\mathbb{P}(u)}\ar[d]_{a} & \mathbb{P}(Z)\ar[r]^{\mathbb{P}(v)}\ar[d]_{b} &\mathbb{P}(X)\ar[d]^{c}\\
    Y'\ar[r]^{u'} & Z'\ar[r]^{v'} & X'.
    }\]
    Finally, since $a$ and $c$ are isomorphisms, it follows from \cite[Corollary 3.6~(1)]{NP} that $b$ is also an isomorphism.
\end{remark}

To define presilting and silting subcategories, we require the notion of higher extensions. Although higher extensions can be defined for general extriangulated categories \cite[Section 3]{GNP1}, we only recall the construction introduced in \cite[Section~3]{HLN2} for extriangulated categories with enough projective objects, as the two specific categories we focus on in the sequel satisfy this property. This construction is equivalent to the general one in this setting; see \cite[Corollary~3.21]{GNP1}. An object $X$ in an extriangulated category $(\c,\e)$ is called \emph{projective} if $\e(X,Y)=0$ for any object $Y\in \c$. We say that $\c$ has \emph{enough projective objects} if for any object $X$ in $\c$, there exists an extriangle
\begin{equation}\label{eq:procov}
    X_1\to P\to X\dashrightarrow,
\end{equation}
with $P$ projective. For any $i>1$, the $i$-th extension group is defined recursively by
\begin{equation}\label{eq:higherext}
    \e^i(X,Y):=\e^{i-1}(X_1,Y).
\end{equation}
This group is well-defined up to canonical isomorphism, independent of the choice of the extriangle \eqref{eq:procov}. Dually, one defines the notions of injective objects and having enough injective objects. For any extriangle
\[Y\to I\to Y_1\dashrightarrow,\]
with $I$ injective in $\c$, we have
\begin{equation}\label{eq:extinj}
    \e^i(X,Y)\cong\e^{i-1}(X,Y_1),\ \text{for all $i>1$}.
\end{equation}
The following lemma will be used later.

\begin{lemma}\label{lem:siltiff}
    Let $(\mathbb{P},\Omega)\colon (\c,\e) \to (\c',\e')$ be an extriangulated functor such that $\Omega$ is a natural isomorphism. Assume that both $(\c,\e)$ and $(\c',\e')$ have enough projective objects, and that $\mathbb{P}$ sends projective objects to projective objects. Then for any objects $X$ and $Y$ in $\c$, there are isomorphisms
    \[\e^i(X,Y)\cong{\e'}^i(\mathbb{P}(X),\mathbb{P}(Y)), \ \text{for all $i>0$}.\]
\end{lemma}

\begin{proof}
    We proceed by induction on $i$. The base case $i=1$ follows from the fact that $\Omega$ is a natural isomorphism. Assume that the statement holds for $i=s$, where $s$ is a positive integer. Consider the case $i=s+1$. Since $\c$ has enough projective objects, there exists an extriangle as in \eqref{eq:procov}:
    \[X_1\to P\to X\dashrightarrow,\]
    where $P$ is projective. Then, by definition, we have $\e^{s+1}(X,Y)\cong\e^s(X_1,Y)$. 
    
    On the other hand, since $(\mathbb{P},\Omega)$ is an extriangulated functor, by definition, there is an induced extriangle in $\c'$:
    \[\mathbb{P}(X_1)\to\mathbb{P}(P)\to\mathbb{P}(X)\dashrightarrow.\]
    Since $\mathbb{P}(P)$ is projective in $\c'$, the definition of higher extensions yields ${\e'}^{s+1}(\mathbb{P}(X),\mathbb{P}(Y))\cong{\e'}^s(\mathbb{P}(X_1),\mathbb{P}(Y))$. By the induction hypothesis, we have $\e^s(X_1,Y)\cong {\e'}^s(\mathbb{P}(X_1),\mathbb{P}(Y))$. Combining these isomorphisms yields the required isomorphism for $i=s+1$, completing the induction.
\end{proof}

We are now in a position to define presilting subcategories. Let $(\c,\e)$ be an extriangulated category with enough projective objects.

\begin{definition}
    A subcategory $\X$ of $\c$ is called \emph{presilting} if $\X=\add\X$ and
    \[\e^i(\X,\X)=0,\ \text{for all $i>0$.}\]
\end{definition}

A subcategory $\X$ of $\c$ is called \emph{thick} provided that it is closed under direct summands, and that for any extriangle
\[Y\to Z\to X\dashrightarrow\]
in $\c$, if two of the objects $X$, $Y$, and $Z$ belong to $\X$, then so does the third. For any subcategory $\X$ of $\c$, we denote by $\mathrm{thick}_{\c}\X$ the smallest thick subcategory of $\c$ containing $\X$.

\begin{definition}[{\cite[Definition~5.1]{AT22}}]
    A presilting subcategory $\X$ is called \emph{silting} if
    \[\c=\mathrm{thick}_{\c}\X.\]
    An object $X$ in $\c$ is called \emph{silting} if $\add X$ is a silting subcategory.
\end{definition}

We recall from \cite{AT25} the notion of mutations of silting subcategories of an extriangulated category.

\begin{definition}[{\cite[Definitions~4.6 and~4.9]{AT25}}]
    Let $\X$ be a silting subcategory of $\c$. A subcategory $\Y$ of $\X$ is called \emph{good covariantly finite} if for each object $X$ in $\X$, there exists an extriangle
    \[X\xrightarrow{f} Y\to Z_X\dashrightarrow,\]
    where $f$ is a left $\Y$-approximation of $X$, that is, $Y\in\Y$ and any morphism from $X$ to an object in $\Y$ factors through $f$. 
    
    Let $\Y$ be a good covariantly finite subcategory of $\X$. The subcategory
    \[\mu^L(\X;\Y):=\add(\Y\cup\{Z_X\mid X\in\X\})\]
    of $\c$ is called the \emph{left mutation} of $\X$ with respect to $\Y$. The right mutation $\mu^R(\X;\Y)$ of $\X$ with respect to a good contravariantly finite subcategory $\Y$ of $\X$ is defined dually. 
\end{definition}

The left (resp., right) mutation $\mu^L(\X;\Y)$ (resp., $\mu^R(\X;\Y)$) of a silting subcategory $\X$ with respect to a good covariantly (resp., contravariantly) finite subcategory of $\X$ is again a silting subcategory \cite[Theorem~4.12]{AT25}.

Let $(\c,\e)$ be an extriangulated category with enough projective objects. Let $m$ be a non-positive integer. We call an object $X\in\c$ has \emph{projective dimension at most $m$} if there exist $m+1$ extriangles
\[X_{i+1}\to P_i\to X_i\dashrightarrow,\ 0\leq i\leq m,\]
where each $P_i$ is projective and $X_{m+1}=0$. We call $\c$ has \emph{global dimension at most $m$} if each object in $\c$ has projective dimension at most $m$. We call $\c$ has \emph{dominant dimension at least $m$} if for any projective object $P$ in $\c$, there exist $m$ extriangles
\[Y_i\to I_i\to Y_{i+1}\dashrightarrow,\ 0\leq i\leq m-1,\]
where $Y_0=P$, each $I_i$ is projective-injective, and $Y_m$ is injective.


\begin{definition}[{\cite[Definition 3.3]{Chen2}, \cite[Definition 3.7]{GNP2} and \cite[Definition 7.3]{PZ}}]
    Let $d$ be a non-negative integer. An extriangulated category $(\c,\e)$ with enough projective objects is \emph{$d$-Auslander} if $(\c,\e)$ has global dimension at most $d+1$ and has dominant dimension at least $d+1$. A $d$-Auslander extriangulated category is \emph{reduced}, if the only projective-injective objects are $0$. 
\end{definition}

\subsection{Two extriangulated categories from rigid subcategories}
	
Let $\t$ be a triangulated category with suspension functor $\Sigma$. Fix a positive integer $n$ and an $n$-rigid subcategory $\m$ of $\t$; that is, 
\[\Hom_{\t}(\m,\Sigma^{i}\m)=0, \ \text{for all $1\leq i\leq n$.}\]
We assume that $\m$ is closed under direct summands.

For the remainder of this section, we consider two extriangulated categories arising from $\m$: the $(n+1)$-term subcategory $\pr{\t}{n+1}{\m}$ of $\t$ generated by $\m$, and the subcategory $K^{[-n,0]}(\m)$ of the bounded homotopy category $K^b(\m)$ consisting of object isomorphic to complexes concentrated in degrees $-n$ to $0$.

\begin{definition}\label{def:m+1term}
    Let $m$ be a non-negative integer. An object $X$ of $\t$ is said to be \emph{$(m+1)$-presented} by $\m$ if there exist $m+1$ triangles in $\t$ of the form
    \begin{equation}\label{eq:tri for prt}
	X_{i+1} \xrightarrow{\alpha^X_i} M^X_{i} \xrightarrow{\beta^X_i} X_{i} \xrightarrow{\gamma^X_i} \Sigma X_{i+1}, \ 0\leq i\leq m, 
    \end{equation}
    where $X=X_0$, $X_{m+1}=0$, and $M^X_i\in\m$ for all $0\leq i\leq m$. We denote by $\pr{\t}{m+1}{\m}$ the subcategory of $\t$ consisting of all objects $(m+1)$-presented by $\m$, which is usually called the \emph{$(m+1)$-term subcategory of $\t$ generated by $\m$}. In particular, $\pr{\t}{1}{\m}=\m$.
\end{definition}
	
By definition, we have
\[\pr{\t}{m+1}{\m}=\m*\Sigma\m*\cdots*\Sigma^m\m.\]
This yields the following ascending chain of subcategories:
\[\pr{\t}{1}{\m} \subseteq \pr{\t}{2}{\m} \subseteq \cdots \subseteq \pr{\t}{n}{\m} \subseteq \pr{\t}{n+1}{\m}.\]

\begin{remark}\label{rmk:presen}
    Note that since $\m$ is $n$-rigid, we have
    \begin{equation}\label{eq:Mtopr}
        \Hom_{\t}(\m,\Sigma\pr{\t}{n}{\m})=0.
    \end{equation}
    Furthermore, in the triangles \eqref{eq:tri for prt}, for any $0\leq i\leq n-1$, we have
    \[X_{i+1}\in \pr{\t}{n-i}{\m} \ \text{and} \ \Hom_{\t}(\m, \Sigma X_{i+1})=0.\]
\end{remark}

Before equipping $\pr{\t}{n+1}{\m}$ with an extriangulated structure, we need to understand the behavior of extensions within the smaller subcategories $\pr{\t}{m}{\m}$, where $m\leq n$. The following lemma shows that they are closed under extensions in $\t$.

\begin{lemma}\label{lem:extforn} 
    For each $1\leq m\leq n$, the subcategory $\pr{\t}{m}{\m}$ of $\t$ is closed under extensions.
\end{lemma}

\begin{proof}
    We prove the lemma by induction on $m$. Since $\m$ is $n$-rigid, the base case $m=1$ holds. Assume that the assertion holds for $m=s$ for some $s<n$, and consider the case $m=s+1$. Let 
    \[Y\rightarrow Z\rightarrow X\xrightarrow{w}\Sigma Y\] 
    be a triangle with $X,Y\in\pr{\t}{s+1}{\m}$. By \eqref{eq:Mtopr}, we have $\Hom_{\t}(M_0^X,\Sigma Y)=0$, which implies $w\circ \beta^X_0=0$. Thus, the octahedral axiom yields the left-hand commutative diagram in \eqref{eq:twodiagrams}. Next, applying the octahedral axiom to the composition $(0\ \gamma^Y_0)\circ f$, we obtain the right-hand commutative diagram in \eqref{eq:twodiagrams}:
    \begin{equation}\label{eq:twodiagrams}
		\begin{tikzcd}[sep=1.25em]
		& X_{1} \arrow[d, dashed, "f"'] \arrow[r, equal] & X_{1} \arrow[d, "\alpha_0^X"]                    &                                         &                                                        & \Sigma^{-1}Z \arrow[d, dashed] \arrow[r, equal] & \Sigma^{-1}Z \arrow[d]                &                                             \\
		Y \arrow[r, dashed] \arrow[d, equal] & M_{0}^{X}\oplus Y \arrow[d, dashed] \arrow[r, dashed]  & M_{0}^{X} \arrow[d, "\beta_0^X"] \arrow[r, "0"] & \Sigma Y \arrow[d, equal] & Y_{1} \arrow[d, equal] \arrow[r, dashed] & {Z_{1}} \arrow[d, dashed] \arrow[r, dashed]          & X_{1} \arrow[d, "f"'] \arrow[r, "(0\ \gamma^Y_0)\circ f"]             & \Sigma Y_{1} \arrow[d, equal] \\
		Y \arrow[r]                                        & Z \arrow[r] \arrow[d, dashed]                          & X \arrow[r, "w"] \arrow[d, "\gamma_0^X"]              & \Sigma Y                                & Y_{1} \arrow[r]                                        & M_{0}^{X}\oplus M_{0}^{Y} \arrow[r] \arrow[d, dashed]         & M_{0}^{X}\oplus Y \arrow[r, "(0\ \gamma^Y_0)"'] \arrow[d] & \Sigma Y_{1}                                \\
		& \Sigma X_{1} \arrow[r, equal]            & \Sigma X_{1},                       &                                         &                                                        & Z \arrow[r, equal]                              & Z.                                     &                                            
		\end{tikzcd}
    \end{equation}
	Since both $X_1$ and $Y_1$ belong to $\pr{\t}{s}{\m}$, the induction hypothesis implies that $Z_1$ also belongs to $\pr{\t}{s}{\m}$. Consequently, the triangle in the second column of the right-hand diagram shows that $Z\in\pr{\t}{s+1}{\m}$. This completes the induction.
\end{proof}

We define a subfunctor $\e_{\m}$ of the bifunctor $\Hom_{\t}(-,\Sigma-)$ as follows: for any $X,Y\in\t$,
\[\e_{\m}(X,Y):=[\Sigma\pr{\t}{n}{\m}](X,\Sigma Y).\]
For any $w\in\e_\m(X,Y)$, we define
\[\mathfrak{s}_\m(w)=[Y\xrightarrow{u}Z\xrightarrow{v}X],\]
where $Y\xrightarrow{u}Z\xrightarrow{v}X\xrightarrow{w}\Sigma Y$ is a triangle in $\t$.

\begin{lemma}\label{esubfunctor}
    The triple $(\t,\e_{\m},\mathfrak{s}_\m)$ is an extriangulated category.
\end{lemma}

\begin{proof}
    Consider the first triangle in \eqref{eq:tri for prt}:
	\[X_1\xrightarrow{\alpha^X_0} M^X_0\xrightarrow{\beta^X_0} X\xrightarrow{\gamma^X_0} \Sigma X_1,\]
	where $X_1\in \pr{\t}{n}{\m}$ and $M^X_0\in\m$. Since $\m$ is $n$-rigid, we have $\Hom_\t(\m,\Sigma\pr{\t}{n}{\m})=0$ by \eqref{eq:Mtopr}. It then follows from this triangle that any morphism from an object in $\m$ to $X$ factors through $\beta^X_0$, and any morphism in $[\Sigma\pr{\t}{n}{\m}](X,\Sigma Y)$ factors through $\gamma^X_0$. This implies that
    \[
    \e_{\m}(X,Y) = \{ f\in\Hom_{\t}(X,\Sigma Y) \mid f\circ g=0\text{ for all $g\colon M\to X$ with $M\in\m$}\}.\]
    Therefore, by \cite[Proposition 3.19]{HLN1}, the assertion follows.
\end{proof}
	
Unlike the subcategories $\pr{\t}{m}{\m}$ for $m \leq n$, which are closed under extensions in $\t$ by Lemma~\ref{lem:extforn}, $\pr{\t}{n+1}{\m}$ does not generally share this property in the triangulated category $\t$. However, as shown in the following lemma, it becomes extension-closed when viewed as a subcategory of the extriangulated category $(\t, \e_\m)$.
	
\begin{lemma}\label{extclose}
	Let $X,Y\in\pr{\t}{n+1}{\m}$. If $Y\rightarrow Z\rightarrow X\xrightarrow{w}\Sigma Y$ is a triangle in $\t$ such that $w\in\e_{\m}(X,Y)$, then $Z\in\pr{\t}{n+1}{\m}$.
\end{lemma}
	
\begin{proof} 
    By \eqref{eq:Mtopr}, we have $\Hom_{\t}(M_{0}^{X},\Sigma\pr{\t}{n}{\m})=0$. Since $w$ factors through an object in $\Sigma\pr{\t}{n}{\m}$ by definition, we obtain $w\circ \beta_0^X=0$. Thus, the octahedral axiom yields the left-hand commutative diagram in \eqref{eq:twodiagrams}. Next, applying the octahedral axiom again to the composition $(0\ \gamma^Y_0)\circ f$, we obtain the right-hand commutative diagram. Since both $X_1$ and $Y_1$ belong to $\pr{\t}{n}{\m}$, Lemma~\ref{lem:extforn} implies that $Z_1$ also belongs to $\pr{\t}{n}{\m}$. Consequently, the triangle in the second column of the right-hand diagram in \eqref{eq:twodiagrams} shows that $Z\in\pr{\t}{n+1}{\m}$. This completes the proof.
\end{proof}
	
Having established this closure property, we arrive at the following proposition.
	
\begin{prop}\label{prop:extri}
	The triple $(\pr{\t}{n+1}{\m},\e_{\m},\mathfrak{s}_\m)$ is an extriangulated category. 
\end{prop}

\begin{proof}
    By Lemma~\ref{esubfunctor}, the ambient category $(\t, \e_\m, \mathfrak{s}_\m)$ is extriangulated. Furthermore, Lemma~\ref{extclose} shows that $\pr{\t}{n+1}{\m}$ is closed under extensions in $(\t, \e_\m)$. Therefore, it naturally inherits the extriangulated structure (see, for instance, \cite[Remark 2.18]{NP}).
\end{proof}

Henceforth, we view $\pr{\t}{n+1}{\m}$ as an extriangulated category endowed with the structure above. In this setup, its extriangles are exactly the triangles $Y\rightarrow Z\rightarrow X\xrightarrow{w}\Sigma Y$ in $\t$ satisfying $X, Y, Z\in\pr{\t}{n+1}{\m}$ and $w\in[\Sigma\pr{\t}{n}{\m}](X,\Sigma Y)$. Thus, by a slight abuse of terminology, we will simply call such triangles extriangles in $\pr{\t}{n+1}{\m}$.

Since $\m$ is $n$-rigid, it is straightforward to see that $\pr{\t}{n+1}{\m}$ has enough projective and injective objects, which are given by the subcategories $\m$ and $\Sigma^n\m$, respectively. By construction, each object in $\pr{\t}{n+1}{\m}$ has projective dimension at most $n$. For any $M\in\m$, there exist extriangles:
\[
\Sigma^{i}M\to 0\to \Sigma^{i+1}M\xdashrightarrow{\id_{\Sigma^{i+1}M}},\ 0\le i\le n-1,
\]
which show that the dominant dimension of $(\pr{\t}{n+1}{\m},\e_\m)$ is at least $n$. Moreover, since $\Hom(\m,\Sigma^n\m)=0$, we have $\m\cap\Sigma^n\m=\{0\}$. Thus, we obtain the following result.

\begin{prop}\label{prop:reduced}
    The extriangulated category $(\pr{\t}{n+1}{\m},\e_{\m})$ is reduced $(n-1)$-Auslander.
\end{prop}

The following lemma will be used later.

\begin{lemma}\label{lem:shorter}
    Let $X$ and $Y$ be objects in $\pr{\t}{n+1}{\m}$.
    \begin{enumerate}
        \item If $X\in\pr{\t}{i}{\m}$ for some $1\leq i\leq n$, then
        \[
            \e_\m^j(X,Y)=0,\ \text{for all $j\geq i$}.\]
        \item If $Y\in\pr{\t}{i}{\m}$ for some $1\leq i\leq n$, then
        \[\e_\m^j(X,Y)\cong\begin{cases}
            \Hom_\t(X,\Sigma^j Y) & \text{if $1\leq j\leq n+1-i$,}\\
            [\Sigma\pr{\t}{n}{\m}](X,\Sigma^j Y) & \text{if $j=n+2-i$.}
        \end{cases}\]
    \end{enumerate}
\end{lemma}

\begin{proof}
    For (1), since $X\in\pr{\t}{i}{\m}$, there exist triangles in $\t$:
    \[X_{s}\to M_{s-1}\to X_{s-1}\to\Sigma X_s,\ \text{$s\geq 1$},\]
    where $X_s\in\m$ for all $s\geq i-1$ and $X_{0}=X$. Hence we have
    \[\e_\m^j(X,Y)=\e_\m^{j-1}(X_1,Y)=\cdots=\e_m(X_{j-1},Y)=0,\ \text{for all $j\geq i$.}\]

    For (2), since $Y\in\pr{\t}{i}{\m}$, there exist triangles in $\t$:
    \[\Sigma^{s-1}Y\to 0\to \Sigma^s Y\to \Sigma^s Y,\ \text{for all $1\leq s\leq n+1-i$.}\]
    Hence, by \eqref{eq:extinj}, for any $1\leq j\leq n+2-i$, we have
    \[\e_\m^j(X,Y)\cong\e_\m^{j-1}(X,\Sigma Y)\cong\cdots\cong\e_\m(X,\Sigma^{j-1}Y),\]
    where the last one equals $\Hom_\t(X,\Sigma^j Y)$ when $1\leq j\leq n+1-i$ since $\Sigma^j Y\in\Sigma\pr{\t}{n}{\m}$, and equals $[\Sigma\pr{\t}{n}{\m}](X,\Sigma^j Y)$ when $j=n+2-i$ by definition.
\end{proof}

\begin{exam}\label{exm:1}
    Let $Q$ be the quiver $2\to 1$, and let $\t=D^{b}(\k Q)/(\tau^{-1}\Sigma^{2})$ be the $2$-cluster category of $Q$, where $\tau$ denotes the Auslander-Reiten translation. The Auslander-Reiten quiver of $\t$ is given by
    
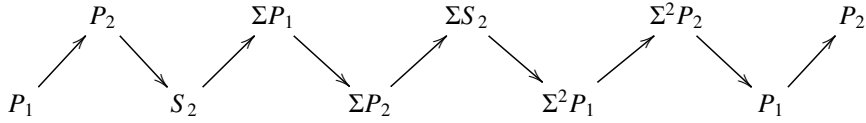
\[\xymatrix@R=.6cm@C=.5cm{
    &P_2\ar[rd]&&\Sigma P_1\ar[rd]&&\Sigma S_2\ar[rd]&&\Sigma^2 P_2\ar[rd]&&P_2\\
    P_1\ar[ru]&&S_2\ar[ru]&&\Sigma P_2\ar[ru]&&\Sigma^2 P_1\ar[ru]&&P_1\ar[ru]
    }\]
    Here, $P_{i}$ (resp., $S_i$) denotes the indecomposable projective (resp., simple) module corresponding to vertex $i$. Let $\m=\add (P_1\oplus\Sigma P_2)$. Then $\m$ is a $2$-cluster subcategory of $\t$. In particular, it is $2$-rigid. Consider the following triangles:
    \[\Sigma P_1\to \Sigma P_2\to \Sigma S_2\to \Sigma^2 P_1,\]
    \[P_1\to 0\to \Sigma P_1\to\Sigma P_1.\]
    By definition, we have 
    \[\e_\m^2(\Sigma S_2,P_2)=\e_\m^1(\Sigma P_1,P_2)=[\Sigma\pr{\t}{2}{\m}](\Sigma P_1,\Sigma P_2)=\Hom_\t(\Sigma P_1,\Sigma P_2)\neq 0.\]
    Note that $\Hom_\t(\Sigma S_2,\Sigma^2 P_2)=0$. This implies that $\e_\m^2(\Sigma S_2,P_2)$ is not isomorphic to a subgroup of $\Hom_\t(\Sigma S_2,\Sigma^2 P_2)$.
\end{exam}
	
We now turn our attention to another extriangulated category associated with $\m$. Let $K^b(\m)$ be the homotopy category of bounded complexes of objects in $\m$. Let $K^{[-n,0]}(\m)$ be the subcategory of $K^{b}(\m)$ whose objects are isomorphic to complexes concentrated in degrees $-n$ to $0$. Since the subcategory $K^{[-n,0]}(\m)$ is closed under extensions in $K^b(\m)$, it naturally inherits an extriangulated structure. Its extension groups are given by
\[\e_{K^{[-n,0]}(\m)}(X,Y):=\Hom_{K^{b}(\m)}(X,\Sigma Y),\ \text{for $X,Y\in K^{[-n,0]}(\m)$},\]
and its extriangles are of the form
\[Y\xrightarrow{u} Z\xrightarrow{v} X\xdashrightarrow{w},\]
where $Y\xrightarrow{u} Z\xrightarrow{v} X\xrightarrow{w}\Sigma Y$ is a triangle in $K^b(\m)$ with $X,Y,Z\in K^{[-n,0]}(\m)$. It is straightforward to see that $K^{[-n,0]}(\m)$ has enough projective and injective objects, which are given by the subcategories $\m$ and $\Sigma^n\m$, respectively. Furthermore, we have the isomorphisms
\[\e_{K^{[-n,0]}(\m)}^i(X,Y)\cong\Hom_{K^{b}(\m)}(X,\Sigma^i Y), \ \text{for all $i>0$}.\]
Note that $K^{[-n,0]}(\m)=\pr{K^b(\m)}{n+1}{\m}$.

\begin{remark}\label{rmk:termsilt}
    The silting subcategories in the extriangulated category $K^{[-n,0]}(\m)$ are precisely the silting subcategories in the triangulated category $K^b(\m)$ that lie in $K^{[-n,0]}(\m)$, namely, the $(n+1)$-term silting subcategories of $K^b(\m)$.
\end{remark}

\begin{proof}
    By definition, any silting subcategory of $K^{[-n,0]}(\m)$ is silting in $K^b(\m)$. Conversely, let $\X$ be a subcategory of $K^{[-n,0]}(\m)$ that is silting in $K^b(\m)$. Then we have $\m\subseteq\Sigma^{-n}\X*\Sigma^{1-n}\X*\cdots*\Sigma^{-1}\X*\X$. In particular, for any object $M\in\m$, there exist the following triangles in $K^b(\m)$:
    \[M_{i-1}\to X_i\to M_i\to \Sigma M_{i-1},\ 1\leq i\leq n,\]
    where $M_0=M$, $M_n\in\X$, and $X_i\in\X$ for all $i$. From these triangles, we obtain
    \[M_i\in\X * \add \Sigma M_{i-1}\subseteq\X * \Sigma\X * \add\Sigma^2 M_{i-2}\subseteq\cdots\subseteq \X * \Sigma\X *\cdots*\Sigma^{i-1}\X* \Sigma^i\m\subseteq K^{[-n-i+1,0]}(\m).\]
    On the other hand, we have
    \[M_i\in\add\Sigma^{-1} M_{i+1}*\X\subseteq\add\Sigma^{-2} M_{i+2}*\Sigma^{-1}\X*\X\subseteq\cdots\subseteq \Sigma^{i-n}\X*\cdots*\Sigma^{-1}\X*\X\subseteq K^{[-n,n-i]}(\m).\]
    Hence, $M_i\in K^{[-n,0]}(\m)$ for each $i$. Consequently, the above triangles are all extriangles in $K^{[-n,0]}(\m)$, which implies that $\X$ is silting in $K^{[-n,0]}(\m)$.
\end{proof}

\begin{lemma}\label{lem:silequiv}
    Let $\X$ be a presilting subcategory of $K^{[-n,0]}(\m)$ that is covariantly finite in $K^b(\m)$. Then $\X$ is silting if and only if every object $Y\in K^{[-n,0]}(\m)$ satisfying
    \[\Hom_{K^b(\m)}(\X,\Sigma^i Y)=0=\Hom_{K^b(\m)}(Y,\Sigma^i\X)\ \text{for all $i>0$}\]
    belongs to $\X$.
\end{lemma}

\begin{proof}
    If $\X$ is silting in $K^{[-n,0]}(\m)$, then it is silting in the triangulated category $K^b(\m)$. Let $Y$ be an object in $K^{[-n,0]}(\m)$ satisfying $\Hom_{K^b(\m)}(\X,\Sigma ^i Y)=0=\Hom_{K^b(\m)}(Y,\Sigma^i\X)$ for all $i>0$. By \cite[Proposition~2.17]{AI}, there exists a positive integer $l$ such that $Y\in\operatorname{smd}(\Sigma^{-l}\X*\Sigma^{-l+1}\X*\cdots*\Sigma^{l-1}\X*\Sigma^l\X)$, where $\operatorname{smd}(\Y)$ denotes the full subcategory consisting of all direct summands of objects in $\Y$. Since $\Hom_{K^b(\m)}(\X,\Sigma^i Y)=0=\Hom_{K^b(\m)}(Y,\Sigma^i \X)$ for all $i>0$, by \cite[Lemma~2.16]{AI} we have $Y\in\operatorname{smd}(\X)=\X$. This proves the ``only if" part.

    For the ``if" part: For any object $M\in\m$, since $\X$ is covariantly finite, there exists a left $\X$-approximation $f\colon M\to X_1$ of $M$. We extend this morphism to a triangle in $K^b(\m)$:
    \[M\xrightarrow{f} X_1\to M_1\to \Sigma M.\]
    Repeating this process, we obtain triangles in $K^b(\m)$:
    \[M_{i-1}\xrightarrow{f_i} X_i\to M_i\to\Sigma M_{i-1},\ 1\leq i\leq n,\]
    where each $f_i$ is a left $\X$-approximation of $M_{i-1}$, and $M_0=M$. Applying the functors $\Hom_{K^b(\m)}(\X,-)$ and $\Hom_{K^b(\m)}(-,\X)$ to these triangles yields $\Hom_{K^b(\m)}(\X,\Sigma^i M_n)=0=\Hom_{K^b(\m)}(M_n,\Sigma^i\X)$ for all $i>0$. Hence, $M_n\in\X$. Thus, $M\in\mathrm{thick}_{K^b(\m)}(\X)$, which implies that $\X$ is silting in $K^b(\m)$. By Remark~\ref{rmk:termsilt}, $\X$ is therefore also silting in $K^{[-n,0]}(\m)$.
\end{proof}

\section{Construction of the functor \texorpdfstring{$\Pn$}{P}}\label{section3}

Let $\t$ be an algebraic triangulated category, meaning it is the stable category of a Frobenius exact category $\f$. Fix a positive integer $n$ and an $n$-rigid subcategory $\m$ of $\t$.

In this section, we construct an additive functor
\[\Pn\colon\pr{\t}{n+1}{\m} \to K^{[-n,0]}(\m)\]
by first defining a functor on a corresponding subcategory of $\f$ and then passing to the stable category $\t$.

Let $\p$ be the subcategory of $\f$ consisting of all projective-injective objects, and let $\pi\colon\f\to \t$ be the canonical projection functor. We set $\M=\pi^{-1}(\m)$. Note that $\p\subseteq\M$. We denote by $K^b(\M)$ the homotopy category of bounded complexes of objects in $\M$, and by $K^{[-n,0]}(\M)$ the subcategory of $K^b(\M)$ consisting of complexes concentrated in degrees $-n$ to $0$.

\begin{definition}
    Let $m$ be a non-negative integer. For any object $X$ in $\f$, an \emph{$\M$-presentation of $X$ of length $m+1$} is a sequence $\eta^X=(\eta^X_i)_{0\leq i\leq m}$ of conflations in $\f$ of the form
    \begin{equation}\label{presinf}
    \eta^X_i\colon X_{i+1}\xrightarrow{\talpha_{i}^X}M_{i}^X\xrightarrow{\tbeta^X_{i}}X_{i}, \ 0\leq i\leq m,
    \end{equation}
    where $X_0=X$ and $M^X_i\in\M$ for all $0\leq i\leq m$.

    An object $X$ in $\f$ is said to be \emph{$(m+1)$-presented by $\M$} if it admits an $\M$-presentation \eqref{presinf} of length $m+1$ such that $X_{m+1}=0$ (in which case $\talpha^X_m=0$ and $\tbeta^X_{m}$ may be taken to be the identity). We denote by $\pr{\f}{m+1}{\M}$ the subcategory of $\f$ consisting of all objects $(m+1)$-presented by $\M$. In particular, $\pr{\f}{1}{\M}=\M$.
\end{definition}

This yields the chain of subcategories:
\[\pr{\f}{1}{\M}\subseteq \pr{\f}{2}{\M} \subseteq \cdots \subseteq \pr{\f}{n}{\M} \subseteq \pr{\f}{n+1}{\M}.\]
By construction, each object $X_i$ in \eqref{presinf} belongs to $\pr{\f}{m+1-i}{\M}$. When $m=n$, the $n$-rigidity of $\M$ implies that 
\begin{equation}\label{extMXi}
    \Ext^1_\f(\M,X_i)=0, \ \text{for all }1\leq i\leq n.
\end{equation}

\begin{definition}
    Let $X$ and $Y$ be objects in $\pr{\f}{n+1}{\M}$, and let $\eta$ and $\varepsilon$ be $\M$-presentations of length $n+1$ of $X$ and $Y$, respectively. For any morphism $\tf\colon X\to Y$, a \emph{presentation of $\tf$ from $\eta$ to $\varepsilon$} is a sequence of morphisms between conflations in $\f$:
    \begin{equation}\label{presinf2}
        \begin{tikzcd}
		\eta_i\colon & X_{i+1} \arrow[d, "\tF_{i+1}"] \arrow[r, "\talpha_{i}^{X}"] & M_{i}^{X} \arrow[d, "\tf_{i}"] \arrow[r, "\tbeta_{i}^{X}"] & X_{i} \arrow[d, "\tF_{i}"] \\
		\varepsilon_i\colon & Y_{i+1} \arrow[r, "\talpha_{i}^{Y}"] & M_{i}^{Y} \arrow[r, "\tbeta_{i}^{Y}"] & {Y_{i}}       
		\end{tikzcd}, \ 0\leq i\leq n,
    \end{equation}
    where $\tF_0=\tf$. Note that the morphisms $\tF_{i+1}$ are uniquely determined by the morphisms $\tf_i$ via the universal property of kernels. Therefore, for convenience, we shall often identify the presentation with the sequence of morphisms $(\tf_i)_{0 \leq i \leq n}$, and denote it by $\tf_\bullet\colon \eta\to \varepsilon$.
\end{definition}

\begin{remark}\label{rmk:comp}
    Let $\tf\colon X\to Y$ and $\tg\colon Y\to Z$ be morphisms in $\pr{\f}{n+1}{\M}$, and let $\tf_\bullet=(\tf_i)_{0 \leq i \leq n}\colon\eta\to\varepsilon$ and $\tg_\bullet=(\tg_i)_{0 \leq i \leq n}\colon\varepsilon\to \omega$ be presentations of $\tf$ and $\tg$, respectively. Then $\tg_\bullet\circ\tf_\bullet:=(\tg_i\circ\tf_i)_{0 \leq i \leq n}$ is a presentation of $\tg\circ\tf$ from $\eta$ to $\omega$.
\end{remark}

The following lemma establishes the existence of a presentation for any morphism, with respect to any chosen $\M$-presentations of its domain and codomain.

\begin{lemma}\label{lem:pre-mor}
    Let $X$ and $Y$ be objects in $\pr{\f}{n+1}{\M}$, and let $\eta$ and $\varepsilon$ be $\M$-presentations of length $n+1$ of $X$ and $Y$, respectively. Then any morphism $\tf\colon X\to Y$ admits a presentation from $\eta$ to $\varepsilon$.
\end{lemma}

\begin{proof}
    Set $\tF_0=\tf$. We construct $\tf_i$ and $\tF_{i+1}$ for $0\leq i\leq n$ by induction as follows. Assume that $\tF_i$ has been constructed for some $i\geq 0$. Since $\Ext^1_\f(M^X_i,Y_{i+1})=0$ by \eqref{extMXi}, the morphism $\tF_i\circ\tbeta^X_i$ factors through $\tbeta^Y_i$; that is, there exists a morphism $\tf_i\colon M^X_i\to M^Y_i$ such that the right square of the diagram in \eqref{presinf2} commutes. Consequently, the universal property of kernels ensures the existence of a morphism $\tF_{i+1}\colon X_{i+1}\to Y_{i+1}$ that makes the left square commute. By induction, we obtain a presentation of $\tf$ from $\eta$ to $\varepsilon$.
\end{proof}

An $\M$-presentation $\eta$ of $X\in\pr{\f}{n+1}{\M}$ of length $n+1$ as in \eqref{presinf} induces a complex of the form
\begin{equation}\label{eq:cplx}
    \cplx{X}{\eta}\colon M^X_{n}\xrightarrow{\talpha^X_{n-1}\circ\tbeta^X_n}M^X_{n-1}\xrightarrow{\talpha^X_{n-2}\circ\tbeta^X_{n-1}}\cdots\xrightarrow{\talpha^X_{0}\circ\tbeta^X_1}M^X_{0}.
\end{equation}
Correspondingly, a presentation of a morphism $\tf\colon X\to Y$ as in \eqref{presinf2} gives rise to a chain map between the respective complexes:
\begin{equation}\label{eq:mor}
    \xymatrix@C=1.5cm{
    M_{n}^{X}\ar[r]^{\talpha_{n-1}^{X}\circ\tbeta^X_n}\ar[d]^{\tf_n} & M_{n-1}^{X}\ar[d]^{\tf_{n-1}}\ar[r]^{\talpha_{n-2}^{X}\circ\tbeta_{n-1}^{X}} & \cdots\ar[r]^{\talpha_{0}^{X}\circ\tbeta_{1}^{X}} & M_{0}^{X}\ar[d]^{\tf_0}\\
	M_{n}^{Y}\ar[r]_{\talpha_{n-1}^{Y}\circ\tbeta^Y_n} & M_{n-1}^{Y}\ar[r]_{\talpha_{n-2}^{Y}\circ\tbeta_{n-1}^{Y}} & \cdots\ar[r]_{\talpha_{0}^{Y}\circ\tbeta_{1}^{Y}} & M_{0}^{Y}.
}
\end{equation}
The homotopy class of this chain map is denoted by $[\tf_\bullet]:=[(\tf_{0},\tf_{1},\dots,\tf_{n})]$.


\begin{lemma}\label{homotopy}
    Let $\tf\colon X\rightarrow Y$ be a morphism in $\pr{\f}{n+1}{\M}$ with a presentation $\tf_\bullet$ of the form \eqref{presinf2}. If $\tf$ factors through an object $P\in\p$, say via $\tf=\tv\circ\tu$ with $\tu\colon X\rightarrow P$ and $\tv\colon P\rightarrow Y$, then the following chain map of complexes is null-homotopic:
	\[\xymatrix@C=1.5cm{
		M_{n}^{X}\ar[r]^{\talpha_{n-1}^{X}\circ\tbeta^X_n}\ar[d]^{\tf_n} & M_{n-1}^{X}\ar[d]^{\tf_{n-1}}\ar[r]^{\talpha_{n-2}^{X}\circ\tbeta_{n-1}^{X}} & \cdots\ar[r]^{\talpha_{0}^{X}\circ\tbeta_{1}^{X}} & M_{0}^{X}\ar[d]^{\tf_0}\ar[r]^{\tu\circ\tbeta_{0}^{X}} & P\\
		M_{n}^{Y}\ar[r]_{\talpha_{n-1}^{Y}\circ\tbeta^Y_n} & M_{n-1}^{Y}\ar[r]_{\talpha_{n-2}^{Y}\circ\tbeta_{n-1}^{Y}} & \cdots\ar[r]_{\talpha_{0}^{Y}\circ\tbeta_{1}^{Y}} & M_{0}^{Y}.
		}\]
	Consequently, the following statements hold.
    \begin{enumerate}
        \item[(1)] If $\tf=0$ in $\f$, then $[\tf_\bullet]=0$ in $K^{b}(\M)$.
        \item[(2)] If $\pi(\tf)=0$ in $\t$, then $\pi([\tf_\bullet])=0$ in $K^{b}(\m)$.
    \end{enumerate}
\end{lemma}
    
\begin{proof}
    For convenience, we set $M^X_{-1}=P$ and $\talpha_{-1}^X=\tu\colon X_0\to M^X_{-1}$. Let $\tg_{-1}=\tv\colon M^X_{-1}\to Y_0$. We construct morphisms $\th_{i-1}\colon M_{i-1}^X\to M_i^Y$ and $\tg_i\colon M_i^X\to Y_{i+1}$, $0\leq i\leq n-1$, which satisfy
    \begin{equation}\label{eq:cons g h}
        \begin{cases}
            \tg_{i-1}  =  \tbeta_{i}^{Y}\circ \th_{i-1},\\
			\tf_{i}  =  \th_{i-1}\circ\talpha_{i-1}^{X}\circ\tbeta_{i}^{X}+\talpha_{i}^{Y}\circ \tg_{i},\\
			\tF_{i+1}  =  \tg_{i}\circ\talpha_{i}^{X}.
        \end{cases}
    \end{equation}
    See the following diagram.
    \[\begin{tikzcd}
		X_{i+1}\arrow[rr, "\talpha_{i}^{X}"]\arrow[dd, "\tF_{i+1}"] &  & M_{i}^{X}\arrow[dd, "\tf_{i}"]\arrow[rr, "\tbeta_{i}^{X}"]\arrow[lldd, "\tg_{i}", dotted] &  & X_{i}\arrow[dd, "\tF_{i}"]\arrow[rrd, "\talpha_{i-1}^{X}"] &  & \\
		& & & & & & M_{i-1}^{X}\arrow[lld, "\tg_{i-1}"]\arrow[lllld, "\th_{i-1}"' {pos=0.7}, dashed] \\
		Y_{i+1}\arrow[rr, "\talpha_{i}^{Y}"'] & & M_{i}^{Y}\arrow[rr, "\tbeta_{i}^{Y}"'] &  & Y_{i} &  &
	\end{tikzcd}\]
    Assume that $\th_{i-2}$ and $\tg_{i-1}$ satisfying the required conditions have been constructed. Since $\Ext^1_\f(M_{i-1}^{X},Y_{i+1})=0$, there exists a morphism $\th_{i-1}\colon M_{i-1}^{X} \rightarrow M_{i}^{Y}$ such that
	\[\tg_{i-1}=\tbeta_{i}^{Y}\circ \th_{i-1}.\]
    Since $\tbeta_{i}^{Y} \circ \th_{i-1} \circ \talpha_{i-1}^{X}= \tg_{i-1} \circ \talpha_{i-1}^{X} = \tF_i$, we have $\tbeta_{i}^{Y}\circ(\tf_{i}-\th_{i-1}\circ \talpha_{i-1}^{X}\circ\tbeta_{i}^{X})=\tbeta_{i}^{Y}\circ\tf_{i}-\tF_i\circ\tbeta_{i}^{X}=0$. By the universal property of kernels, there exists a morphism $\tg_i\colon M_{i}^{X}\rightarrow Y_{i+1}$ such that 
	\[\tf_i - \th_{i-1} \circ \talpha_{i-1}^X \circ \tbeta_{i}^{X} = \talpha_{i}^{Y} \circ \tg_i.\]
	Consequently, $\talpha_i^{Y}\circ(\tF_{i+1}-\tg_{i}\circ\talpha_{i}^{X})=\talpha_{i}^{Y}\circ\tF_{i+1}-\tf_{i}\circ\talpha_{i}^{X}=0.$ Since $\talpha_{i}^{Y}$ is an inflation, we conclude that 
	\[\tF_{i+1}=\tg_{i}\circ\talpha_{i}^{X}.\]
    This completes the construction of the morphisms. Finally, let $\th_{n-1}=\tg_{n-1}$. It follows from \eqref{eq:cons g h} that
    \[\tf_{i}=(\talpha_{i}^{Y}\circ\tbeta_{i+1}^{Y}) \circ \th_{i}+\th_{i-1} \circ (\talpha_{i-1}^{X}\circ\tbeta_{i}^{X}), \ 0\leq i\leq n-1.\]
	  and
    \[\tf_{n}=\th_{n-1}\circ\talpha_{n-1}^{X}.\]
	Therefore, $(\th_{n-1},\dots,\th_1,\th_0,\th_{-1})$ forms a null-homotopy of the shown chain map.

    For (1), since $\tf=0$, we may take $P=0$. In this case, the chain map shown above coincides with the chain map in \eqref{eq:mor}. Hence, $[\tf_\bullet]=0$ in $K^{b}(\M)$.

    For (2), since $\t$ is the stable category of $\f$, the condition $\pi(\tf)=0$ implies exactly that $\tf$ factors through some projective-injective object $P\in\p$. By the main assertion proved above, the corresponding augmented chain map ending in $P$ is null-homotopic in $K^b(\M)$. Applying the additive functor $\pi$ to this null-homotopy, the term $P$ vanishes since $\pi(P)=0$. As a result, the image of this augmented chain map under $\pi$ coincides exactly with $\pi(\tf_\bullet)$, showing that $\pi(\tf_\bullet)$ is null-homotopic in $K^b(\m)$. Thus, $\pi([\tf_\bullet])=0$.
\end{proof}

The following straightforward consequence of Lemma~\ref{homotopy}~(1) is essential for our subsequent development.

\begin{lemma}\label{lem:homotopy=}
    Let $X$ and $Y$ be objects in $\pr{\f}{n+1}{\M}$, and let $\eta$ and $\varepsilon$ be $\M$-presentations of length $n+1$ of $X$ and $Y$, respectively.  If $\tf_\bullet=(\tf_i)_{0 \leq i \leq n}$ and $\tf'_\bullet=(\tf'_i)_{0 \leq i \leq n}$ are two presentations of a morphism $\tf \colon X \to Y$ from $\eta$ to $\varepsilon$, then
    $[\tf_\bullet]=[\tf'_\bullet]$
    in $K^{b}(\M)$.
\end{lemma}

We now construct an assignment $\tPn$ from $\pr{\f}{n+1}{\M}$ to $K^{b}(\M)$.

\begin{construction}\label{cons:tP}
    For any $X\in\pr{\f}{n+1}{\M}$, we fix an $\M$-presentation $\eta^X$ of $X$ of the form \eqref{presinf} and define $\tPn(X)=\cplx{X}{\eta^X}$ (see \eqref{eq:cplx}). 
    
    For any morphism $\tf\colon X\to Y$ in $\pr{\f}{n+1}{\M}$, Lemma~\ref{lem:pre-mor} guarantees the existence of a presentation $\tf_\bullet$ of $\tf$ from $\eta^X$ to $\eta^Y$. We define $\tPn(\tf)=[\tf_\bullet]\colon \cplx{X}{\eta^X}\to \cplx{Y}{\eta^Y}$ (see \eqref{eq:mor}). By Lemma~\ref{lem:homotopy=}, this assignment is well-defined.
\end{construction}

Let $X$ be an object in $\pr{\f}{n+1}{\M}$, and let $\eta$ and $\varepsilon$ be two $\M$-presentations of $X$ of length $n+1$. By Lemma~\ref{lem:pre-mor}, there exists a presentation $\cmap{X}{\eta}{\varepsilon}$ of the identity $\id_X$ from $\eta$ to $\varepsilon$ and a presentation $\cmap{X}{\varepsilon}{\eta}$ of the identity $\id_X$ from $\varepsilon$ to $\eta$. By Lemma~\ref{lem:homotopy=}, the morphisms
\begin{equation}\label{eq:id1}
    [\cmap{X}{\eta}{\varepsilon}]\colon\cplx{X}{\eta}\to\cplx{X}{\varepsilon}
\end{equation}
and
\begin{equation}\label{eq:id2}
[\cmap{X}{\varepsilon}{\eta}]\colon\cplx{X}{\varepsilon}\to\cplx{X}{\eta}
\end{equation}
in $K^{[-n,0]}(\M)$ depend only on the presentations $\eta$ and $\varepsilon$.

\begin{lemma}\label{lem:mutinv}
    The morphisms $[\cmap{X}{\eta}{\varepsilon}]$ and $[\cmap{X}{\varepsilon}{\eta}]$ are mutually inverse isomorphisms. Moreover, for any morphism $\tf\colon X\to Y$ in $\pr{\f}{n+1}{\M}$, let $\eta'$ and $\varepsilon'$ be two $\M$-presentations of $Y$ of length $n+1$, and let $\tf^{\eta}_\bullet$ and $\tf^{\varepsilon}_\bullet$ be presentations of $\tf$ from $\eta$ to $\eta'$ and from $\varepsilon$ to $\varepsilon'$, respectively. Then we have the following commutative diagram
    \begin{equation}\label{eq:comdia}
        \xymatrix{
    \cplx{X}{\eta}\ar[r]^{[\cmap{X}{\eta}{\varepsilon}]}\ar[d]_{[\tf^{\eta}_\bullet]} & \cplx{X}{\varepsilon}\ar[d]^{[\tf^{\varepsilon}_\bullet]}\\
    \cplx{Y}{\eta'}\ar[r]_{[\cmap{Y}{\eta'}{\varepsilon'}]} & \cplx{Y}{\varepsilon'}.
    }
    \end{equation}
\end{lemma}

\begin{proof}
    By Remark~\ref{rmk:comp}, $\cmap{X}{\varepsilon}{\eta}\circ\cmap{X}{\eta}{\varepsilon}$ is a presentation of $\id_X$ from $\eta$ to itself, and $\cmap{X}{\eta}{\varepsilon}\circ\cmap{X}{\varepsilon}{\eta}$ is a presentation of $\id_X$ from $\varepsilon$ to itself. Hence, by Lemma~\ref{lem:homotopy=}, we have
    \[[\cmap{X}{\varepsilon}{\eta}]\circ[\cmap{X}{\eta}{\varepsilon}]=[\cmap{X}{\varepsilon}{\eta}\circ\cmap{X}{\eta}{\varepsilon}]=[\cmap{X}{\eta}{\eta}]=\id_{\cplx{X}{\eta}},\]
    and
    \[[\cmap{X}{\eta}{\varepsilon}]\circ[\cmap{X}{\varepsilon}{\eta}]=[\cmap{X}{\eta}{\varepsilon}\circ\cmap{X}{\varepsilon}{\eta}]=[\cmap{X}{\varepsilon}{\varepsilon}]=\id_{\cplx{X}{\varepsilon}}.\]
    The commutativity of the diagram~\eqref{eq:comdia} can be proved similarly.
\end{proof}

The following proposition shows that Construction~\ref{cons:tP} indeed yields a well-defined functor, which is unique up to natural isomorphism.

\begin{prop}\label{niso}
	The assignment $\tPn$ is an additive functor from $\pr{\f}{n+1}{\M}$ to $K^{b}(\M)$. Moreover, up to natural isomorphism, $\tPn$ does not depend on the choice of presentations of objects in $\pr{\f}{n+1}{\M}$.
\end{prop}
	
\begin{proof}
    For any object $X\in\pr{\f}{n+1}{\M}$, the identity morphism $\id_X$ admits a presentation $(\id_{M_n^X}, \dots, \id_{M_{0}^X})$. Hence, $\tPn(\id_X)=\id_{\tPn(X)}$. By Remark~\ref{rmk:comp}, for any two morphisms $\tf\colon X\to Y$ and $\tg\colon Y\to Z$ in $\pr{\f}{n+1}{\M}$, and any presentations $\tf_\bullet$ of $\tf$ and $\tg_\bullet$ of $\tg$, the composition $\tg_\bullet\circ\tf_\bullet$ is a presentation of $\tg\circ \tf$. It follows that
    \[\tPn(\tg\circ \tf)=[\tg_\bullet\circ\tf_\bullet]=[\tg_\bullet]\circ[\tf_\bullet]=\tPn(\tg)\circ \tPn(\tf).\]
    Therefore, $\tPn$ is a functor. Similarly, one can verify that the functor $\tPn$ is additive.

    Suppose that we choose another presentation $\varepsilon^X$ for each object $X$ in $\pr{\f}{n+1}{\M}$. Following Construction~\ref{cons:tP}, this defines another functor $\tPn'\colon\pr{\f}{n+1}{\M} \to K^{b}(\M)$ such that $\tPn'(X)=\cplx{X}{\varepsilon^X}$. By Lemma~\ref{lem:mutinv}, the family of isomorphisms $\{[\cmap{X}{\eta^X}{\varepsilon^X}]\}_{X \in \pr{\f}{n+1}{\M}}$ provides a natural isomorphism from $\tPn$ to $\tPn'$.
\end{proof}

\begin{remark}
    If the category $\f$ is $\k$-linear, where $\k$ is a field, then the functor $\tPn$ is $\k$-linear.
\end{remark}

We recall from \cite[Chapter 1]{H} the triangulated structure of $\t$. For any object $Y$ in $\f$, we fix a conflation
\[Y\xrightarrow{\ti_Y}I(Y)\xrightarrow{\tp_Y}\Sigma Y\]
in $\f$, where $I(Y)\in\p$ and $\Sigma$ is the suspension functor of $\t$. Let $Y\xrightarrow{\tu}Z\xrightarrow{\tv}X$ be an arbitrary conflation in $\f$. Since $I(Y)$ is injective in $\f$ and $\tu$ is an inflation, there exist morphisms yielding the following commutative diagram:
\begin{equation}\label{eq:tri-in-stable}
    \begin{tikzcd}
		Y \arrow[r, "\tu"] \arrow[d, equal] & Z \arrow[r, "\tv"] \arrow[d, dashed] & X \arrow[d, "{\tw}", dashed] \\
		Y \arrow[r, "\ti_Y"']                                 & I(Y) \arrow[r, "\tp_Y"']        & \Sigma Y.                        
	\end{tikzcd}
\end{equation}
Then $Y\xrightarrow{\pi(\tu)}Z\xrightarrow{\pi(\tv)}X\xrightarrow{\pi(\tw)}\Sigma Y$ is called a \emph{standard triangle} in $\t$. The triangles in $\t$ are defined to be the sequences isomorphic to standard triangles. 

Moreover, any morphism $\alpha\colon Y\to Z$ in $\t$ is isomorphic to $\pi(\tu)\colon Y\to Z'$ for some inflation $\tu\colon Y\to Z'$ in $\f$, and any morphism $\beta\colon Z\to X$ in $\t$ is isomorphic to $\pi(\tv)\colon Z''\to X$ for some deflation $\tv\colon Z''\to X$ in $\f$.

\begin{lemma}\label{inclusion}
    The canonical functor $\pi$ induces an identification of categories 
    \[\pi(\pr{\f}{n+1}{\M})=\pr{\t}{n+1}{\m}.\] 
    Consequently, $\pi$ restricts to a functor
    \[\pi\colon\pr{\f}{n+1}{\M}\to \pr{\t}{n+1}{\m},\]
    which is full and surjective on objects. The kernel ideal of this restriction consists of all morphisms that factor through objects in $\p$.
\end{lemma}
	
\begin{proof} 
    Since the canonical functor $\pi \colon \f \to \t$ is full and acts as the identity on objects, it suffices to show that the objects of $\pr{\f}{n+1}{\M}$ and $\pr{\t}{n+1}{\m}$ coincide. Since conflations of the form \eqref{presinf} are mapped under $\pi$ to triangles of the form \eqref{eq:tri for prt}, every object in $\pr{\f}{n+1}{\M}$ belongs to $\pr{\t}{n+1}{\m}$. It thus remains to show the reverse inclusion.
    
    We prove by induction on $m$ that for each $1\le m\leq n+1$, every object in $\pr{\t}{m}{\m}$ belongs to $\pr{\f}{m}{\M}$. For the base case $m=1$, we have $\pr{\t}{1}{\m}=\m$ and $\pr{\f}{1}{\M}=\M$, hence the assertion holds. Assume that the assertion holds when $m=s$ for some $1\leq s\leq n$. 
    
    Now consider the case $m=s+1$. Let $X$ be an object in $\pr{\t}{s+1}{\m}$. By definition, there exists a triangle
    \[X_1\xrightarrow{\alpha_0^X} M_0^X\xrightarrow{\beta_0^X} X\xrightarrow{\gamma_0^X} \Sigma X_1,\]
    where $M_0^X\in\m$ and $X_1\in\pr{\t}{s}{\m}$. By the facts recalled above, the morphism $\beta^X_0$ is isomorphic in $\t$ to $\pi(\tv)$ for some deflation $\tv\colon M'\to X$ in $\f$. We complete the deflation $\tv$ to a conflation $X'\xrightarrow{\tu}M'\xrightarrow{\tv}X$ in $\f$, which induces a standard triangle
    \[X'\xrightarrow{\pi(\tu)}M'\xrightarrow{\pi(\tv)}X\xrightarrow{\pi(\tw)}\Sigma X'\] 
    in $\t$. Since $\beta^X_0$ is isomorphic to $\pi(\tv)$ as morphisms ending in $X$, this isomorphism of triangles implies that $M'$ and $X'$ are isomorphic in $\t$ to $M_0^X$ and $X_1$, respectively. Since the subcategories $\m$ and $\pr{\t}{s}{\m}$ are closed under isomorphisms in $\t$, we obtain $M' \in \m$ and $X' \in \pr{\t}{s}{\m}$. By the induction hypothesis, $M'$ and $X'$ belong to $\M$ and $\pr{\f}{s}{\M}$, respectively. Consequently, $X$ belongs to $\pr{\f}{s+1}{\M}$, which completes the proof.
\end{proof}

The projection $\pi \colon \M \to \m$ naturally induces a functor
\begin{equation}\label{eq:pi}
    \pi \colon K^{[-n,0]}(\M) \to K^{[-n,0]}(\m).
\end{equation}
By Construction~\ref{cons:tP}, we have a functor $\tPn \colon \pr{\f}{n+1}{\M} \to K^{[-n,0]}(\M)$. By Lemma~\ref{homotopy}~(2), the composite functor 
\[
\pi \circ \tPn \colon \pr{\f}{n+1}{\M} \to K^{[-n,0]}(\m)
\]
annihilates all morphisms factoring through objects in $\p$. Consequently, Lemma~\ref{inclusion} and the universal property of quotient categories guarantee that $\pi \circ \tPn$ factors uniquely through the quotient functor $\pi \colon \pr{\f}{n+1}{\M} \to \pr{\t}{n+1}{\m}$.

\begin{definition}
    We define $\Pn$ to be the unique additive functor 
    \[
    \Pn \colon \pr{\t}{n+1}{\m} \to K^{[-n,0]}(\m)
    \]
    making the following diagram commute:
	\[\begin{tikzcd}
			\pr{\f}{n+1}{\M} \arrow[d, "\tPn"'] \arrow[r, "\pi"] & \pr{\t}{n+1}{\m} \arrow[d, "\Pn"] \\
			K^{[-n,0]}(\M) \arrow[r, "\pi"']                                                        & {K^{[-n,0]}(\m)}.
		\end{tikzcd}\]
\end{definition}

To investigate the action of $\Pn$ on extriangles in the next section, we need the following construction.

\begin{construction}\label{cons:tG}
    We construct, for any $X_1\in\pr{\f}{n}{\M}\subseteq\pr{\f}{n+1}{\M}$, a morphism 
    \[\tGamma_{X_1}\colon\tPn(\Sigma X_1)\to\Sigma \tPn(X_1)\]
    as follows.

    Since $X_1\in\pr{\f}{n}{\M}$, it admits an $\M$-presentation of length $n$. Anticipating its later role in the presentation of an object $X$, we denote this presentation by $(\varepsilon^{X_1}_i)_{0\leq i\leq n-1}$ and index its terms accordingly:
    \[\varepsilon^{X_1}_i\colon X_{i+2}\xrightarrow{\talpha_{i+1}^X}M^X_{i+1}\xrightarrow{\tbeta^X_{i+1}}X_{i+1}, \ 0\leq i\leq n-1.\]
    On the one hand, we extend it to an $\M$-presentation $\varepsilon^{X_1}=(\varepsilon^{X_1}_i)_{0\leq i\leq n}$ of $X_1$ of length $n+1$ by adding a conflation $\varepsilon^{X_1}_n$, all of whose terms are zero. On the other hand, we construct an $\M$-presentation $\varepsilon^{\Sigma X_1}=(\varepsilon^{\Sigma X_1}_i)_{0\leq i\leq n}$ of $\Sigma X_1$ of length $n+1$, via
    \[\varepsilon^{\Sigma X_1}_0\colon X_1 \xrightarrow{\ti_{X_1}} I(X_1)\xrightarrow{\tp_{X_1}} \Sigma X_1,\]
    and
    \[\varepsilon^{\Sigma X_1}_i\colon X_{i+1}\xrightarrow{\talpha_{i}^X}M^X_i\xrightarrow{-\tbeta^X_{i}}X_{i}, \ 1\leq i\leq n.\]
    As in \eqref{eq:cplx}, we obtain two complexes in $K^{[-n,0]}(\M)$:
    \[\cplx{X_1}{\varepsilon^{X_1}}\colon 0\to M^X_{n-1}\xrightarrow{\talpha^X_{n-2}\circ\tbeta^X_{n-1}}M^X_{n-2}\xrightarrow{\talpha^X_{n-3}\circ\tbeta^X_{n-2}}\cdots\xrightarrow{\talpha^X_1\circ\tbeta^X_2}M^X_1,\]
    and
    \[\cplx{\Sigma X_1}{\varepsilon^{\Sigma X_{1}}}\colon M^X_{n-1}\xrightarrow{-\talpha^X_{n-2}\circ\tbeta^X_{n-1}}M^X_{n-2}\xrightarrow{-\talpha^X_{n-3}\circ\tbeta^X_{n-2}}\cdots\xrightarrow{-\talpha^X_1\circ\tbeta^X_2}M^X_1\xrightarrow{-\ti_{X_1}\circ\tbeta^X_1}I(X_1).\]
    Let $\tGamma_{X_1;\,\varepsilon^{X_1}}\colon\cplx{\Sigma X_1}{\varepsilon^{\Sigma X_1}}\to\Sigma\cplx{X_1}{\varepsilon^{X_1}}$ be the homotopy class of the following chain map:
    \begin{equation}\label{eq:tGXvar}
        \xymatrix@C=4em{
        M^X_{n}\ar[d]^{\id}\ar[r]^{-\talpha^X_{n-1}\circ\tbeta^X_{n}}&M^X_{n-1}\ar[d]^{\id}\ar[r]^{-\talpha^X_{n-2}\circ\tbeta^X_{n-1}}&\cdots\ar[r]^{-\talpha^X_1\circ\tbeta^X_2}&M^X_1\ar[d]^{\id}\ar[r]^{-\ti_{X_1}\circ\tbeta^X_1}&I(X_1)\ar[d]^{0}\\
        M^X_{n}\ar[r]_{-\talpha^X_{n-1}\circ\tbeta^X_{n}}&M^X_{n-1}\ar[r]_{-\talpha^X_{n-2}\circ\tbeta^X_{n-1}}&\cdots\ar[r]_{-\talpha^X_1\circ\tbeta^X_2}&M^X_1\ar[r]_{0}&0.
        }
    \end{equation}
    Recall from Construction~\ref{cons:tP} that $\tPn(\Sigma X_1)=\cplx{\Sigma X_1}{\eta^{\Sigma X_1}}$ and $\Sigma\tPn(X_1)=\Sigma\cplx{X_1}{\eta^{X_1}}$. We define
    \begin{equation}\label{eq:tGX}
        \tGamma_{X_1}=\Sigma[\cmap{X_1}{\varepsilon^{X_1}}{\eta^{X_1}}]\circ\tGamma_{X_1;\, \varepsilon^{X_1}}\circ[\cmap{\Sigma X_1}{\eta^{\Sigma X_1}}{\varepsilon^{\Sigma X_1}}]\colon \tPn(\Sigma X_1)\to\Sigma\tPn(X_1).
    \end{equation}
    For any $X_1\in\pr{\t}{n}{\m}$, since $X_1$ is also an object of $\pr{\f}{n}{\M}$ by Lemma~\ref{inclusion}, we define 
    \[\Gamma_{X_1}:=\pi(\tGamma_{X_1})\colon\Pn(\Sigma X_1)\to\Sigma \Pn(X_1),\]
    where $\pi\colon K^{[-n,0]}(\M)\to K^{[-n,0]}(\m)$ is the functor shown in \eqref{eq:pi}.
\end{construction}

The following lemma verifies that the morphisms $\Gamma_{X_1}$ constructed above assemble into a natural isomorphism.

\begin{lemma}\label{lem:natiso}
    $\Gamma:=(\Gamma_{X_1})_{X_1\in\pr{\t}{n}{\m}}$ is a natural isomorphism from $\Pn\circ\Sigma$ to $\Sigma\circ\Pn$, as functors from $\pr{\t}{n}{\m}$ to $K^{[-n,0]}(\m)$.
\end{lemma}

\begin{proof}
    For any object $X\in\pr{\t}{n}{\m}$, since $I(X)=0$ in $\t$, the morphism $\pi(\tGamma_{X;\, \varepsilon^{X}})$ is an isomorphism in $K^{[-n,0]}(\m)$. By Lemma~\ref{lem:mutinv}, both $[\cmap{\Sigma X}{\eta^{\Sigma X}}{\varepsilon^{\Sigma X}}]$ and $[\cmap{X}{\varepsilon^{X}}{\eta^{X}}]$ are isomorphisms, hence their images under $\pi$ are also isomorphisms. Therefore, $\Gamma_{X}$ is an isomorphism.

    For any morphism $\tf\colon X\to Y$ in $\pr{\f}{n}{\M}$, let $\tf^{\eta}_\bullet$ and $\tf^{\varepsilon}_\bullet$ be presentations of $\tf$ from $\eta^{X}$ to $\eta^{Y}$ and from $\varepsilon^{X}$ to $\varepsilon^{Y}$, respectively. For the morphism $\Sigma\tf$, let $(\Sigma\tf)^{\eta}_\bullet$ and $(\Sigma\tf)^{\varepsilon}_\bullet$ be presentations of $\Sigma\tf$ from $\eta^{\Sigma X}$ to $\eta^{\Sigma Y}$ and from $\varepsilon^{\Sigma X}$ to $\varepsilon^{\Sigma Y}$, respectively. We construct another presentation $(\Sigma\tf)^{\varepsilon,\prime}_\bullet$ of $\Sigma\tf$ from $\varepsilon^{\Sigma X}$ to $\varepsilon^{\Sigma Y}$ as follows. Let $(\Sigma\tf)^{\varepsilon,\prime}_{i+1}=\tf^{\varepsilon}_{i}$ for $0\le i\le n-1$, and let $(\Sigma\tf)^{\varepsilon,\prime}_{0}$ be the morphism induced by the injectivity of $I(X)$. Set $(\Sigma\tf)^{\varepsilon,\prime}_\bullet:=((\Sigma\tf)^{\varepsilon,\prime}_{i})_{0\le i\le n}$. By Lemma~\ref{lem:homotopy=}, we have $[(\Sigma\tf)^{\varepsilon}_\bullet]=[(\Sigma\tf)^{\varepsilon,\prime}_\bullet]$. Using this construction, a straightforward check establishes the equality
    \[\Sigma[\tf^{\varepsilon}_\bullet]\circ\tGamma_{X;\, \varepsilon^X} = \tGamma_{Y;\, \varepsilon^Y}\circ[(\Sigma\tf)^{\varepsilon}_\bullet].\]
    Applying this, we can now compute
    \[\begin{array}{rcl}
        \Sigma\tPn(\tf)\circ\tGamma_X & = & \Sigma[\tf^{\eta}_\bullet]\circ\Sigma[\cmap{X}{\varepsilon^X}{\eta^X}]\circ\tGamma_{X;\, \varepsilon^X}\circ[\cmap{\Sigma X}{\eta^{\Sigma X}}{\varepsilon^{\Sigma X}}] \\
        & = & \Sigma[\cmap{Y}{\varepsilon^Y}{\eta^Y}]\circ\Sigma[\tf^{\varepsilon}_\bullet]\circ\tGamma_{X;\, \varepsilon^X}\circ[\cmap{\Sigma X}{\eta^{\Sigma X}}{\varepsilon^{\Sigma X}}] \\
        & = & \Sigma[\cmap{Y}{\varepsilon^Y}{\eta^Y}]\circ\tGamma_{Y;\, \varepsilon^Y}\circ[(\Sigma\tf)^{\varepsilon}_\bullet]\circ[\cmap{\Sigma X}{\eta^{\Sigma X}}{\varepsilon^{\Sigma X}}] \\
        & = & \Sigma[\cmap{Y}{\varepsilon^Y}{\eta^Y}]\circ\tGamma_{Y;\, \varepsilon^Y}\circ[\cmap{\Sigma Y}{\eta^{\Sigma Y}}{\varepsilon^{\Sigma Y}}]\circ[(\Sigma\tf)^{\eta}_\bullet] \\
        & = & \tGamma_Y\circ\tPn(\Sigma\tf),
    \end{array}\]
    where the first and fifth equalities follow from the constructions, and the second and fourth equalities follow from the commutative diagram~\eqref{eq:comdia} in Lemma~\ref{lem:mutinv}. Applying the functor $\pi$, we obtain the following commutative diagram for any morphism $f\colon X\to Y$ in $\pr{\t}{n}{\m}$:
    \[\xymatrix{
    \Pn(\Sigma X)\ar[r]^{\Gamma_X}\ar[d]_{\Pn(\Sigma f)}&\Sigma \Pn(X)\ar[d]^{\Sigma \Pn(f)}\\
    \Pn(\Sigma Y)\ar[r]_{\Gamma_Y}&\Sigma \Pn(Y).
    }\]
    Therefore, $\Gamma$ is a natural isomorphism from $\Pn\circ\Sigma$ to $\Sigma\circ\Pn$.
\end{proof}

Before proceeding to the next section, it is instructive to explicitly compute the action of the functor $\tPn$ on the morphisms of a standard triangle arising from an $\M$-presentation. These explicit formulas will be crucial for understanding the action of $\Pn$ on extriangles.
	
\begin{exam}\label{intrinsic} 
    Let $X\in\pr{\f}{n+1}{\M}$ with its $\M$-presentation $\eta^X=(\eta^X_i)_{0\leq i\leq n}$:
    \[\eta^X_i\colon X_{i+1}\xrightarrow{\talpha^X_i}M^X_i\xrightarrow{\tbeta^X_i}X_i, \ 0\leq i\leq n.\]
    The first conflation $\eta^X_0$ gives rise to a triangle in $\t$:
    \[X_1\xrightarrow{\pi(\talpha^X_0)}M^X_0\xrightarrow{\pi(\tbeta^X_0)}X\xrightarrow{\pi(\tgamma^X_0)} \Sigma X_1,\]
    where $\tgamma^X_0$ comes from the following commutative diagram in $\f$:
    \[\begin{tikzcd}
		X_1 \arrow[r, "\talpha^X_0"] \arrow[d, equal] & M^X_0 \arrow[r, "\tbeta^X_0"] \arrow[d, "\tz"] & X \arrow[d, "{\tgamma^X_0}"] \\
		X_1 \arrow[r, "\ti_{X_1}"']                                 & I(X_1) \arrow[r, "\tp_{X_1}"']        & \Sigma X_1.                        
	\end{tikzcd}\]
    Consider the $\M$-presentation $\varepsilon^{X_1}=(\varepsilon^{X_1}_i)_{0\leq i\leq n}$ of $X_1$, where $\varepsilon^{X_1}_i=\eta^X_{i+1}$ for all $0\leq i\leq n-1$ and $\varepsilon^{X_1}_n$ is the conflation, all of whose items are zero. As in Construction~\ref{cons:tG}, $\Sigma X_1$ admits an $\M$-presentation $\varepsilon^{\Sigma X_1}=(\varepsilon^{\Sigma X_1}_i)_{0\leq i\leq n}$, where $\varepsilon^{\Sigma X_1}_0$ is the conflation in the last row of the above diagram, and 
    \[\varepsilon^{\Sigma X_1}_i\colon X_{i+1}\xrightarrow{\talpha^X_i}M^X_i\xrightarrow{-\tbeta^X_i}X_i,\ 1\leq i\leq n.\]

    A direct calculation shows that 
    \begin{itemize}
        \item $\tPn(\talpha_{0}^{X})$ is the composition of the morphism 
        \[[\cmap{X_1}{\eta^{X_1}}{\varepsilon^{X_1}}]\colon\tPn(X_1)\to\cplx{X_1}{\varepsilon^{X_1}}\] 
        with the homotopy class of the chain map
        \[\begin{tikzcd}[column sep=3em]
				0 \arrow[r] \arrow[d] & M_{n}^{X} \arrow[r, "\talpha_{n-1}^{X}\circ\tbeta_{n}^{X}"] \arrow[d] & M_{n-1}^{X} \arrow[r, "\talpha_{n-2}^{X}\circ\tbeta_{n-1}^{X}"] \arrow[d] & \cdots \arrow[r, "\talpha_{1}^{X}\circ\tbeta_{2}^{X}"] & M_{1}^{X} \arrow[d, "\talpha_{0}^{X}\circ\tbeta_{1}^{X}"] \\
				0 \arrow[r]           & 0 \arrow[r]                                       & 0 \arrow[r]                                                             & \cdots \arrow[r]                                     & M_{0}^{X},                                           
			\end{tikzcd}\]
        \item $\tPn(\tbeta_{0}^{X})$ is the homotopy class of the chain map
		\[\begin{tikzcd}[column sep=3em]
				0 \arrow[r] \arrow[d]                   & 0 \arrow[r] \arrow[d]                                         & \cdots \arrow[r]                                     & M_{0}^{X} \arrow[d, "\id_{M_{0}^{X}}"] \\
				M_{n}^{X} \arrow[r, "\talpha_{n-1}^{X}\circ\tbeta_{n}^{X}"] & M_{n-1}^{X} \arrow[r, "\talpha_{n-2}^{X}\circ\tbeta_{n-1}^{X}"] & \cdots \arrow[r, "\talpha_{0}^{X}\circ\tbeta_{1}^{X}"] & M_{0}^{X},                              
		\end{tikzcd}\]
        \item $\tPn(\tgamma_{0}^{X})$ is the composition of the homotopy class of the chain map
		\begin{equation}\label{eq:tg}
		    \begin{tikzcd}[column sep=4em]
				M_{n}^{X} \arrow[r, "\talpha_{n-1}^{X}\circ\tbeta_{n}^{X}"] \arrow[d, "(-1)^n\id"'] & M_{n-1}^{X} \arrow[r, "\talpha_{n-2}^{X}\circ\tbeta_{n-1}^{X}"] \arrow[d, "(-1)^{n-1}\id"] & \cdots \arrow[r, "\talpha_{1}^{X}\circ\tbeta_{2}^{X}"] & M_{1}^{X} \arrow[r, "\talpha_{0}^{X}\circ\tbeta_{1}^{X}"] \arrow[d, "-\id"'] & M_{0}^{X} \arrow[d, "\tz"] \\
				M_{n}^{X} \arrow[r, "-\talpha_{n-1}^{X}\circ\tbeta_{n}^{X}"']                                & M_{n-1}^{X} \arrow[r, "-\talpha_{n-2}^{X}\circ\tbeta_{n-1}^{X}"']                                & \cdots \arrow[r, "-\talpha_{1}^{X}\circ\tbeta_{2}^{X}"'] & M_{1}^{X} \arrow[r, "-\ti_{X_1}\circ\tbeta_{1}^{X}"']                                                                    & I(X_{1}).                 
			\end{tikzcd}
		\end{equation}
        with the morphism $[\cmap{\Sigma X_1}{\varepsilon^{\Sigma X_1}}{\eta^{\Sigma X_1}}]\colon\cplx{\Sigma X_1}{\varepsilon^{\Sigma X_1}}\to\tPn(\Sigma X_1)$.
    \end{itemize}
\end{exam}

\section{Extriangulated structure on \texorpdfstring{$\Pn$}{P}}\label{sec:extri}
	
In this section, we construct a natural transformation $\Omegan$ such that the pair $(\Pn,\Omegan)$ is an extriangulated functor from $\pr{\t}{n+1}{\m}$ to $K^{[-n,0]}(\m)$.

\begin{definition}\label{def:Omega}
	Let $X$ and $Y$ be objects in $\pr{\t}{n+1}{\m}$. We define a map
	\[\Omegan^{X,Y}\colon \e_{\m}(X,Y)\rightarrow\Hom_{K^{b}(\m)}(\Pn(X),\Sigma\Pn(Y))\]
	as follows. For any $f\in \e_{\m}(X,Y)$, by definition, there exist an object $V\in \pr{\t}{n}{\m}$ and morphisms $h\colon X\to \Sigma V$ and $g\colon V\to Y$ such that $f=\Sigma g \circ h$. We then define $\Omegan^{X,Y}(f)$ to be $\Sigma\Pn(g)\circ\Gamma_V\circ \Pn(h)$, that is,
    \[\Omegan^{X,Y}(f)\colon\Pn(X)\xrightarrow{\Pn(h)}\Pn(\Sigma V)\xrightarrow{\Gamma_V}\Sigma \Pn(V)\xrightarrow{\Sigma\Pn(g)}\Sigma\Pn(Y),\]
    where the morphism $\Gamma_V$ is defined in Construction~\ref{cons:tG}.
\end{definition}

The following lemma shows that each $\Omegan^{X,Y}$ is a well-defined map.

\begin{lemma}
    Let $f=\Sigma g_i \circ h_i\in \e_\m(X,\Sigma Y)$, $i=1,2$, where $h_i\colon X\to \Sigma V_i$, $g_i\colon V_i\to Y$ and $V_i\in\pr{\t}{n}{\m}$. Then 
    \[\Sigma\Pn(g_1)\circ\Gamma_{V_1}\circ \Pn(h_1) = \Sigma\Pn(g_2)\circ\Gamma_{V_2}\circ \Pn(h_2).\]
\end{lemma}

\begin{proof}
    Consider the first conflation $X_1\xrightarrow{\talpha^X_0} M^X_0\xrightarrow{\tbeta^X_0} X$ in the presentation of $X$ involved in the construction of $\Pn(X)$. Applying the functor $\pi$ to it, we obtain a triangle
    \[X_1\xrightarrow{\alpha^X_0} M^X_0\xrightarrow{\beta^X_0} X\xrightarrow{\gamma^X_0}\Sigma X_1\]
    in $\t$. Since $\Hom_\t(\m,\Sigma\pr{\t}{n}{\m})=0$, the morphism $h_i\colon X\to\Sigma V_i$ factors through $\gamma^X_0$. Hence, there exists a morphism $h'_i\colon X_1\to V_i$ such that $h_i=\Sigma h'_i\circ\gamma^X_0$. Then, on one hand, we have $f=\Sigma (g_i\circ h'_i) \circ \gamma^X_0$. It follows that $\Sigma(g_1\circ h_1'-g_2\circ h_2')\circ \gamma^X_0=0$. This implies that $g_1\circ h_1'-g_2\circ h_2'$ factors through $\alpha^X_0$, i.e., there exists a morphism $\varphi\colon M_0^X\to Y$ such that 
    \[g_1\circ h_1'-g_2\circ h_2'=\varphi\circ\alpha^X_0.\]
    On the other hand, we have
    \[\begin{array}{rcl}
        \Sigma\Pn(g_i)\circ\Gamma_{V_i}\circ\Pn(h_i) & = & \Sigma\Pn(g_i)\circ\Gamma_{V_i}\circ \Pn(\Sigma h'_i)\circ\Pn(\gamma^X_0)\\
        & = & \Sigma\Pn(g_i)\circ\Sigma\Pn(h'_i)\circ\Gamma_{X_1}\circ\Pn(\gamma^X_0)\\
        & = & \Sigma\Pn(g_i\circ h'_i)\circ\Gamma_{X_1}\circ\Pn(\gamma^X_0)
    \end{array}\]
    where the second equality follows from the naturality of $\Gamma$ (see Lemma~\ref{lem:natiso}). Hence, 
    \[\begin{array}{rcl}
         \Sigma\Pn(g_1)\circ\Gamma_{V_1}\circ\Pn(h_1)-\Sigma\Pn(g_2)\circ\Gamma_{V_2}\circ \Pn(h_2) & = & \Sigma\Pn(g_1\circ h'_1-g_2\circ h'_2)\circ\Gamma_{X_1}\circ\Pn(\gamma^X_0) \\
         & = & \Sigma\Pn(\varphi\circ\alpha^X_0)\circ\Gamma_{X_1}\circ\Pn(\gamma^X_0)\\
         & = & \Sigma\Pn(\varphi)\circ\Sigma\Pn(\alpha^X_0)\circ\Gamma_{X_1}\circ\Pn(\gamma^X_0)\\
         & = & \Sigma\Pn(\varphi)\circ\Gamma_{M^X_0}\circ\Pn(\Sigma\alpha^X_0)\circ\Pn(\gamma^X_0)\\
         & = & \Sigma\Pn(\varphi)\circ\Gamma_{M^X_0}\circ\Pn(\Sigma \alpha^X_0\circ\gamma^X_0)\\
         & = & \Sigma\Pn(\varphi)\circ\Gamma_{M^X_0}\circ\Pn(0)\\
         & = & 0,
    \end{array}
    \]
    as required.
\end{proof}

\begin{remark}\label{Gamma}
    Let us consider a special case where $X\in\pr{\t}{n+1}{\m}$ and $Y\in\pr{\t}{n}{\m}$. In this case, $\e_{\m}(X,Y)=\Hom_\t(X,\Sigma Y)$. For any $f\in\Hom_\t(X,\Sigma Y)$, by taking $V=Y$, we obtain the decomposition $f=\Sigma \id_Y \circ f$. It follows from the definition that
    \[\Omegan^{X,Y}(f)=\Gamma_Y\circ\Pn(f).\]
    Since $\Gamma_Y$ is an isomorphism, the map $\Omegan^{X,Y}$ is injective, surjective, or bijective if and only if the map
    \[\Pn\colon\Hom_\t(X,\Sigma Y)\to \Hom_{K^{[-n,0]}(\m)}(\Pn(X),\Pn(\Sigma Y))\]
    is injective, surjective, or bijective, respectively.
\end{remark}

Let $\Omegan=\{\Omegan^{X,Y}\mid X,Y\in\pr{\t}{n+1}{\m}\}$. Recall from Definition~\ref{defn:extri fun} the notion of extriangulated functors.
	
\begin{theorem}\label{preserve}
	The collection $\Omegan$
    is a natural transformation from the functor $\e_\m(-,-)$ to the functor $\e_{K^{[-n,0]}(\m)}(\Pn(-),\Pn(-))$. Moreover, the pair $(\Pn,\Omegan)$ is an extriangulated functor from $(\pr{\t}{n+1}{\m},\e_\m)$ to $(K^{[-n,0]}(\m),\e_{K^{[-n,0]}(\m)})$.
\end{theorem}
	
\begin{proof}
    It follows directly from the definition that for any objects $X$ and $Y$ in $\pr{\t}{n+1}{\m}$ and any elements $f_1,f_2\in\e_\m(X,Y)$, we have $\Omegan^{X,Y}(f_1+f_2)=\Omegan^{X,Y}(f_1)+\Omegan^{X,Y}(f_2)$. Thus, each map $\Omegan^{X,Y}\colon\e_\m(X,Y)\to \Hom_{K^b(\m)}(\Pn(X),\Sigma \Pn(Y))$ is a homomorphism of abelian groups. For any objects $X,X',Y,Y'\in\pr{\t}{n+1}{\m}$, any $f\in\e_{\m}(X,Y)$, and any morphisms $u\colon X'\to X$ and $v\colon Y\to Y'$, let $f=\Sigma g \circ h$, where $h\colon X\to \Sigma V$ and $g\colon V\to Y$ are morphisms with $V\in \pr{\t}{n}{\m}$. Then $\Sigma v \circ f\circ u = (\Sigma (v\circ g))\circ (h\circ u)$. Hence, by definition, we have
    \[\begin{array}{rcl}
        \Omegan^{X',Y'}(\Sigma v \circ f\circ u) & = & \Sigma\Pn(v\circ g)\circ\Gamma_V\circ \Pn(h\circ u)\\
        & = & \Sigma\Pn(v)\circ\Sigma\Pn(g)\circ\Gamma_V\circ\Pn(h)\circ\Pn(u)\\
        & = & \Sigma\Pn(v)\circ\Omegan^{X,Y}(f)\circ\Pn(u).
    \end{array}\]
    This shows that $\Omegan$ is a natural transformation.
    
    To prove that the pair $(\Pn,\Omegan)$ is an extriangulated functor, by definition, it suffices to show that for any extriangle $Y'\xrightarrow{u'}Z'\xrightarrow{v'}X'\xrightarrow{w'}\Sigma Y'$ in $\pr{\t}{n+1}{\m}$, the sequence
    \[\Pn(Y')\xrightarrow{\Pn(u')}\Pn(Z')\xrightarrow{\Pn(v')}\Pn(X')\xrightarrow{\Omegan^{X',Y'}(w')}\Sigma\Pn(Y')\]
    is a triangle in $K^b(\m)$. By the triangulated structure of $\t$, the triangle $Y'\xrightarrow{u'}Z'\xrightarrow{v'}X'\xrightarrow{w'}\Sigma Y'$ is isomorphic to the triangle induced by a conflation $Y\xrightarrow{\tu}Z\xrightarrow{\tv}X$ in $\f$. That is, there exists an isomorphism of triangles in $\t$:
    \[\xymatrix{
    Y\ar[r]^{\pi(\tu)} \ar[d]^{f}& Z\ar[r]^{\pi(\tv)}\ar[d]^{g} & X\ar[r]^{\pi(\tw)}\ar[d]^{h} & \Sigma Y\ar[d]^{\Sigma f} \\
    Y'\ar[r]^{u'} & Z'\ar[r]^{v'} & X'\ar[r]^{w'} & \Sigma Y',
    }\]
    where $\tw$ is obtained as shown in diagram~\eqref{eq:tri-in-stable}. Then we obtain a diagram in $K^{[-n,0]}(\m)$:
    \[\xymatrix@C=1.5cm{
    \Pn(Y)\ar[r]^{\Pn(\pi(\tu))} \ar[d]^{\Pn(f)}& \Pn(Z)\ar[r]^{\Pn(\pi(\tv))}\ar[d]^{\Pn(g)} & \Pn(X)\ar[r]^{\Omegan^{X,Y}(\pi(\tw))}\ar[d]^{\Pn(h)} & \Sigma \Pn(Y)\ar[d]^{\Sigma \Pn(f)} \\
    \Pn(Y')\ar[r]^{\Pn(u')} & \Pn(Z')\ar[r]^{\Pn(v')} & \Pn(X')\ar[r]^{\Omegan^{X',Y'}(w')} & \Sigma \Pn(Y'),
    }\]
    where the first two squares commute since $\Pn$ is a functor, and the third square commutes since $\Omegan$ is a natrual transformation. Therefore, to show the second row is a triangle in $K^b(\m)$, it suffices to show that the first row is a triangle.
    
    Starting with the presentations 
    \[X_{i+1}\xrightarrow{\talpha^X_i}M^X_i\xrightarrow{\tbeta^X_i}X_i,\ 0\leq i\leq n\text{ and }Y_{i+1}\xrightarrow{\talpha^Y_i}M^Y_i\xrightarrow{\tbeta^Y_i}Y_i,\ 0\leq i\leq n\]
    of $X$ and $Y$ in $\f$ respectively, we adapt the method in~\cite[Theorem 12.8]{Buhler} (horseshoe lemma) to construct a presentation of $Z$ in $\f$ as follows. 
    
    We recursively construct conflations in $\f$:
    \[Y_i\xrightarrow{\tu_i}Z_i\xrightarrow{\tv_i}X_i,\ 0\leq i\leq n+1,\]
    where the initial case ($i=0$) is $Y\xrightarrow{\tu}Z\xrightarrow{\tv}X$. Assume that we have constructed the conflation $Y_i\xrightarrow{\tu_i}Z_i\xrightarrow{\tv_i}X_i$ for some $0\leq i\leq n$. Consider the following commutative diagram in $\f$:
    \[\begin{tikzcd}
		Y_i \arrow[r, "\tu_i"] \arrow[d, equal] & Z_i \arrow[r, "\tv_i"] \arrow[d, dashed] & X_i \arrow[d, "{\tw_i}", dashed] \\
		Y_i \arrow[r, "\ti_{Y_i}"']                                 & I(Y_i) \arrow[r, "\tp_{Y_i}"']        & \Sigma Y_i.
	\end{tikzcd}\]
    If $i=0$, since $w'$ factors through an object in $\Sigma\pr{\t}{n}{\m}$, then so does $\pi(\tw_0)=\pi(\tw)$; if $i>0$, the fact that $Y_i\in \pr{\t}{n}{\m}$ implies that $\pi(\tw_i)$ factors through an object in $\Sigma\pr{\t}{n}{\m}$. In either case, we have $\pi(\tw_i)\circ\pi(\tbeta^X_i)=0$, which implies that $\tw_i\circ\tbeta^X_i$ factors through the morphism $\tp_{Y_i}$. Therefore, $\tbeta^X_i$ factors through $\tv_i$, that is, there exists a morphism $\tf_i\colon  M^X_i\to Z_i$ such that 
    \begin{equation}\label{eq:bvf}
        \tbeta^X_i=\tv_i\circ\tf_i.
    \end{equation}
    This yields the two commutative squares shown in the right part of the following diagram.
    \begin{equation}\label{eq:3times3}
        \begin{tikzcd}[ampersand replacement=\&,column sep=4em, row sep=3em]
		Y_{i+1} \arrow[r, "\talpha_{i}^{Y}"] \arrow[d, "{\tu_{i+1}}"', dashed]                                                                         \& M_{i}^{Y} \arrow[r, "\tbeta_{i}^{Y}"] \arrow[d, "\begin{pmatrix}0\\\id\end{pmatrix}"]                                                           \& Y_{i} \arrow[d, "\tu_{i}"] \\
		Z_{i+1} \arrow[d, "{\tv_{i+1}}"', dashed] \arrow[r, "{\begin{pmatrix}\tg_{i+1}\\\th_{i+1}\end{pmatrix}}", dashed] \& M_{i}^{X}\oplus M_{i}^{Y} \arrow[d, "{\begin{pmatrix} \id & 0 \end{pmatrix}}"] \arrow[r, "{\begin{pmatrix}\tf_{i} & \tu_{i}\circ\tbeta_{i}^{Y}\end{pmatrix}}", dashed] \& Z_{i} \arrow[d, "\tv_{i}"] \\
		X_{i+1} \arrow[r, "\talpha_{i}^{X}"']                                                                                                                   \& M_{i}^{X} \arrow[r, "\tbeta_{i}^{X}"']                                                                                                          \& X_{i}.                     
	   \end{tikzcd}
    \end{equation}
    Let $\begin{pmatrix}\tg_{i+1}\\\th_{i+1}\end{pmatrix}\colon Z_{i+1}\to M_{i}^{X}\oplus M_{i}^{Y}$ be the kernel of $\left(\tf_{i}\;\;\tu_{i}\circ\tbeta_{i}^{Y}\right)$. Then by the universal property of kernels, there exist morphisms $\tu_{i+1}$ and $\tv_{i+1}$ such that the two squares in the left part commute. Since all the rows, the middle column and the right column are conflations, by \cite[Corollary 3.6]{Buhler}, the left column is also a conflation, which is what we want to construct. 

    Since $X_{n+1}=Y_{n+1}=0$, we have $Z_{n+1}=0$. Therefore, the conflations in the middle row of the above diagram form a presentation of $Z$ in $\f$, which is denoted by $\varepsilon^Z$. We may assume that $\varepsilon^Z=\eta^Z$ and hence $\tPn(Z)=\cplx{Z}{\varepsilon^Z}$. This is because otherwise, using Lemma~\ref{lem:mutinv}, we can equivalently prove that
    \[\Pn(Y)\xrightarrow{\pi(\cmap{Z}{\eta^Z}{\varepsilon^Z})\circ\Pn(\pi(\tu))}\cplx{Z}{\varepsilon^Z}\xrightarrow{\Pn(\pi(\tv))\circ\pi(\cmap{Z}{\varepsilon^Z}{\eta^Z})}\Pn(X)\xrightarrow{\Omegan^{X,Y}(\pi(\tw))}\Sigma\Pn(Y)\]
    is a triangle. 
    
    Note that the commutativity of the bottom-left square in diagram~\ref{eq:3times3} implies 
    \begin{equation}\label{eq:gav}
        \tg_{i+1}=\talpha_{i}^{X}\circ\tv_{i+1},
    \end{equation}
    while the commutativity of the up-left square implies
    \begin{equation}\label{eq:ahu}
        \talpha^Y_{i}=\th_{i+1}\circ\tu_{i+1}.
    \end{equation}
    Let $\tdelta_{i}:=-\th_{i+1}\circ\tf_{i+1}$ for $0\leq i\leq n-1$. Then using \eqref{eq:bvf}, \eqref{eq:gav} and \eqref{eq:ahu}, we have
	\[
	\begin{pmatrix}
		\tg_{i+1} \\ \th_{i+1} 
	\end{pmatrix} \circ \begin{pmatrix}
		\tf_{i+1} & \tu_{i+1} \circ \tbeta_{i+1}^{Y} \end{pmatrix} = \begin{pmatrix}
		\talpha_{i}^{X} \circ \tbeta_{i+1}^{X} & 0\\-\tdelta_{i} & \talpha_{i}^{Y} \circ \tbeta_{i+1}^{Y}\end{pmatrix}.
	\]
    Hence, by Construction~\ref{cons:tP}, $\tPn(Z)$ is the following complex
    \[M_{n}^{X}\oplus M_{n}^{Y}
    \xrightarrow{\begin{pmatrix}\talpha^X_{n-1}\circ\tbeta^X_n & 0\\-\tdelta_{n-1} & \talpha_{n-1}^{Y}\circ\tbeta^Y_n\end{pmatrix}}
    M_{n-1}^{X}\oplus M_{n-1}^{Y}\to
		\cdots\xrightarrow{\begin{pmatrix}\talpha_{0}^{X}\circ\tbeta_{1}^{X} & 0\\-\tdelta_{0} &  \talpha_{0}^{Y}\circ\tbeta_{1}^{Y}\end{pmatrix}}M_{0}^{X}\oplus M_{0}^{Y}.\]
    Thus, we obtain the following chain map between complexes of objects in $\M$:
    \begin{equation}\label{eq:chainmap}
        \begin{tikzcd}
			& M_{n}^{X} \arrow[r, "\talpha_{n-1}^{X}\circ\tbeta^X_n"] \arrow[d, "\tdelta_{n-1}"'] & M_{n-1}^{X} \arrow[r, "\talpha_{n-2}^{X}\circ\tbeta_{n-1}^{X}"] \arrow[d, "\tdelta_{n-2}"'] & \cdots \arrow[r, "\talpha_{1}^{X}\circ\tbeta_{2}^{X}"]   & M_{1}^{X} \arrow[d, "\tdelta_{0}"] \arrow[r, "\talpha_{0}^{X}\circ\tbeta_{1}^{X}"] & M_{0}^{X} \\
			M_{n}^{Y} \arrow[r, "-\talpha_{n-1}^{Y}\circ\tbeta^Y_n"'] & M_{n-1}^{Y} \arrow[r, "-\talpha_{n-2}^{Y}\circ\tbeta_{n-1}^{Y}"'] & M_{n-2}^{Y} \arrow[r, "-\talpha_{n-3}^{Y}\circ\tbeta_{n-2}^{Y}"']                     & \cdots \arrow[r, "-\talpha_{0}^{Y}\circ\tbeta_{1}^{Y}"'] & M_{0}^{Y}.
		\end{tikzcd}
    \end{equation}
    
    By the commutative diagram~\ref{eq:3times3}, $(\begin{pmatrix}0\\\id\end{pmatrix},\dots,\begin{pmatrix}0\\\id\end{pmatrix})$ and $(\begin{pmatrix}\id & 0\end{pmatrix},\dots,\begin{pmatrix}\id & 0\end{pmatrix})$ are presentations of $\tu$ and $\tv$, respectively. Hence, we have
    \[\Pn(\tu)=[\begin{pmatrix}0\\\id\end{pmatrix},\dots,\begin{pmatrix}0\\\id\end{pmatrix}]\text{ and }\Pn(\tv)=[\begin{pmatrix}\id & 0\end{pmatrix},\dots,\begin{pmatrix}\id & 0\end{pmatrix}].\]
    Thus, to complete the proof, it suffices to show that $\Omegan^{X,Y}(\pi(\tw))$ is the homotopy class of the image of the chain map~\eqref{eq:chainmap} under the functor $\pi\colon K^{[-n,0]}(\M)\to K^{[-n,0]}(\m)$.
    
    We recursively construct morphisms $\tDelta_i\colon X_{i+1}\to Y_{i}$ for $i=n,n-1,\dots,0$, such that for any $0\leq i\leq n-1$, the following is a morphism of conflations in $\f$:
    \begin{equation}\label{eq:morph}
        \xymatrix{
    X_{i+2}\ar[r]^{\talpha^X_{i+1}}\ar[d]_{\tDelta_{i+1}} & M^X_{i+1}\ar[r]^{\tbeta^X_{i+1}}\ar[d]^{(-1)^{i+1}\tdelta_i} & X_{i+1}\ar[d]^{\tDelta_{i}}\\
    Y_{i+1}\ar[r]_{\talpha^Y_i} & M^Y_{i}\ar[r]_{\tbeta^Y_i} & Y_i.
    }
    \end{equation}
    For the initial case $i=n$, we set $\tDelta_{n}=0$. Since $X_{n+1}=0$, it follows that $(-1)^{n}\tdelta_{n-1}\circ \talpha^X_{n}=\talpha^Y_{n-1}\circ\tDelta_n=0$. Now, assume that $\tDelta_{i+1}$ has been constructed for some $1\leq i+1\leq n$ such that the left square in the above diagram commutes. By the universal property of cokernels, there exists a morphism $\tDelta_i$, making the right square commute. Then
    \[\talpha^Y_{i-1}\circ\tDelta_i\circ\tbeta^X_{i+1}=\talpha^Y_{i-1}\circ\tbeta^Y_i\circ(-1)^{i+1}\tdelta_{i}=(-1)^{i}\tdelta_{i-1}\circ\talpha^X_{i}\circ\tbeta^X_{i+1},\]
    which implies $\talpha^Y_{i-1}\circ\tDelta_i=(-1)^{i}\tdelta_{i-1}\circ\talpha^X_{i}$, because $\tbeta^X_{i+1}$ is an epimorphism. This completes the recursive construction. Let $\tdelta=\tDelta_0\colon X_1\to Y_0=Y$. Then the morphisms of conflations in \eqref{eq:morph} form a presentation of $\tdelta$.

    Since $\tu\circ\tdelta\circ\tbeta^X_1=-\tu\circ\tbeta^Y_0\circ\tdelta_0=\tf_0\circ\talpha^X_0\circ\tbeta^X_1$, where the second equality is due to the conflation in diagram~\eqref{eq:3times3}, we have $\tu\circ\tdelta=\tf_0\circ\talpha^X_0$. Then we obtain the following morphism of conflations in $\f$:
    \[\begin{tikzcd}
	    X_{1} \arrow[r, "\talpha_{0}^{X}"] \arrow[d, "\tdelta"'] & M_{0}^{X} \arrow[r, "\tbeta_{0}^{X}"] \arrow[d, "\tf_{0}"] & X \arrow[d, "\id"] \\
	    Y \arrow[r, "\tu"'] & Z \arrow[r, "\tv"']                                & X
	\end{tikzcd},\]
    which induces a morphism of triangles in $\t$ (cf. e.g., \cite[Proof of Proposition~3.3.2]{Kra22}):
    \[\begin{tikzcd}
	    X_{1} \arrow[r, "\pi(\talpha_{0}^{X})"] \arrow[d, "\pi(\tdelta)"'] & M_{0}^{X} \arrow[r, "\pi(\tbeta_{0}^{X})"] \arrow[d, "\pi(\tf_{0})"] & X \arrow[d, equal] \arrow[r, "\pi(\tgamma_{0}^{X})"] & \Sigma X_{1} \arrow[d, "\Sigma\pi(\tdelta)"] \\
	    Y \arrow[r, "\pi(\tu)"'] & Z \arrow[r, "\pi(\tv)"'] & X \arrow[r, "\pi(\tw)"'] & \Sigma Y.
	\end{tikzcd}\]
	In particular, we have $\pi(\tw)=\Sigma \pi(\tdelta)\circ\pi(\tgamma_{0}^{X})$. Hence, by Definition~\ref{def:Omega},
    \[\Omegan^{X,Y}(\pi(\tw))=\Sigma\Pn(\pi(\tdelta))\circ\Gamma_{X_1}\circ\Pn(\pi(\tgamma^X_0))=\Sigma\pi(\tPn(\tdelta))\circ\pi(\tGamma_{X_1})\circ\pi(\tPn(\tgamma^X_0))=\pi(\Sigma \tPn(\tdelta)\circ\tGamma_{X_1}\circ \tPn(\tgamma^X_0)),\]
    where the second equality follows from the definitions of $\Pn$ and $\Gamma_{X_1}$. Hence, it suffices to show that $\Sigma \tPn(\tdelta)\circ\tGamma_{X_1}\circ \tPn(\tgamma^X_0)$ equals the chain map shown in \eqref{eq:chainmap}. 
    
    Let $\varepsilon^{X_1}$ and $\varepsilon^{\Sigma X_1}$ be the presentations of $X_1$ and $\Sigma X_1$ considered in Example~\ref{intrinsic}. Then the morphisms of conflations shown in \eqref{eq:morph} form a presentation of $\tdelta$ from $\varepsilon^{X_1}$ to $\eta^{Y}$. Thus, the morphism $\tPn(\tdelta)$ is the composition of the morphism
    \[[\cmap{X_1}{\eta^{X_1}}{\varepsilon^{X_1}}]\colon\tPn(X_1)\to\cplx{X_1}{\varepsilon^{X_1}}\]
    and the homotopy class of the chain map
    \begin{equation}\label{eq:tdelta}
        \begin{tikzcd}[column sep=4em]
		& M_{n}^{X} \arrow[r, "\talpha_{n-1}^{X}\circ\tbeta^X_n"] \arrow[d, "(-1)^n\tdelta_{n-1}"'] & M_{n-1}^{X} \arrow[r, "\talpha_{n-2}^{X}\circ\tbeta_{n-1}^{X}"] \arrow[d, "(-1)^{n-1}\tdelta_{n-2}"'] & \cdots \arrow[r, "\talpha_{1}^{X}\circ\tbeta_{2}^{X}"]   & M_{1}^{X} \arrow[d, "-\tdelta_{0}"]  \\
		M_{n}^{Y} \arrow[r, "\talpha_{n-1}^{Y}\circ\tbeta^Y_n"'] & M_{n-1}^{Y} \arrow[r, "\talpha_{n-2}^{Y}\circ\tbeta_{n-1}^{Y}"'] & M_{n-2}^{Y} \arrow[r, "\talpha_{n-3}^{Y}\circ\tbeta_{n-2}^{Y}"']                     & \cdots \arrow[r, "\talpha_{0}^{Y}\circ\tbeta_{1}^{Y}"'] & M_{0}^{Y}.
	\end{tikzcd}
    \end{equation}
    By Construction~\ref{cons:tG}, $\tGamma_{X_1}=\Sigma[\cmap{X_1}{\varepsilon^{X_1}}{\eta^{X_1}}]\circ\tGamma_{X_1;\, \varepsilon^{X_1}}\circ[\cmap{\Sigma X_1}{\eta^{\Sigma X_1}}{\varepsilon^{\Sigma X_1}}]$, where $\tGamma_{X_1;\, \varepsilon^{X_1}}$ is shown in \eqref{eq:tGXvar}. By Example~\ref{intrinsic}, $\tPn(\tgamma_{0}^{X})$ is the composition of the homotopy class of the chain map shown in \eqref{eq:tg} with $[\cmap{\Sigma X_1}{\varepsilon^{\Sigma X_1}}{\eta^{\Sigma X_1}}]$. Hence, by Lemma~\ref{lem:mutinv}, the composition $\Sigma \tPn(\tdelta)\circ\tGamma_{X_1}\circ \tPn(\tgamma^X_0)$ equals the composition of the chain map in \eqref{eq:tg}, the chain map in \eqref{eq:tGXvar} and the shift of the chain map in \eqref{eq:tdelta}. A direct check shows that this composition is exactly the one shown in \eqref{eq:chainmap}, which completes the proof.
\end{proof}

\section{The main result}\label{section5}

In the previous section, we proved that for an $n$-rigid subcategory $\m$ of an algebraic triangulated category $\t$, there exists an extriangulated functor $(\Pn,\Omegan)\colon\pr{\t}{n+1}{\m}\to K^{[-n,0]}(\m)$   satisfying $\Pn|_\m=\id_\m$. Recall that $\pr{\t}{n+1}{\m}$ is a reduced $(n-1)$-Auslander extriangulated category. We now proceed to investigate the properties of such a functor within a broader setting.

Throughout this section, let $n$ be a positive integer, and let $(\c,\e)$ be a reduced $(n-1)$-Auslander extriangulated category. We denote by $\m$ the subcategory of $\c$ consisting of all projective objects. Since any direct summand of a projective object is again projective, the subcategory $\m$ is closed under direct summands.

Since the global dimension of $\c$ is at most $n$, each object $X\in\c$ has projective dimension at most $n$, i.e., there exist $n+1$ extriangles:
\begin{equation}\label{eq:gldim}
    X_{i+1}\to M_{i}\to X_{i}\dashrightarrow,\quad 0\le i\le n,
\end{equation}
where $X_0=X$, $M_{i}\in\m$ for all $0\leq i\leq n$, and $X_{n+1}=0$. 

\begin{definition} 
Let $m$ be a non-negative integer. We denote by $\pj{m}$ the subcategory of $\c$ consisting of all objects with projective dimension at most $m$.
\end{definition}

By the above definition, we have a chain of subcategories of $\c$:
\[\m=\pj{0}\subseteq \pj{1} \subseteq \cdots \subseteq \pj{n-1} \subseteq\pj{n}=\c.\]

\begin{exam}
    Let $\c=\pr{\t}{n+1}{\m}$, where $\t$ is a triangulated category and $\m$ is an $n$-rigid subcategory of $\t$. Then $\pj{m}$ coincides with $\pr{\t}{m+1}{\m}$, the $(m+1)$-term subcategory of $\t$ generated by $\m$ introduced in Definition~\ref{def:m+1term}.
\end{exam}

For any $X\in\c$, if the morphism $X\rightarrow 0$ occurs as the first morphism in an extriangle, we denote by $\Sigma X$ the third term in the extriangle, i.e., there exists an extriangle
\[X\to 0\to \Sigma X\dashrightarrow.\]

\begin{prop}\label{prop:extrishift}
Let $l$ be an integer such that $1\le l\le n$. For any $X\in\pj{l-1}$, there exist extriangles 
\begin{equation}\label{eq:extrishift}
    \Sigma^{i}X\rightarrow 0\rightarrow\Sigma^{i+1}X\xdashrightarrow{\delta_{i}^{X}},\ 0\le i\le n-l.
\end{equation}
\end{prop}

\begin{proof}
We proceed by induction on $l$. For the base case $l=1$, we have $\pj{0}=\m$, and the assertion holds since the dominant dimension of $\c$ is at least $n$.

Assume that the assertion holds for $l=s$ for some $s<n$, and consider the case $l=s+1$.  Since $\c$ has enough projectives $\m$, there exists an extriangle
\[X_1\to M\to X\xdashrightarrow{\omega^{X,0}},\]
where $X_{1}\in\pj{s-1}$ and $M\in\m$. To complete the induction, it suffices to show that for any $0\leq i\leq n-s-1$, the existence of an extriangle of the form
\[
\Sigma^{i}X_{1}\to\Sigma^{i}M\to\Sigma^{i}X\xdashrightarrow{\omega^{X,i}}
\]
implies the existence of both $\delta_i^X$ and $\omega^{X,i+1}$.

Applying \cite[Lemma 3.14]{NP} to $\omega^{X,i}$, $\delta_{i}^{X_{1}}$ (whose existence follows from the induction hypothesis) and $\delta_{i}^{M}$ (whose existence follows from the base case $l=1$) yields the left-hand commutative diagram of extriangles below. Then, applying \cite[Definition~2.12~(ET4)]{NP} to $\epsilon_{i}$ (which arises from the left-hand diagram) and $\delta_{i+1}^{X_{1}}$ (whose existence follows from the induction hypothesis) yields the right-hand commutative diagram of extriangles, where the second row ensures the existence of $\delta_i^X$. 
\[\begin{tikzcd}[column sep=.5cm]
\Sigma^{i}X_{1} \arrow[r] \arrow[d, equal] & \Sigma^{i}M \arrow[r] \arrow[d] & \Sigma^{i}X \arrow[d, dashed] \arrow[r, "\omega^{X,i}", dotted] & {} & \Sigma^{i}X \arrow[r] \arrow[d, equal] & \Sigma^{i+1}X_{1} \arrow[d] \arrow[r] & \Sigma^{i+1}M \arrow[d] \arrow[r, "\epsilon_{i}", dotted] & {} \\
\Sigma^{i}X_{1} \arrow[r] & 0 \arrow[d] \arrow[r] & \Sigma^{i+1}X_{1} \arrow[d, dashed] \arrow[r, "\delta_{i}^{X_{1}}", dotted] & {} & \Sigma^{i}X \arrow[r, dashed] & 0 \arrow[d] \arrow[r, dashed] & \Sigma^{i+1}X \arrow[d] \arrow[r, "\delta_{i}^{X}", dotted]       & {} \\ 
& \Sigma^{i+1}M \arrow[r, equal] \arrow[d, "\delta_{i}^{M}", dotted] & \Sigma^{i+1}M \arrow[d, "{\exists\,\epsilon_{i}}", dotted]          &    &                                                      & \Sigma^{i+2}X_{1} \arrow[r, equal] \arrow[d, "\delta_{i+1}^{X_{1}}", dotted] & \Sigma^{i+2}X_{1} \arrow[d, "{\exists\,\eta_{i+1}}", dotted]      &    \\                    & {} & {} &    & & {} & {} &   
\end{tikzcd}\]
If $i<n-s-1$, applying~\cite[Proposition 3.15]{NP} to $\eta_{i+1}$ (which arises from the above right-hand diagram) and $\delta^{X_1}_{i+1}$, we obtain the following commutative diagram of extriangles.
\begin{center}
\begin{tikzcd}
    & \Sigma^{i+1}X_{1} \arrow[r, equal] \arrow[d, dashed] & \Sigma^{i+1}X_{1} \arrow[d] & \\
    \Sigma^{i+1}M \arrow[r, "\id_{\Sigma^{i+1}M}"] \arrow[d, equal] & \Sigma^{i+1}M \arrow[d, dashed] \arrow[r, dashed]           & 0 \arrow[d] \arrow[r, "0", dotted] & {} \\
    \Sigma^{i+1}M \arrow[r] & \Sigma^{i+1}X \arrow[r] \arrow[d, "{\exists\,\omega^{X,i+1}}", dotted] & \Sigma^{i+2}X_{1} \arrow[r, "\eta_{i+1}", dotted] \arrow[d, "\delta_{i+1}^{X_{1}}", dotted] & {} \\
    & {}  & {} &   
\end{tikzcd}
\end{center}
Here, the second column ensures the existence of $\omega^{X,i+1}$, as required.
\end{proof}

For the rest of this section, for each $X\in\pj{n-1}$, we fix an extriangle 
\[X\rightarrow 0\rightarrow\Sigma X\xdashrightarrow{\epsilon_{0}^{X}}.\]
We also set $\epsilon_{i}^{Y}:=\epsilon_{0}^{\Sigma^{i}Y}$ for every $Y\in\pj{n-i-1}$.

\begin{lemma}\label{lem:immdcor}
The following statements hold:
\begin{enumerate}
    \item[(1)] For any $X\in\c$ and $Y\in\pj{n-1}$, the map $(\epsilon_{0}^{Y})_{\#}\colon\Hom_{\c}(X,\Sigma Y)\to \e(X,Y)$ is an isomorphism.
    \item[(2)] For any $X\in\pj{n-1}$ and $Y\in\c$, the map $(\epsilon_{0}^{X})^{\#}\colon\Hom_{\c}(X,Y)\to\e(\Sigma X,Y)$ is an isomorphism.
\end{enumerate}
\end{lemma}

\begin{proof}
    This follows from Proposition~\ref{def:exactseq}.
\end{proof}

\begin{remark}\label{rmk:Sigma}
Let $f\colon X\to Y$ be a morphism in $\pj{n-1}$. Then $f_{*}(\epsilon_0^X)\in\e(\Sigma X,Y)$. By Lemma~\ref{lem:immdcor}~(1), the map $(\epsilon_0^Y)_{\#}\colon\Hom_\c(\Sigma X,\Sigma Y)\to\e(\Sigma X,Y)$ is an isomorphism. We define $\Sigma f\in\Hom_\c(\Sigma X, \Sigma Y)$ to be the unique morphism satisfying $(\Sigma f)^{*}(\epsilon_{0}^{Y}) = (\epsilon_{0}^{Y})_{\#}(\Sigma f) = f_{*}(\epsilon_{0}^{X})$. Thus, $\Sigma$ gives rise to a functor between additive categories
\begin{equation}\label{eq:shiftfunc}
    \Sigma\colon\pj{n-1}\rightarrow\Sigma\pj{n-1},
\end{equation}
where $\Sigma\pj{n-1}$ is the subcategory of $\c$ whose objects are isomorphic to objects of the form $\Sigma X$ for some $X\in\pj{n-1}$. We claim that this functor is an equivalence. Indeed, by definition, this functor is dense and faithful. To show it is full, let $g\colon\Sigma X\to \Sigma Y$ be an arbitrary morphism in $\Sigma\pj{n-1}$. Since $g^*(\epsilon_0^Y)\in\e(\Sigma X,Y)$, by Lemma~\ref{lem:immdcor}~(2), there exists a (unique) $f\in\Hom_{\c}(X,Y)$ such that $f_*(\epsilon_0^X)=(\epsilon_0^X)^{\#}(f)=g^*(\epsilon_0^Y)$, which implies $g=\Sigma f$. Therefore, the functor~\eqref{eq:shiftfunc} is full.

\end{remark}

Certain extriangles in $\c$ can be ``rotated" in the following sense.

\begin{prop}\label{prop:exrotation}
Let $X\xrightarrow{f}Y\xrightarrow{g}Z\xdashrightarrow{\omega}$ be an extriangle in $\c$. If $X\in\pj{n-1}$, there exists an extriangle
    \[
    Y\xrightarrow{g}Z\xrightarrow{h}\Sigma X\xdashrightarrow{f_{*}(\epsilon_{0}^{X})},
    \]
    where $h$ is the unique morphism in $\Hom_{\c}(Z,\Sigma X)$, satisfying $(-h)^{*}(\epsilon_{0}^{X})=\omega$. Moreover, if in addition $Y\in\pj{n-1}$, we have the following extriangle in $\c$:
    \[
    Z\xrightarrow{h}\Sigma X\xrightarrow{-\Sigma f}\Sigma Y\xdashrightarrow{g_{*}(\epsilon_{0}^{Y})}. 
    \]
\end{prop}

\begin{proof}
Applying the dual of~\cite[Proposition 3.15]{NP} to $\omega$ and $\epsilon_0^X$, we obtain the following commutative diagram of extriangles:
\begin{center}
\begin{tikzcd}
X \arrow[r, "f"] \arrow[d]                                                  & Y \arrow[r, "g"] \arrow[d, "g", dashed]                             & Z \arrow[d, equal] \arrow[r, "\omega", dotted] & {} \\
0 \arrow[d] \arrow[r, dashed]                                               & Z \arrow[r, "\id_Z"] \arrow[d, "{\exists\, h}", dashed] & Z \arrow[r, "0", dotted]                                     & {} \\
\Sigma X \arrow[r, equal] \arrow[d, "\epsilon_{0}^{X}", dotted] & \Sigma X \arrow[d, "f_{*}(\epsilon_{0}^{X})", dotted]                 &                                                              &    \\
{}                                                                          & {}                                                                  &                                                              &   
\end{tikzcd}
\end{center}
which satisfies $h^{*}(\epsilon_{0}^{X})+(\id_{Z})^{*}(\omega)=0$ in $\e(Z,X)$. Consequently, $(-h)^{*}(\epsilon_{0}^{X})=\omega$. Since $\omega=(-h)^{*}(\epsilon_{0}^{X})=(\epsilon_{0}^{X})_{\#}(-h)$, the uniqueness of $h$ follows from Lemma~\ref{lem:immdcor} (1). 

Similarly, rotating the extriangle $Y\xrightarrow{g}Z\xrightarrow{h}\Sigma X\xdashrightarrow{f_{*}(\epsilon_{0}^{X})}$, we obtain an extriangle:
\[
Z\xrightarrow{h}\Sigma X\xrightarrow{\rho}\Sigma Y\xdashrightarrow{g_{*}(\epsilon_{0}^{Y})},
\]
where $\rho$ is the unique morphism in $\Hom_{\c}(\Sigma X,\Sigma Y)$ satisfying $(-\rho)^{*}(\epsilon_{0}^{Y})=f_{*}(\epsilon_{0}^{X})$. By Remark~\ref{rmk:Sigma}, we have $\rho=-\Sigma f$.
\end{proof}

\begin{corollary}\label{cor:exinject}
An object $X\in\c$ is injective if and only if $X\in\Sigma^{n}\m$.
\end{corollary}

\begin{proof}
Suppose first that $X=\Sigma^{n}M$ with $M\in\m$. Then we have the following extriangles:
\[
\Sigma^{i}M\rightarrow 0\rightarrow\Sigma^{i+1}M\xdashrightarrow{\epsilon_{i}^{M}},\quad 0\le i\le n-1. 
\]
For any $C\in\c$, applying~\cite[Proposition 3.20]{HLN2} to the above extriangles, we have
\[
\e(C,X)=\e(C,\Sigma^{n}M)\cong\e^{2}(C,\Sigma^{n-1}M)\cong\cdots\cong\e^{n+1}(C,M)=0,
\]
where the last equality follows from $\c=\pj{n}$. Hence $X$ is injective in $\c$. 

Conversely, suppose that $X$ is an injective object in $\c$. Since $\c=\pj{n}$, there are $n$ extriangles:
\[
X_{i+1}\rightarrow M_{i}\rightarrow X_{i}\dashrightarrow,\quad 0\le i\le n-1,
\]
where $X_{0}=X$, $M_{i}\in\m$ for all $0\le i\le n-1$, and $X_{n}\in\m$. 

By Proposition~\ref{prop:exrotation}, we have the following extriangles:
\[
\Sigma^{i}X_{i}\xrightarrow{f_{i}}\Sigma^{i+1}X_{i+1}\rightarrow\Sigma^{i+1} M_{i}\dashrightarrow,\quad 0\le i\le n-1.
\]
Set $f=f_{n-1}\circ f_{n-2}\circ\cdots\circ f_{0}\colon X=X_0\to \Sigma^n X_n$. Since each $f_{i}$ is an inflation in $\c$ (i.e. the first morphism in an extriangle) for all $0\le i\le n-1$, by applying \cite[Definition~2.12~(ET4)]{NP} repeatedly, $f$ is also an inflation in $\c$, that is, there exists an extriangle:
\[
X\xrightarrow{f}\Sigma^{n}X_{n}\rightarrow Y\dashrightarrow. 
\]
Since $X$ is injective, $X$ is a direct summand of $\Sigma^{n}X_{n}\in\Sigma^{n}\m$. Hence $X\in\Sigma^{n}\m$. 
\end{proof}

The following proposition shows that certain morphisms in $\c$ can be completed into extriangles in $\c$. 

\begin{prop}\label{prop:complete}
Let $f:X\rightarrow Y$ be a morphism in $\c$. If $X\in\pj{n-1}$, then $f$ is an inflation in $\c$. Dually, if $Y=\Sigma Y'$ with $Y'\in\pj{n-1}$, then $f$ is a deflation in $\c$. 
\end{prop}

\begin{proof}
We only prove the first statement since the second can be proved dually. By Proposition~\ref{prop:extrishift}, we have the extriangle $X\to 0\to\Sigma X\xdashrightarrow{\epsilon_{0}^{X}}$ in $\c$. Consider an extriangle $Y\xrightarrow{g}Z\xrightarrow{h}\Sigma X\xdashrightarrow{f_{*}(\epsilon_{0}^{X})}$, realizing the element $f_{*}(\epsilon_{0}^{X})\in\e(\Sigma X,Y)$. By~\cite[Proposition 3.17]{NP}, we obtain the following commutative diagram of extriangles: 
\begin{center}
\begin{tikzcd}
X \arrow[r] \arrow[d, "f'", dashed]                                            & 0 \arrow[r] \arrow[d]                           & \Sigma X \arrow[d, equal] \arrow[r, "\epsilon_{0}^{X}", dotted] & {} \\
Y \arrow[d, "g", dashed] \arrow[r, "g"]                                        & Z \arrow[d, "\id_Z"] \arrow[r, "h"] & \Sigma X \arrow[r, "f_{*}(\epsilon_{0}^{X})", dotted]                         & {} \\
Z \arrow[r, equal] \arrow[d, "(-h)^{*}(\epsilon_{0}^{X})", dotted] & Z \arrow[d, "0", dotted]                        &                                                                             &    \\
{}                                                                             & {}                                              &                                                                             &   
\end{tikzcd}
\end{center}
and $f'_{*}(\epsilon_{0}^{X})=f_{*}(\epsilon_{0}^{X})$ holds. This implies that $(\epsilon_{0}^{X})^{\#}(f')=(\epsilon_{0}^{X})^{\#}(f)$. Hence, by Lemma~\ref{lem:immdcor}(2), we have $f=f'$, which means $f$ is an inflation in $\c$.
\end{proof}

In the special case $n=1$, we refer the reader to \cite[Section~4.5]{PPPP} for some of these results.

Now we assume that there exists an extriangulated functor $(\Pn,\Omegan)$ from a reduced $(n-1)$-Auslander extriangulated category $(\c,\e)$ to $(K^{[-n,0]}(\m),\e_{K^{[-n,0]}(\m)})$ such that $\Pn|_{\m}=\id_{\m}$, where $\m$ is the subcategory of $\c$ consisting of the projective objects.

The following lemma describes a relationship between $\Pn$ and $\Sigma$.

\begin{lemma}\label{newlem1}
    Let $X\in\pj{n-1}$. Then the map $\Omegan^{\Sigma X,X}(\epsilon_{0}^{X})\colon \Pn(\Sigma X)\to\Sigma \Pn(X)$ is an isomorphism. Consequently, we have $\Pn(\Sigma^i\m)\simeq\Sigma^i\Pn(\m)$ and $\Pn(\pj{i})\subseteq K^{[-i,0]}(\m)$ for all $0\leq i\leq n$.
\end{lemma}

\begin{proof}
    Applying the extriangulated functor $(\Pn,\Omegan)$ to the extriangle
    $X\to 0\to \Sigma X\xdashrightarrow{\epsilon_{0}^{X}}$
    yields a triangle
    \[\Pn(X)\to 0\to \Pn(\Sigma X)\xrightarrow{\Omegan^{\Sigma X,X}(\epsilon_{0}^{X})}\Sigma\Pn(X)\]
    in $K^b(\m)$. This implies that $\Omegan^{\Sigma X,X}(\epsilon_{0}^{X})$ is an isomorphism.
\end{proof}

The following lemma describes a relationship between $\Pn$ and $\Omegan$.

\begin{lemma}\label{newlem2}
    For any objects $X\in\c$ and $Y\in\pj{n-1}$, we have the following commutative diagram
    \[\xymatrix@C=2cm{
    \Hom_\c(X,\Sigma Y)\ar[r]^{(\epsilon_{0}^{Y})_{\#}(-)}\ar[d]_{\Pn(-)}&\e(X,Y)\ar[d]^{\Omegan^{X,Y}(-)}\\
    \Hom_{K^b(\m)}(\Pn(X),\Pn(\Sigma Y))\ar[r]_{\Omega^{\Sigma Y,Y}(\epsilon_{0}^{Y})\circ-}&\Hom_{K^b(\m)}(\Pn(X),\Sigma \Pn(Y))
    }\]
\end{lemma}

\begin{proof}
    Since $\Omega$ is a natural transformation, for any $f\in\Hom_\c(X,\Sigma Y)$, we have
    \[\Omega^{X,Y}((\epsilon_{0}^{Y})_{\#}(f))=\Omega^{X,Y}(f^{*}(\epsilon_{0}^{Y}))=\Omega^{\Sigma Y,Y}(\epsilon_{0}^{Y})\circ \Pn(f).\]
    Thus, we obtain the required commutative diagram.
\end{proof}

The following lemma extends the equivalence induced by $\Pn$ on $\m$ to $\Sigma^i\m$ for all $0\leq i\leq n$.

\begin{lemma}\label{newlem3}
    Let $M,M'\in\m$. For any $0\leq i\leq n$, the map
    \[\Phi_i\colon\Hom_\c(\Sigma^i M,\Sigma^i M')\to\Hom_{K^b(\m)}(\Pn(\Sigma^i M),\Pn(\Sigma^i M'))\] induced by $\Pn$ is a bijection. For any $1\leq i\leq n$, the map \[\Psi_i:=\Omega^{\Sigma^{i+1}M,\Sigma^i M'}\colon\e(\Sigma^{i+1}M,\Sigma^i M')\to\Hom_{K^b(\m)}(\Pn(\Sigma^{i+1}M),\Sigma\Pn(\Sigma^i M'))\] is a bijection. Moreover, $\Pn$ restricts to an equivalence from $\Sigma^i\m$ to $\Sigma^i\m$ for all $0\leq i\leq n$.
\end{lemma}

\begin{proof}
    We first show that for any $0\leq i\leq n-1$,  $\Phi_i$ is a bijection if and only if $\Psi_i$ is a bijection. Due to $0\leq i\leq n-1$, we have $\Sigma^{i+1} M'\in\Sigma\pj{n-1}$. 
    Since $\Omega$ is a natural transformation, for any $f\in\Hom_\c(\Sigma^i M,\Sigma^i M')$, we have
    \[\Psi_i((\epsilon_{i}^{M})^{\#}(f))=\Psi_i(f_{*}(\epsilon_{i}^{M}))=\Sigma\Phi_i(f)\circ\Omega^{\Sigma^{i+1}M,\Sigma^{i}M}(\epsilon_{i}^{M}).\]
    Hence, we have $(\Psi_i\circ(\epsilon_{i}^{M})^{\#})(-)=\Sigma\Phi_i(-)\circ\Omega^{\Sigma^{i+1}M,\Sigma^{i}M}(\epsilon_{i}^{M})$. By Lemma~\ref{lem:immdcor} (2) and Lemma~\ref{newlem1}, the maps $(\epsilon_{i}^{M})^{\#}$ and $\Omega^{\Sigma^{i+1}M,\Sigma^{i}M}(\epsilon_{i}^{M})$ are isomorphisms. Therefore, $\Phi_i$ is a bijection if and only if $\Psi_i$ is a bijection.

    We next show that for any $0\leq i\leq n-1$, $\Psi_i$ is a bijection if and only if $\Phi_{i+1}$ is a bijection. Applying Lemma~\ref{newlem2} to the case that $X=\Sigma^{i+1}M$ and $Y=\Sigma^i M'$, we obtain $\Psi_i((\epsilon_{i}^{M'})_{\#}(-))=\Omega^{\Sigma^{i+1}M',\Sigma^i M'}(\epsilon_{i}^{M'})\circ\Phi_{i+1}(-)$. Therefore, $\Psi_i$ is a bijection if and only if $\Phi_{i+1}$ is a bijection.

    Now, starting with the assumption $\Pn|_\m=\id_\m$, which shows that $\Phi_0$ is a bijection, we deduce recursively the required bijections, using the above results.

    The last statement follows from Lemmas~\ref{newlem1} and the bijections $\Phi_i$.
\end{proof}

The fullness of $\Pn$ implies its denseness.

\begin{lemma}\label{dense}
	Suppose that the functor $\Pn$ is full. Then for any $Y\in K^{[-l,0]}(\m)$ with an integer $1\leq l\leq n$, there exists an object $X$ in $\pj{l}$ such that $\Pn(X)\cong Y$. In particular, the functor $\Pn$ is dense.
\end{lemma}
	
\begin{proof}
	We proceed by induction on $l$. The starting case $l=0$ follows from that $\Pn|_\m=\id_\m$. 
		
	Suppose the assertion holds for $l\le s-1$, where $1\leq s\leq n$, and consider the case $l=s$. Since $Y\in K^{[-s,0]}(\m)$, there exists a triangle
    \[Y_{1}\xrightarrow{g}M_{0}\rightarrow Y\rightarrow\Sigma Y_{1}\]
    in $K^{b}(\m)$, where $Y_{1}\in K^{[-s+1,0]}(\m)$ and $M_{0}\in\m$. By the induction hypothesis, there is $X_{1}\in\pj{s-1}$ such that there exists an isomorphism $\varepsilon\colon\Pn(X_{1})\to Y_{1}$. Since $\Pn$ is full, there exists a morphism $f\in\Hom_{\c}(X_{1},M_{0})$ such that $\Pn(f)=g\circ\varepsilon$ (noting that $\Pn(M_0) = M_0$ since $M_0 \in \m$). Applying Proposition~\ref{prop:complete} to $f$ yields an extriangle 
    \[X_{1}\xrightarrow{f}M_{0}\rightarrow X\dashrightarrow\]
    in $\c$. Then by definition, we have $X\in\pj{s}$. Applying the extriangulated functor $(\Pn,\Omegan)$ to the above extriangle, we obtain a triangle in $K^b(\m)$:
    \[\Pn(X_{1})\xrightarrow{g\circ\varepsilon}\Pn(M_0)\rightarrow\Pn(X)\rightarrow\Sigma \Pn(X_{1}).\]
    Since $\varepsilon$ is an isomorphism, it follows that $\Pn(X)\cong Y$, which completes the induction. 
\end{proof}

We shall need the following lemma. 

\begin{lemma}\label{lem:five}
	Consider the following commutative diagram of abelian groups with exact rows
	\[
	\begin{tikzcd}
		M_1 \arrow[r,"a_1"] \arrow[d,"c_1"] & M_2 \arrow[r,"a_2"] \arrow[d,"c_2"] & M_3 \arrow[r,"a_3"] \arrow[d,"c_3"] & M_4 \arrow[r,"a_4"] \arrow[d,"c_4"] & M_5 \arrow[d,"c_5"] \\
			N_1 \arrow[r,"b_1"] & N_2 \arrow[r,"b_2"] & N_3 \arrow[r,"b_3"] & N_4 \arrow[r,"b_4"] & N_5.
	\end{tikzcd}
	\]
	If $c_4$ and $c_5$ are bijective and $c_1$ and $c_2$ are surjective, then $c_3$ is surjective and $\ker c_3 = \im(a_2|_{\ker c_2})$. In particular, if in addition, $c_2$ is bijective, then so is $c_3$.
\end{lemma}
	
\begin{proof}
    By \cite[Lemma A.4]{BY}, $c_3$ is surjective. A proof of $\ker c_3 = \im(a_2|_{\ker c_2})$ is included in the proof of~\cite[Proposition~A.5]{BY}. We include it here for the convenience of the reader.
		
	For any $\alpha\in\ker c_{2}$, we have $c_{3}(a_{2}(\alpha))=b_{2}(c_{2}(\alpha))=0$. So $\im(a_2|_{\ker c_2})\subseteq\ker c_3$.
		
	Conversely, for any $\beta\in\ker c_{3}$, we have $c_{4}(a_{3}(\beta))=b_{3}(c_{3}(\beta))=0$. Since $c_4$ is bijective (and thus injective), we have $a_3(\beta)=0$. So $\beta\in\ker a_{3}=\im a_2$. That is, there exists an $\alpha\in M_2$ such that $\beta =a_2(\alpha)$. Since $b_2(c_2(\alpha))=c_3(a_2(\alpha))=c_3(\beta)=0$, we have $c_2(\alpha)\in\ker b_2 =\im b_1$. So there exists a $\gamma\in N_1$ such that $c_2(\alpha) = b_1(\gamma)$. Since $c_1$ is surjective, there exists a $\delta\in M_1$ such that $\gamma = c_1(\delta)$. Let $\alpha'=\alpha-a_1(\delta)$. Then we have 
    \[c_2(\alpha') = c_2(\alpha)-c_2(a_1(\delta))=c_2(\alpha)-b_1(c_1(\delta))=b_1(\gamma)-b_1(\gamma) = 0.\]
    This implies $\alpha'\in\ker c_2$. On the other hand, since $a_2 \circ a_1 = 0$, we have $a_2(\alpha')=a_2(\alpha-a_1(\delta))=a_2(\alpha)=\beta$. So $\beta\in\im(a_2|_{\ker c_2})$. Hence, $\ker c_3\subseteq\im(a_2|_{\ker c_2})$.
\end{proof}

Recall that $[\Sigma^n\m,\m]$ denotes the ideal of $\c$ consisting of morphisms factoring through morphisms $\Sigma^{n}M\rightarrow M'$ with $M,M'\in\m$. The main result in this section is the following.

\begin{theorem}\label{finalequiv}
	Let $(\c,\e)$ be a reduced $(n-1)$-Auslander extriangulated category and $\m$ be the subcategory of $\c$ consisting of the projective objects. Let $(\Pn,\Omegan)\colon\c\to K^{[-n,0]}(\m)$ be an extriangulated functor satisfying $\Pn|_\m=\id_\m$. Then the following statements are equivalent.
    \begin{enumerate}
        \item[(i)] The functor $\Pn$ is full.
        \item[(ii)] The map $\Omegan^{X,Y}$ is injective for all objects $X,Y\in\c$.
        \item[(iii)] $\Hom_{\c}(\Sigma^{i}\m,\m)=0$ for all $1\leq i\leq n-1$.
    \end{enumerate}
    If any of the above three equivalent conditions holds, then $\Pn$ is dense, the kernel ideal of $\Pn$ is $[\Sigma^n \m,\m]$, and $\Omegan$ is a natural isomorphism.
\end{theorem}

\begin{proof}
    (i) $\implies$ (iii): Let $f\in\Hom_\c(\Sigma^i M,M')$ for some $1\leq i\leq n-1$ and some objects $M,M'\in\m$. Consider an extriangle realizing $(\epsilon_{i}^{M})^{\#}(f)$ in $\c$:
\[M'\to N\xrightarrow{h}\Sigma^{i+1} M\xdashrightarrow{(\epsilon_{i}^{M})^{\#}(f)}.\]
Applying the extriangulated functor $(\Pn,\Omegan)$ to this extriangle in $\c$, we obtain a triangle in $K^b(\m)$:
\[\Pn(M')\to\Pn(N)\xrightarrow{\Pn(h)}\Pn(\Sigma^{i+1}M)\xrightarrow{\Omegan((\epsilon_{i}^{M})^{\#}(f))}\Sigma\Pn(M').\]
By Lemma~\ref{newlem1}, we have $\Pn(\Sigma^{i+1}M)\cong\Sigma^{i+1}\Pn(M)$, which implies \[\Hom_{K^b(\m)}(\Pn(\Sigma^{i+1}M),\Sigma\Pn(M'))\cong\Hom_{K^b(\m)}(\Sigma^{i+1}\Pn(M),\Sigma\Pn(M'))=0.\]
Thus, $\Omegan((\epsilon_{i}^{M})^{\#}(f))=0$ and hence the identity $\id_{\Pn(\Sigma^{i+1}M)}$ of $\Pn(\Sigma^{i+1}M)$ factors through $\Pn(N)$. That is, there exists a morphism $a\colon \Pn(\Sigma^{i+1}M)\to \Pn(N)$ such that $\id_{\Pn(\Sigma^{i+1}M)}=\Pn(h)\circ a$. Since $\Pn$ is full by assumption, there exists a morphism $g:\Sigma^{i+1}M\to N$ such that $a=\Pn(g)$. Thus, we have
\[\Pn(\id_{\Sigma^{i+1}M})=\id_{\Pn(\Sigma^{i+1}M)}=\Pn(h\circ g).\]
By Lemma~\ref{newlem3}, $\Pn$ restricts to an equivalence from $\Sigma^{i+1}\m$ to $\Sigma^{i+1}\m$. Therefore, we have $\id_{\Sigma^{i+1}M}=h\circ g$. It follows that $(\epsilon_{i}^{M})^{\#}(f)=0$. So we have $f=0$ by Lemma~\ref{lem:immdcor} (2).

(ii) $\implies$ (iii): Let $f\in\Hom_\c(\Sigma^i M,M')$ for some $1\leq i\leq n-1$ and some objects $M,M'\in\m$. Applying $\Omegan$ to $(\epsilon_{i}^{M})^{\#}(f)$, we obtain an element $\Omegan^{\Sigma^{i+1} M,M'}((\epsilon_{i}^{M})^{\#}(f))$ in $\Hom_{K^b(\m)}(\Pn(\Sigma^{i+1} M),\Sigma \Pn(M'))$. Since $\Pn(\Sigma^{i+1} M)\cong\Sigma^{i+1}\Pn(M)$ by Lemma~\ref{newlem1}, we have $\Hom_{K^b(\m)}(\Pn(\Sigma^{i+1} M),\Sigma \Pn(M'))=0$, which implies $\Omegan^{\Sigma^{i+1} M,M'}((\epsilon_{i}^{M})^{\#}(f))=0$. Using the injectivity of $\Omegan^{\Sigma^{i+1} M,M'}$, we have $(\epsilon_{i}^{M})^{\#}(f)=0$ and hence $f=0$ by Lemma~\ref{lem:immdcor} (2).

 
The remaining part of this theorem follows from Lemma~\ref{dense} and Proposition~\ref{prop:bi+surj} below.
\end{proof}

We first analyze how the functor acts on $\Hom$ spaces involving objects in $\Sigma^i\m$, $0\leq i\leq n$.
	
\begin{lemma}\label{XtoM}
	Suppose $\Hom_{\c}(\Sigma^{i}\m,\m)=0$ for all $1\leq i\leq n-1$. For any $X\in\c$ and $M\in\m$, the functor $\Pn$ induces a bijection
    \begin{equation}\label{eq:XtolM}
		\Hom_{\c}(X,\Sigma^{l}M)\xrightarrow{\Pn(-)}\Hom_{K^{b}(\m)}(\Pn(X),\Pn(\Sigma^{l}M))
    \end{equation}
    for any integer $1\leq l\leq n$, and a surjective map
	\[
		\Hom_{\c}(X,M)\xrightarrow{\Pn(-)}\Hom_{K^{b}(\m)}(\Pn(X),\Pn(M))
	\]
    whose kernel is $[\Sigma^n\m](X,M)$.
    
    Moreover, we have the following bijections: 
    \[
		\e(X,\Sigma^{l}M)\xrightarrow{\Omegan(-)}\Hom_{K^{b}(\m)}(\Pn(X),\Sigma\Pn(\Sigma^{l}M)), \ 0\leq l\leq n.
	\]
\end{lemma}
	
\begin{proof}
    To prove the first statement, we establish a more general claim.
    
    Claim: For any $X \in \pj{p}$ with $0 \leq p \leq n$, and integers $0 \leq k \leq n-p$, $0 \leq l \leq n$, the map
    \begin{equation}\label{eq:bi0}
        \Hom_{\c}(\Sigma^{k}X,\Sigma^{l}M)\xrightarrow{\Pn(-)}\Hom_{K^{b}(\m)}(\Pn(\Sigma^{k}X),\Pn(\Sigma^{l}M)),
    \end{equation}
	is a bijection if $(k,l)\neq (n-p,0)$, and is a surjective map if $k=n-p$ and $l=0$, whose kernel consists of the morphisms factoring through objects in $\Sigma^{n}\m$. 

	We proceed by induction on $p$. For the base case $p=0$, we have $X\in\m$. Then by Lemma~\ref{newlem1}, we have $\Pn(\Sigma^i X)\cong\Sigma^i \Pn(X)$ and $\Pn(\Sigma^i M)\cong\Sigma^i \Pn(M)$ for any $0\leq i\leq n$. Hence, if $k\neq l$, we have
    \[\begin{array}{rcl}
        \Hom_{K^{[-n,0]}(\m)}(\Pn(\Sigma^{k}X),\Pn(\Sigma^{l}M))&\cong&\Hom_{K^{[-n,0]}(\m)}(\Sigma^{k}\Pn(X),\Sigma^{l}\Pn(M))\\
        &\cong&\Hom_{K^b(\m)}(\Pn(X),\Sigma^{l-k}\Pn(M))\\
        & = &0.
    \end{array}\]
    Hence, if $k=n$ and $l=0$, the map shown in \eqref{eq:bi0} is surjective and its kernel equals its domain $\Hom_\c(\Sigma^n X,M)=[\Sigma^n\m](\Sigma^n X,M)$. If $(k,l)\neq (n,0)$ and $k\neq l$, since $\m$ consists of projective objects in $\c$ and $\Hom_\c(\Sigma^i\m,\m)=0$ for all $1\leq i\leq n-1$, by Remark~\ref{rmk:Sigma}, we have 
    \[\Hom_\c(\Sigma^{k}X,\Sigma^{l}M)\cong
    \begin{cases}
    \Hom_\c(X,\Sigma^{l-k}M)=0, & \text{if } k<l,\\
\Hom_\c(\Sigma^{k-l}X,M)=0, & \text{if } k>l.
    \end{cases}
    \]
    Hence, the map~\eqref{eq:bi0} is bijective in this case. If $k=l$, the bijectivity of this map follows from Lemma~\ref{newlem3}.
    
    Assume that the claim holds for all $p\leq s-1$ for some integer $1\leq s\leq n$, and consider the case $p=s$. 
		
	Since $X\in\pj{s}$, by definition, there exists an extriangle in $\c$:
    \[
    X_1\rightarrow M_0\rightarrow X\dashrightarrow
    \]
    with $M_0\in\m$ and $X_1\in\pj{s-1}$. By repeated application of Proposition~\ref{prop:exrotation}, we obtain three extriangles in $\c$:
    \[\Sigma^k X_1\xrightarrow{\alpha}\Sigma^k M_0\xrightarrow{\beta}\Sigma^k X\dashrightarrow,\]
    \[\Sigma^k M_0\xrightarrow{\beta}\Sigma^k X\xrightarrow{\gamma}\Sigma^{k+1} X_1\dashrightarrow,\]
    \[\Sigma^k X\xrightarrow{\gamma}\Sigma^{k+1} X_1\xrightarrow{-\Sigma\alpha}\Sigma^{k+1} M_0\dashrightarrow.\]
    Applying the extriangulated functor $(\Pn,\Omegan)$ to these three extriangles, we obtain three triangles in $K^b(\m)$:
    \[\Pn(\Sigma^k X_1)\xrightarrow{\Pn(\alpha)}\Pn(\Sigma^k M_0)\xrightarrow{\Pn(\beta)}\Pn(\Sigma^k X)\rightarrow\Sigma\Pn(\Sigma^{k} X_1),\]
    \[\Pn(\Sigma^k M_0)\xrightarrow{\Pn(\beta)}\Pn(\Sigma^k X)\xrightarrow{\Pn(\gamma)}\Pn(\Sigma^{k+1} X_1)\rightarrow\Sigma\Pn(\Sigma^k M_0),\]
    \[\Pn(\Sigma^k X)\xrightarrow{\Pn(\gamma)}\Pn(\Sigma^{k+1} X_1)\xrightarrow{-\Pn(\Sigma\alpha)} \Pn(\Sigma^{k+1} M_0)\rightarrow\Sigma\Pn(\Sigma^k X).\]
    Applying the functors $\Hom_\c(-,\Sigma^{l}M)$ and $\Hom_{K^b(\m)}(-,\Pn(\Sigma^{l}M))$ to these two sequences of three extriangles respectively, we obtain a commutative diagram of exact sequences of abelian groups:
    \begingroup\footnotesize
    \[\begin{tikzcd}[column sep=.45em]
		{(\Sigma^{k+1}M_{0},\Sigma^{l}M)} \arrow[r, "-\circ\Sigma\alpha"] \arrow[d, "f_{1}:=\Pn(-)"']             & {(\Sigma^{k+1}X_{1},\Sigma^{l}M)} \arrow[r, "-\circ\gamma"] \arrow[d, "f_{2}:=\Pn(-)"'] & {(\Sigma^{k}X,\Sigma^{l}M)} \arrow[r, "-\circ\beta"] \arrow[d, "f_{3}:=\Pn(-)"]              & {(\Sigma^{k}M_{0},\Sigma^{l}M)} \arrow[r, "-\circ\alpha"] \arrow[d, "f_{4}:=\Pn(-)"]              & {(\Sigma^{k}X_{1},\Sigma^{l}M)} \arrow[d, "f_{5}:=\Pn(-)"]              \\
		{(\Pn(\Sigma^{k+1}M_{0}),\Pn(\Sigma^{l}M))} \arrow[r, "-\circ\Pn(\Sigma\alpha)"' {yshift=-1.5pt}] & {(\Pn(\Sigma^{k+1}X_{1}),\Pn(\Sigma^{l}M))} \arrow[r, "-\circ\Pn(\gamma)"' {yshift=-1.5pt}]          & {(\Pn(\Sigma^{k}X),\Pn(\Sigma^{l}M))} \arrow[r, "-\circ\Pn(\beta)"' {yshift=-1.5pt}] & {(\Pn(\Sigma^{k}M_{0}),\Pn(\Sigma^{l}M))} \arrow[r, "-\circ\Pn(\alpha)"' {yshift=-1.5pt}] & {(\Pn(\Sigma^{k}X_{1}),\Pn(\Sigma^{l}M))}.
	\end{tikzcd}\]
    \endgroup
    Here, $(Y,Z)$ in the first row denotes $\Hom_\c(Y,Z)$, while it denotes $\Hom_{K^b(\m)}(Y,Z)$ in the second row.
    
	If either $k\ne n-s$ or $l\ne 0$, the maps $f_1$, $f_2$, $f_4$, and $f_5$ are bijections by the induction hypothesis. Hence, by Lemma~\ref{lem:five}, the map $f_3$ is also a bijection. 
		
	If $k=n-s$ and $l=0$, by the induction hypothesis, $f_1$ and $f_2$ are surjective, and $f_{4}$ and $f_{5}$ are bijections. Hence, by Lemma~\ref{lem:five}, the map $f_3$ is surjective and $\ker f_3=\im((-\circ\gamma)|_{\ker f_2})$. By the induction hypothesis, $\ker f_2=[\Sigma^n\m](\Sigma^{n-s+1} X_1, M)$. By Remark~\ref{rmk:Sigma}, we have $\Hom_\c(\Sigma^{n-s}M_0,\Sigma^n\m)\cong\Hom_{\c}(M_{0},\Sigma^{s}\m)=0$, which implies that any morphism from $\Sigma^{n-s}X$ to an object in $\Sigma^n\m$ factors through $\gamma$. Therefore, we have $\ker f_3=[\Sigma^n\m](\Sigma^{n-s} X, M)$. This completes the inductive step, and hence the proof of the claim.

    For the last statement of this lemma, if $l<n$, 
    by Lemma~\ref{newlem2}, for any morphism $f\in\Hom_{\c}(X,\Sigma^{l+1}M)$, we have $\Omega^{X,\Sigma^{l}M}((\epsilon_{0}^{\Sigma^{l}M})_{\#}(f))=\Omega^{\Sigma^{l+1}M,\Sigma^{l}M}(\epsilon_{0}^{\Sigma^{l}M})\circ\Pn(f)$. By Lemma~\ref{lem:immdcor} and Lemma~\ref{newlem1}, $(\epsilon_{0}^{\Sigma^{l}M})_{\#}$ and $\Omega^{\Sigma^{l+1}M,\Sigma^{l}M}(\epsilon_{0}^{\Sigma^{l}M})$ are isomorphisms. Therefore, the bijectivity of $\Omegan$ follows from the bijectivity of the map~\eqref{eq:XtolM}. 
    If $l=n$, on one hand, by Corollary~\ref{cor:exinject}, $\Sigma^{n}M$ is injective in $\c$, which implies
    $\e(X,\Sigma^{n}M)=0.$ 
    On the other hand, we have
    \[\Hom_{K^{b}(\m)}(\Pn(X),\Sigma\Pn(\Sigma^{n}M))\cong\Hom_{K^{b}(\m)}(\Pn(X),\Sigma^{n+1}\Pn(M))=0.\]
    Hence, the map $\Omegan$ is also a bijection in this case.
\end{proof}

\begin{prop}\label{prop:bi+surj}
    Suppose $\Hom_{\c}(\Sigma^{i}\m,\m)=0$ for all $1\leq i\leq n-1$. For any objects $X,Y\in\c$, there exists a bijection
	\[
	\e(X,Y)\xrightarrow{\Omegan^{X,Y}(-)}\Hom_{K^{b}(\m)}(\Pn(X),\Sigma\Pn(Y)),
	\]
    and a surjection
    \[
    \Hom_{\c}(X,Y)\xrightarrow{\Pn(-)}\Hom_{K^{b}(\m)}(\Pn(X),\Pn(Y))
    \]
    whose kernel is $[\Sigma^n\m,\m](X,Y)$.
\end{prop}

\begin{proof}
    We fix $X \in \c$ and prove a more general claim.
    
    Claim: for any $Y\in\pj{q}$ with $0\leq q\leq n$ and integer $0\leq l\leq n-q$, the functor $\Pn$ induces a bijection
	\[
		\Hom_{\c}(X,\Sigma^{l}Y)\xrightarrow{\Pn(-)}\Hom_{K^{b}(\m)}(\Pn(X),\Pn(\Sigma^{l}Y)),
	\]
	if $l\ne 0$, a surjective map
	\[
	\Hom_{\c}(X,Y)\xrightarrow{\Pn(-)}\Hom_{K^{b}(\m)}(\Pn(X),\Pn(Y))
	\]
	whose kernel consists of the morphisms factoring through morphisms $\Sigma^{n}M\rightarrow M'$ with $M,M'\in\m$, and a bijection
	\[
		\e(X,\Sigma^{l}Y)\xrightarrow{\Omegan^{X,\Sigma^{l}Y}(-)}\Hom_{K^{b}(\m)}(\Pn(X),\Sigma\Pn(\Sigma^{l}Y)). 
	\]

    We proceed by induction on $q$. The starting case $q=0$ is Lemma~\ref{XtoM}. We suppose that the claim holds for the case $q\leq s-1$ where $1\leq s\leq n$ is an integer, and consider the case $q=s$. 
    
    Since $Y\in\pj{s}$, there exists an extriangle in $\c$:
    \[
        Y_1\to M\to Y\dashrightarrow,
    \]
    where $M\in\m$ and $Y_1\in\pj{s-1}$. Applying Proposition~\ref{prop:exrotation} to rotate this extriangle, there exists an extriangle in $\c$:
    \begin{equation}\label{eq:triY}
    \Sigma^l Y_1\xrightarrow{\eta}\Sigma^l M\xrightarrow{\varepsilon}\Sigma^l Y\xdashrightarrow{\theta}.
    \end{equation}
    Applying the extriangulated functor $(\Pn,\Omegan)$ to this extriangle, we obtain a triangle in $K^b(\m)$:
    \[\Pn(\Sigma^l Y_1)\xrightarrow{\Pn(\eta)} \Pn(\Sigma^l M)\xrightarrow{\Pn(\varepsilon)} \Pn(\Sigma^l Y)\xrightarrow{\Omegan(\theta)}\Sigma \Pn(\Sigma^l Y_{1}).\]
    Applying the functors $\Hom_\c(X,-)$ and $\Hom_{K^b(\m)}(\Pn(X),-)$ to these two extriangles, respectively, we have the following commutative diagram of exact sequences of abelian groups:
	\begingroup\small
    \[\begin{tikzcd}[column sep=1em]
		{(X,\Sigma^{l}Y_{1})} \arrow[r, "\eta\circ-"] \arrow[d, "f_{1}:=\Pn(-)"']             & {(X,\Sigma^{l}M)} \arrow[r, "\varepsilon\circ-"] \arrow[d, "f_{2}:=\Pn(-)"'] & {(X,\Sigma^{l}Y)} \arrow[r, "\theta_{\#}(-)"] \arrow[d, "f_{3}:=\Pn(-)"]              & {\e(X,\Sigma^{l}Y_{1})} \arrow[r, "\eta_{*}(-)"] \arrow[d, "f_{4}:=\Omegan(-)"]             & {\e(X,\Sigma^{l}M)} \arrow[d, "f_{5}:=\Omegan(-)"]           \\
		{(\Pn(X),\Pn(\Sigma^{l}Y_{1}))} \arrow[r, "\Pn(\eta)\circ-"' {yshift=-1.5pt}] & {(\Pn(X),\Pn(\Sigma^{l}M))} \arrow[r, "\Pn(\varepsilon)\circ-"' {yshift=-1.5pt}]                                & {(\Pn(X),\Pn(\Sigma^{l}Y))} \arrow[r, "\Omegan(\theta)\circ-"' {yshift=-1.5pt}] & {(\Pn(X),\Sigma\Pn(\Sigma^{l}Y_{1}))} \arrow[r, "\Sigma\Pn(\eta)\circ"' {yshift=-1.5pt}] & {(\Pn(X),\Sigma\Pn(\Sigma^{l}M))},
	\end{tikzcd}\]
    \endgroup
    where the commutativity of the first two squares follows from the fact that $\Pn$ is a functor, and the commutativity of the last two squares follows from that $\Omegan$ is a natural transformation. If $l\neq 0$, by the induction hypothesis, the maps $f_1$, $f_2$, $f_4$ and $f_5$ are bijections. Hence, $f_3$ is also a bijection. If $l=0$, by the induction hypothesis, $f_1$ and $f_2$ are surjective, and $f_4$ and $f_5$ are bijections, and by Lemma~\ref{XtoM}, $\ker f_2=[\Sigma^n \m](X,M)$. Therefore, by Lemma~\ref{lem:five}, we have
    \[\ker f_3 = \im (\varepsilon\circ-)|_{\ker f_2}=[\Sigma^n\m,\m](X,Y),\]
    where the last equality follows from the fact that $\varepsilon$ is a right $\m$-approximation of $Y$ in this case.
    
    Applying Proposition~\ref{prop:exrotation} to rotate the extriangle~\eqref{eq:triY}, there exists an extriangle in $\c$:
    \[\Sigma^l Y\xrightarrow{\alpha} \Sigma^{l+1} Y_1\xrightarrow{\beta} \Sigma^{l+1} M\xdashrightarrow{\gamma}.\]
    Applying the extriangulated functor $(\Pn,\Omegan)$ to this extriangle, we obtain a triangle in $K^b(\m)$:
    \[\Pn(\Sigma^l Y)\xrightarrow{\Pn(\alpha)}\Pn(\Sigma^{l+1} Y_{1})\xrightarrow{\Pn(\beta)}\Pn(\Sigma^{l+1} M)\xrightarrow{\Omegan(\gamma)}\Sigma\Pn(\Sigma^l Y).\]
    Applying $\Hom_{\c}(X,-)$ and $\Hom_{K^b(\m)}(\Pn(X),-)$ to these two extriangles, respectively, 
    there exists a commutative diagram of exact sequences of abelian groups:
    \begingroup\footnotesize
    \[\begin{tikzcd}[column sep=1em]
		{(X,\Sigma^{l+1}Y_{1})} \arrow[r, "\beta\circ-"] \arrow[d, "g_1:=\Pn(-)"']             & {(X,\Sigma^{l+1}M)} \arrow[r, "\gamma_{\#}(-)"] \arrow[d, "g_2:=\Pn(-)"'] & {\e(X,\Sigma^{l}Y)} \arrow[r, "\alpha_{*}(-)"] \arrow[d, "g_3:=\Omegan(-)"]             & {\e(X,\Sigma^{l+1}Y_{1})} \arrow[r, "\beta_{*}(-)"] \arrow[d, "g_4:=\Omegan(-)"]             & {\e(X,\Sigma^{l+1}M)} \arrow[d, "g_5:=\Omegan(-)"]             \\
		{(\Pn(X),\Pn(\Sigma^{l+1}Y_{1}))} \arrow[r, "\Pn(\beta)\circ-"' {yshift=-1.5pt}] & {(\Pn(X),\Pn(\Sigma^{l+1}M))} \arrow[r, "\Omegan(\gamma)\circ-"' {yshift=-1.5pt}]                      & {(\Pn(X),\Sigma\Pn(\Sigma^{l}Y))} \arrow[r, "\Sigma\Pn(\alpha)\circ-"' {yshift=-1.5pt}] & {(\Pn(X),\Sigma\Pn(\Sigma^{l+1}Y_{1}))} \arrow[r, "\Sigma\Pn(\beta)\circ-"' {yshift=-1.5pt}] & {(\Pn(X),\Sigma\Pn(\Sigma^{l+1}M))},
	\end{tikzcd}\]
    \endgroup
    where the commutativity of the first square follows from the fact that $\Pn$ is a functor, and the commutativity of the last three squares follows from that $\Omegan$ is a natural transformation. Since by the induction hypothesis, $g_{1}$, $g_{2}$, $g_{4}$ and $g_{5}$ are bijections, the map $g_{3}$ is also bijective by Lemma~\ref{lem:five}. 
\end{proof}

For the reduced $(n-1)$-Auslander extriangulated category $(\c,\e)$, $\m$ is the subcategory consisting of all projective objects, and by Corollary~\ref{cor:exinject}, $\Sigma^{n}\m$ is the subcategory consisting of all injective objects. Consequently, the ideal $[\Sigma^n\m,\m]$ consists of morphisms from injective objects to projective objects. By \cite[Theorem 2.8]{FGPPP}, the additive quotient $\c/[\Sigma^{n}\m,\m]$ is an extriangulated category, where the induced extriangulated structure $\widetilde{\e}$ is given as follows: for any $X,Y\in\c$, the extension group is $\tilde{\e}(X,Y) = \e(X,Y)$, and the $\tilde{\e}$-triangles in $\c/[\Sigma^n\m,\m]$ are of the form
\[
    Y\xrightarrow{\underline{u}}Z\xrightarrow{\underline{v}}X\stackrel{w}{\dashrightarrow},
\]
for any extriangle $Y\xrightarrow{u}Z\xrightarrow{v}X\xrightarrow{w}\Sigma Y$ in $\c$, where $\underline{u}$ and $\underline{v}$ denote the images of the morphisms $u$ and $v$ under the quotient functor $\pi\colon \c\to\c/[\Sigma^n\m,\m]$. Consequently, this quotient functor, together with the identity natural transformation, forms an extriangulated functor.

\begin{theorem}\label{newcor}
    Let $\c$ be a reduced $(n-1)$-Auslander extriangulated category and let $\m$ be the subcategory of $\c$ consisting of projective objects in $\c$. Let $(\Pn,\Omegan)\colon\c\to K^{[-n,0]}(\m)$ be an extriangulated functor satisfying $\Pn|_\m=\id_\m$. Suppose $\Hom_\c(\Sigma^i\m,\m)=0$ for all $1\leq i\leq n-1$.
    \begin{enumerate}
        \item The extriangulated functor $(\Pn,\Omegan)$ induces an extriangle equivalence 
        \[\c/[\Sigma^n \m,\m]\xrightarrow{\simeq} K^{[-n,0]}(\m).\]
        \item For any objects $X$ and $Y$ in $\c$, there exist isomorphisms
        \[\e^i(X,Y)\cong \Hom_{K^b(\m)}(\Pn(X),\Sigma^i \Pn(Y)),\ \text{for all $i>0$.}\]
    \end{enumerate}
\end{theorem}

\begin{proof}
    By Theorem~\ref{finalequiv}, the assertion (1) follows by the induced extriangulated structure of $\c$, while the assertion (2) follows by Lemma~\ref{lem:siltiff}. 
\end{proof}

Since $\m$ consists of projective objects in $\c$, by Lemma~\ref{lem:immdcor} we have
\[
\Hom_{\c}(\m,\Sigma^{n}\m)=\e(\m,\Sigma^{n-1}\m)=0.
\]
Hence, the ideal $[\Sigma^n \m, \m]$ satisfies $[\Sigma^n \m, \m]^2 = 0$. Consequently, we have the following immediate consequence of Theorem~\ref{newcor}.

\begin{corollary}\label{cor:3bi}
    Keeping the setup of Theorem~\ref{newcor}, for any objects $X, Y$ in $\c$, if $\Pn(X)\cong\Pn(Y)$, then $X\cong Y$. In particular, $\Pn$ induces bijections between:
    \begin{enumerate}
        \item[(a)] the isoclasses of objects in $\c$ and those in $K^{[-n,0]}(\m)$;
        \item[(b)] the subcategories of $\c$ and $K^{[-n,0]}(\m)$ (respectively, those closed under direct summands);
        \item[(c)] the thick subcategories of $\c$ and $K^{[-n,0]}(\m)$.
    \end{enumerate}
\end{corollary}

\begin{remark}
    If the $(n-1)$-Auslander extriangulated category $\c$ is idempotent complete, then the extriangulated category $K^{[-n,0]}(\m)$ is idempotent complete. In this case, the bijection in Corollary~\ref{cor:3bi} restricts to a bijection between the isoclasses of indecomposable objects in $\c$ and those in $K^{[-n,0]}(\m)$.
\end{remark}

We further investigate the preservation of silting structures under the functor $\Pn$.

\begin{theorem}\label{silt_corres}
	Keeping the setup of Theorem~\ref{newcor}, the functor $\Pn$ induces a bijection between presilting subcategories of $\c$ and those of $K^{[-n,0]}(\m)$, which restricts to a bijection between silting subcategories of $\c$ and those of $K^{[-n,0]}(\m)$. 
    
    Moreover, let $\X$ be a silting subcategory of $\c$, and let $\Y$ be a good covariantly (resp., contravariantly) finite subcategory of $\X$. Then $\Pn(\Y)$ is a good covariantly (resp., contravariantly) finite subcategory of $\Pn(\X)$, and 
    \[\Pn(\mu^L(\X;\Y))=\mu^L(\Pn(\X);\Pn(\Y))\ \ \text{(resp., $\Pn(\mu^R(\X;\Y))=\mu^R(\Pn(\X);\Pn(\Y))$)}.\]
\end{theorem}
	
\begin{proof}
	Note that $\m$ is the subcategory of both $\c$ and $K^{[-n,0]}(\m)$ consisting of their respective projective objects. Hence, by Theorem~\ref{newcor}, for any subcategory $\X$ of $\c$, $\e^i(\X,\X)=0$ for all $i>0$ if and only if $\e_{K^{[-n,0]}(\m)}^i(\Pn(\X),\Pn(\X))=0$ for all $i>0$. Therefore, the bijection in Corollary~\ref{cor:3bi}~(b) restricts to the first bijection in this theorem.

    Let $\X$ be a subcategory of $\c$.  By the bijection in (c) of Corollary~\ref{cor:3bi}, we have
    \[\Pn(\mathrm{thick}_{\c}(\X))=\mathrm{thick}_{K^{[-n,0]}(\m)}(\Pn(\X)).\]
    It follows that the first bijection in this theorem restricts to the second one, i.e., between silting subcategories.

    Let $\X$ be a silting subcategory of $\c$, and let $\Y$ be a good covariantly finite subcategory of $\X$. By definition, for each object $X$ of $\X$, there exists an extriangle in $\c$:
    \[X\xrightarrow{f} Y\to Z_X\dashrightarrow,\]
    where $Y\in\Y$ and $Z_X\in\mu^L(\X;\Y)$. Applying the extriangulated functor $(\Pn,\Omegan)$ to it, we obtain a triangle in $K^{[-n,0]}(\m)$:
    \[\Pn(X)\xrightarrow{\Pn(f)} \Pn(Y)\to \Pn(Z_X)\rightarrow\Sigma\Pn(X).\]
    Since $\mu^L(\X;\Y)$ is a silting subcategory of $\c$, $\Pn(\mu^L(\X;\Y))$ is a silting subcategory of $K^{[-n,0]}(\m)$. Consequently, we have $\e_{K^{[-n,0]}(\m)}(\Pn(Z_X),\Pn(\Y))=0$. Therefore, for any $Y'\in\Y$, we have the following exact sequence obtained from the above extriangle:
    \[\Hom_{K^{[-n,0]}(\m)}(\Pn(Y),\Pn(Y'))\xrightarrow{-\circ\Pn(f)}\Hom_{K^{[-n,0]}(\m)}(\Pn(X),\Pn(Y'))\to 0.\]
    This implies that the morphism $\Pn(f)$ is a left $\Pn(\Y)$-approximation of $\Pn(X)$. Therefore, $\Pn(\Y)$ is a good covariantly finite subcategory of $\Pn(\X)$. Thus, the left mutation $\mu^L(\Pn(\X);\Pn(\Y))$ is well-defined, and it coincides with $\Pn(\mu^L(\X;\Y))$ by Corollary~\ref{cor:3bi}~(b). The case for contravariantly finite subcategories and right mutations follows by a dual argument.
\end{proof}

Now we return to the case of an algebraic triangulated category $\t$ with a $n$-rigid subcategory $\m$. Applying Theorem~\ref{finalequiv},~\ref{newcor} and~\ref{silt_corres} to this case, we have the following theorem.

\begin{theorem}\label{thm:algtri}
	Let $\t$ be an algebraic triangulated category and let $\m$ be an $n$-rigid subcategory of $\t$ which is closed under direct summands. Let $(\Pn,\Omegan)\colon\pr{\t}{n+1}{\m}\to K^{[-n,0]}(\m)$ be an extriangulated functor satisfying $\Pn|_\m=\id_\m$. Then the following statements are equivalent.
    \begin{enumerate}
        \item[(i)] The functor $\Pn$ is full.
        \item[(ii)] The map $\Omegan^{X,Y}$ is injective for all objects $X,Y\in\pr{\t}{n+1}{\m}$.
        \item[(iii)] $\Hom_{\t}(\Sigma^{i}\m,\m)=0$ for all $1\leq i\leq n-1$.
    \end{enumerate}
    If one of the above three equivalent conditions holds, then $\Pn$ is dense, the kernel ideal of $\Pn$ is $[\Sigma^n \m,\m]$, and $\Omegan$ is a natural isomorphism, and this extriangulated functor induces an extriangle equivalence: 
    \[\pr{\t}{n+1}{\m}/[\Sigma^n \m,\m]\rightarrow K^{[-n,0]}(\m),\]
    where $[\Sigma^n \m,\m]$ denotes the ideal of $\pr{\t}{n+1}{\m}$ consisting of morphisms factoring through a morphism from an object in $\Sigma^{n}\m$ to an object in $\m$. 
    
    Moreover, the functor $\Pn$ induces a bijection between the set of silting subcategories of $\pr{\t}{n+1}{\m}$ and the set of silting subcategories of 
	$K^{[-n,0]}(\m)$ (that is, the set of $(n+1)$-term silting subcategories of $K^b(\m)$), which is compatible with mutation.	
\end{theorem}

\begin{exam}
    Let $\t$ be the 2-cluster category with 2-cluster tilting subcategory $\m$ considered in Example~\ref{exm:1}. Note that $\Hom_{\t}(\Sigma\m,\m)\ne0$, that is, (ii) in Theorem~\ref{finalequiv} does not hold. Then $\Pn$ is not full.

    By~\cite[Theorem 3.1]{IY} we have $\pr{\t}{3}{\m}=\t$. For $K^{[-2,0]}(\m)\simeq K^{[-2,0]}(\End_{\t}(P_1\oplus\Sigma P_2)$, since $\End_{\t}(P_1\oplus\Sigma P_2))$ is semisimple, the number of isoclasses of indecomposable objects in $K^{[-2,0]}(\m)$ is $6$. However, the number of isoclasses of indecomposable objects in $\t$  is $8$, which implies that $\Pn$ cannot induce a bijection between the isoclasses of indecomposable objects in $\pr{\t}{3}{\m}$ and those in $K^{[-2,0]}(\m)$. Indeed, in this case, any indecomposable object in $K^{[-2,0]}(\m)$, up to isomorphism, belongs to one of $\m$, $\Sigma\m$ and $\Sigma^2\m$. Therefore, by Lemma~\ref{newlem1}, $\Pn$ is dense.

\end{exam}

The following relationship regarding negative extensions between the two extriangulated categories will be applied in the next section.

\begin{lemma}\label{lem:firstneg}
    Keeping the setup of Theorem~\ref{thm:algtri}, suppose $\Hom_{\t}(\Sigma^n\m, \m) = 0$. Then for any $X,Y\in\pr{\t}{n+1}{\m}$, there exists a surjective map
    \[
    \Hom_{\t}(X, \Sigma^{-1}Y) \to \Hom_{K^b(\m)}(\Pn(X), \Sigma^{-1}\Pn(Y)),\]
    whose kernel is $[\Sigma^{-1}\m](X,\Sigma^{-1}Y)$.
\end{lemma}

\begin{proof}
    Since $Y \in \pr{\t}{n+1}{\m}$, there exists an extriangle in $\nprt(\m)$:
    \[
    Y_1 \xrightarrow{u} M_0 \xrightarrow{v} Y \xrightarrow{w} \Sigma Y_1,
    \]
    where $M_0\in\m$ and $Y_1\in\pr{\t}{n}{\m}$. Applying the extriangulated functor $(\Pn,\Omegan)$ to it, we obtain a triangle in $K^b(\m)$:
    \[
    \Pn(Y_1) \xrightarrow{\Pn(u)} \Pn(M_0) \xrightarrow{\Pn(v)} \Pn(Y) \xrightarrow{\Omegan(w)} \Sigma \Pn(Y_1)
    \]
    Applying $\Hom_{\t}(X, -)$ and $\Hom_{K^b(\m)}(\Pn(X), -)$ to these two triangles respectively, we obtain the following diagram of exact sequences, where the square commutes:
    \[
    \begin{tikzcd}[column sep=1.7em]
    (X,\Sigma^{-1}M_0) \arrow[r, "(\Sigma^{-1}v) \circ -"]& (X, \Sigma^{-1}Y) \arrow[r, "(\Sigma^{-1}w) \circ -"]& (X, Y_1) \arrow[r, "u \circ -"] \arrow[d, "\Pn"] & (X, M_0) \arrow[d, "\Pn"] \\
    0 = (\Pn(X), \Sigma^{-1}\Pn(M_0)) \arrow[r] & (\Pn(X), \Sigma^{-1}\Pn(Y)) \arrow[r, "(\Sigma^{-1}\Omegan(w)) \circ -"] & (\Pn(X), \Pn(Y_1)) \arrow[r, "\Pn(u) \circ -"] & (\Pn(X), \Pn(M_0))
    \end{tikzcd}\]
    Since $\Hom_{\t}(\Sigma^n\m, \m) = 0$, by Theorem~\ref{newcor}, $\Pn$ is an equivalence. Therefore, by the universal property of kernel, there exists uniquely a map
    \[
\Hom_{\t}(X, \Sigma^{-1} Y) \longrightarrow \Hom_{K^b(\m)}(\Pn(X), \Sigma^{-1}\Pn(Y)),
\]
which is surjective and whose kernel is $\operatorname{Im}((\Sigma^{-1}v)\circ-)$. Since $\m$ is $n$-rigid and $Y_1\in\pr{\t}{n}{\m}$, we have $\Hom_\t(\m,\Sigma Y_1)=0$. Hence, $v$ is a right $\m$-approximation of $Y$. Consequently, we have $\operatorname{Im}((\Sigma^{-1}v)\circ-)=[\Sigma^{-1}\m](X, \Sigma^{-1} Y)$, which completes the proof.
\end{proof}

\section{An application}\label{section6}

Following \cite{IY}, an $n$-rigid subcategory $\m$ of a triangulated category $\t$ is called \emph{$n$-cluster tilting} if it is functorially finite and satisfies:
\begin{align*}
    \m &= \{X \in \t \mid \Hom_{\t}(\m, \Sigma^{i}X) = 0, \ 1 \leq i \leq n\} \\
       &= \{X \in \t \mid \Hom_{\t}(X, \Sigma^{i}\m) = 0, \ 1 \leq i \leq n\}.
\end{align*}
An object $M$ in $\t$ is called \emph{$n$-cluster tilting} if $\add M$ is an $n$-cluster tilting subcategory. Throughout this section, $\k$ will be a field. An endofunctor $\mathbb{S}$ of a $\k$-linear category is called a \emph{Serre} functor if for any objects $X$ and $Y$, there exists a functorial isomorphism
\[\Hom(X,Y)\cong D\Hom(Y,\mathbb{S}X),\]
where $D=\Hom_\k(-,\k)$ is the standard dual. Now we state the application of Theorem~\ref{finalequiv} to $n$-cluster tilting subcategories.
	
\begin{corollary}\label{cor:cluster-tilting}
	Let $\t$ be an algebraic triangulated category with an $n$-cluster tilting subcategory $\m$. Suppose $\Hom_{\t}(\Sigma^{i}\m,\m)=0$ for all $1\leq i\leq n-1$. Then there exists an extriangulated functor
	\[
		(\Pn,\Omegan)\colon (\t,\e_\m)\rightarrow K^{[-n,0]}(\m),
	\]
	which induces an extriangle equivalence
	\[
		\t/[\Sigma^n\m,\m]\xrightarrow{\simeq}K^{[-n,0]}(\m). 
	\]
	Furthermore, the functor $\Pn$ is an equivalence if and only if $\Sigma^{n+1}\m=\m$, in which case $K^{[-n,0]}(\m)$ admits a triangulated structure. 
    
    If $\t$ is $\k$-linear and admits a Serre functor $\mathbb{S}$, then $\Pn$ is an equivalence if and only if $\m=\mathbb{S}\m$, which is also equivalent to $\m$ being self-injective, i.e., the projective and injective objects in $\mod \m$ coincide, where $\mod\m$ denotes the category of finitely presented functors $\m^{op}\to\mod\k$.
\end{corollary}
	
\begin{proof}
	By~\cite[Theorem 3.1]{IY}, we have $\t=\pr{\t}{n+1}{\m}$. The first statement follows directly from Theorem~\ref{finalequiv}. Together with \cite[Proposition 3.6]{IO}, this implies the remaining assertions.
\end{proof}

\begin{remark}
    The vanishing condition $\Hom_{\t}(\Sigma^{i}\m,\m)=0$ for all $1\leq i\leq n-1$ in Corollary~\ref{cor:cluster-tilting} is referred to as the ``vanishing of small negative extensions'' (vosnex) property in \cite[Notation~3.5]{IO}. For the historical context and earlier applications of this condition, we refer the reader to \cite{KR07, KR08}.
\end{remark}

\begin{exam}
    Let $Q$ be the quiver $3\to 2\to 1$, and let $\t=D^{b}(\k Q)/(\tau^{-1}\Sigma^{2})$ be the 2-cluster category of $Q$. Let $P_{i}$ denote the indecomposable projective module corresponding to vertex $i$. Let $T=P_{1}\oplus P_{2}\oplus P_{3}$. Then $T$ is a $2$-cluster tilting object in $\t$ satisfying $\Hom_\t(\Sigma T,T)=0$. Set $A:=\End_{\t}(T)$. By Corollary~\ref{cor:cluster-tilting}, we have an equivalence
    \[
    \Pn\colon\t/[\Sigma^{2}T,T]\xrightarrow{\simeq} K^{[-2,0]}(\proj A).
    \]
\end{exam}

\begin{exam}
    Let $Q$ be the quiver 
    \[ n+2 \to n+1\to n\to \cdots \to 2\to 1,\]
    and let $D^{b}(\mod \k Q)$ be the bounded derived category of finitely generated right modules of the path algebra $\k Q$. Let $P_{1}$ denote the projective $\k Q$-module corresponding to vertex $1$. Let 
    \[ \m = \add\{(\tau\Sigma)^{i}(P_{1}) \mid i \in \mathbb{Z}\},\]
    where $\tau$ is the Auslander-Reiten translation. Then $\m$ is an $n$-cluster tilting subcategory of $D^{b}(\mod \k Q)$. Since $\tau^{n+3}=\Sigma^{-2}$, we have $\m=\Sigma^{n+1}\m$. By Corollary~\ref{cor:cluster-tilting}, there exists an extriangle equivalence 
    \[(\Pn,\Omegan)\colon D^{b}(\mod \k Q) \simeq K^{[-n,0]}(\m).\]
    In this case, $K^{[-n,0]}(\m)$ admits a triangulated structure induced by the equivalence $\Pn$ of additive categories. 

    Let $C_{n}(\k Q):=D^{b}(\mod \k Q)/(\tau^{-1}\Sigma^{n})$ be the $n$-cluster category of $Q$. Let 
    \[p\colon D^{b}(\mod \k Q) \to C_{n}(\k Q)\]
    be the canonical functor. Then $p(\m)$ is an $n$-cluster tilting subcategory of $C_{n}(\k Q)$. Note that $p(\m) = \add M$, where 
    \[ M = P_{1} \oplus (\tau\Sigma)(P_{1}) \oplus \dots \oplus (\tau\Sigma)^{n+1}(P_{1}).\]
    The endomorphism algebra $A:=\End_{C_{n}(\k Q)}(M)$ is isomorphic to the quotient of the path algebra of the quiver
    \[
    \begin{tikzcd}
    1 \arrow[rrrr, bend left=25, "\alpha_{n+2}"] & 2 \arrow[l, "\alpha_1"] & 3 \arrow[l, "\alpha_2"] & \dots \arrow[l, "\alpha_3"] & n+2 \arrow[l, "\alpha_{n+1}"]
    \end{tikzcd}
    \]
    by the ideal generated by $\alpha_{i+1}\alpha_i$ (with indices taken modulo $n+2$). Hence, $A$ is a self-injective algebra. Thus, by Corollary~\ref{cor:cluster-tilting}, there exists an extriangle equivalence 
    \[C_{n}(\k Q) \simeq K^{[-n,0]}(p(\m)),\]
    whose underlying equivalence of additive categories induces a triangulated structure on the $(n+1)$-term subcategory $K^{[-n,0]}(\proj A)$ of the bounded homotopy category $K^b(\proj A)$ of finitely generated projective modules over $A$.
\end{exam}

Recall that a $\k$-linear triangulated category $\t$ is called \emph{$n$-Calabi-Yau} if $\Sigma^n$ is a Serre functor of $\t$. 

Throughout the rest of this section, let $\t$ be an algebraic $\k$-linear $(n+1)$-Calabi-Yau triangulated category with an $n$-cluster tilting subcategory $\m$ satisfying $\Hom_\t(\Sigma\m,\m)=0$ for all $1\leq i\leq n-1$. We call an object (resp., a subcategory) of $\t$ \emph{$2$-term} (with respect to $\m$) if it belongs to (resp., is contained in) $\pr{\t}{2}{\m}$.

\begin{prop}\label{prop:3equiv}
    Let $\X$ be a $2$-term functorially finite subcategory of $\t$. Suppose that $\Pn(\X)$ is covariantly finite in $K^{b}(\m)$. Then the following statements are equivalent.
    \begin{enumerate}
        \item $\X$ is an $n$-cluster tilting subcategory of the triangulated category $\t$.
        \item $\X=\{Y\in\t\mid \e_\m^i(Y,\X)=0=\e^i_\m(\X,Y)\ \text{for all $i>0$}\}$.
        \item $\X$ is silting in $(\t,\e_\m)$.
    \end{enumerate}
\end{prop}

\begin{proof}
    (1)$\implies$(3): Suppose that $\X$ is an $n$-cluster tilting subcategory of $\t$. Since $\X\subseteq\pr{\t}{2}{\m}$, by Lemma~\ref{lem:shorter}~(1), we have $\e_\m^i(\X,\X)=0$ for all $i\geq 2$. By definition, we have $\e_\m(\X,\X)\subseteq\Hom_\t(\X,\Sigma\X)=0$. Thus, $\X$ is presilting in $(\t,\e_\m)$.
    
    For any object $M\in\m$, since $\X$ is functorially finite in $\pr{\t}{n+1}{\m}$, there exists a left $\X$-approximation $f\colon M\to X_0$ of $M$. We extend this morphism to a triangle in $\t$:
    \[M\xrightarrow{f} X_0\to M_1\to \Sigma M.\]
    For any $X'\in\X$, applying the functor $\Hom_\t(-,X')$ to this triangle yields a long exact sequence
    \[\begin{array}{l}
         \Hom_\t(X_0,X')\xrightarrow{-\circ f}\Hom_\t(M,X')\to \Hom_\t(M_1,\Sigma X')\to \Hom_\t(X_0,\Sigma X')\to\cdots \\
         \to \Hom_\t(M,\Sigma^i X')\to \Hom_\t(M_1,\Sigma^{i+1}X')\to\Hom_\t(X_0,\Sigma^{i+1}X')\to\cdots.
    \end{array}
    \]
    Since $f$ is a left $\X$-approximation of $M$, the map $-\circ f$ is surjective. Since $\X$ is $n$-cluster tilting, we have $\Hom_\t(X_0,\Sigma^{i+1}X')=0$ for all $0\leq i\leq n-1$. Since $\X\subseteq\pr{\t}{2}{\m}$, we have $\Hom_\t(M,\Sigma^iX')=0$ for all $1\leq i\leq n-1$. Therefore, $\Hom_\t(M_1,\Sigma^{i+1} X')=0$ for all $0\leq i\leq n-1$. This implies that $M_1\in\X$. Note that the triangle above is an extriangle in $(\t,\e_\m)$, because $\Sigma M\in\Sigma\pr{\t}{n}{\m}$. Hence, $\m\subseteq\mathrm{thick}_{(\t,\e_\m)}(\X)$ and thus, $\X$ is silting in $(\t,\e_\m)$.
    
    (3)$\implies$(2): Since $\X$ is silting in $(\t,\e_\m)$, by Theorem~\ref{silt_corres}, $\Pn(\X)$ is silting in $K^{[-n,0]}(\m)$. Let $Y$ be an object in $\t$ such that $\e_\m^i(\X,Y)=0=\e_\m^i(Y,\X)$ for all $i>0$. Then by Theorem~\ref{finalequiv}~(2), we have $\Hom_{K^b(\m)}(\Pn(\X),\Sigma^i \Pn(Y))=0=\Hom_{K^b(\m)}(\Pn(Y),\Sigma^i\Pn(\X))$ for all $i>0$. Hence, by Lemma~\ref{lem:silequiv}, we have $\Pn(Y)\in\Pn(\X)$. By Corollary~\ref{cor:3bi}, this implies that $Y\in\X$.

    (2)$\implies$(1): Let $Y$ be an object in $\pr{\t}{n+1}{\m}$ satisfying $\Hom_\t(\X,\Sigma^i Y)=0$ for all $1\leq i\leq n$, which, by the $(n+1)$-Calabi-Yau property, is equivalent to the condition that $\Hom_\t(Y,\Sigma^i \X)=0$ for all $1\leq i\leq n$. Then by definition, $\e_\m(\X,Y)\subseteq\Hom_\t(\X,\Sigma Y)=0$. Since $\X\subseteq\pr{\t}{2}{\m}$, by Lemma~\ref{lem:shorter}~(1), we have $\e_\m^i(\X,Y)=0$ for all $i\geq 2$. On the other hand, by Lemma~\ref{lem:shorter}~(2), we have $\e_\m^i(Y,\X)=\Hom_\t(Y,\Sigma^i\X)=0$ for all $1\leq i\leq n-1$, and $\e_\m^n(Y,\X)=[\Sigma\pr{\t}{n}{\m}](Y,\Sigma^n\X)=0$. Furthermore, by Lemma~\ref{lem:shorter}~(1), we have $\Hom_\t(Y,\Sigma^i\X)=0$ for all $i>n$. Consequently, $Y\in\X$, which implies that $\X$ is $n$-cluster tilting in $\t$.
\end{proof}

We now apply the preceding results to establish a correspondence between $n$-cluster tilting objects and classical tilting complexes. A silting object $\n$ in $K^b(\m)$ is called \emph{tilting} if $\Hom_{K^b(\m)}(\Sigma^i\n,\n)=0$ for all $i>0$. 

\begin{corollary}\label{cor:cto}
    Let $\t$ be a Hom-finite Krull-Schmidt algebraic $\k$-linear $(n+1)$-Calabi-Yau triangulated category and let $M$ be an $n$-cluster tilting object in $\t$. Suppose that $\Hom_\t(\Sigma^iM,M)=0$ for all $1\leq i\leq n-1$. Let $A=\End_\t(M)$. Then the functor $\Pn$ induces a bijection between the isoclasses of basic $2$-term $n$-cluster tilting objects in $\t$ and the isoclasses of basic $2$-term silting complexes in $K^b(\proj A)$, which is compatible with mutation. 
    
    Assume in addition that $\Hom_\t(\Sigma^n M,M)=0$. Then the above bijection restricts to a bijection between the isoclasses of basic $2$-term $n$-cluster tilting objects $X$ in $\t$ satisfying $\Hom_\t(\Sigma X,X)=0$ and the isoclasses of basic $2$-term tilting complexes in $K^b(\proj A)$. Moreover, for any $2$-term $n$-cluster tilting object $X$ in $\t$, $\Hom_\t(\Sigma X, X)=0$ if and only if $\Hom_\t(\Sigma^i X,X)=0$ for all $1\leq i\leq n$.
\end{corollary}

\begin{proof}
    The first assertion follows from Theorem~\ref{silt_corres} and Proposition~\ref{prop:3equiv}. When $n=1$, the second assertion is \cite[Theorem~3.16~(d)]{Yang}, while the third assertion is trivial. Now we assume $n>1$. Since $\Hom_\t(\Sigma^n M,M)=0$, by Corollary~\ref{cor:cluster-tilting}, the functor $\Pn\colon\t\to K^{[-n,0]}(\proj A)$ is an equivalence of additive categories. Let $X$ be a $2$-term $n$-cluster tilting object in $\t$. Since $X$ is $2$-term, we have $\Hom_\t(X,\Sigma^{-1}\add M)=0$ and $\Hom_\t(\Sigma^i X,X)=0$ for all $2\leq i\leq n-1$. Hence, by Lemma~\ref{lem:firstneg}, there exists an isomorphism
    \[\Hom_\t(\Sigma X,X)\cong\Hom_{K^b(\add M)}(\Sigma\Pn(X),\Pn(X)).\]
    Thus, $\Hom_\t(\Sigma X,X)=0$ if and only if $\Pn(X)$ is tilting. In this case, by \cite[Theorem~2.1]{Al}, $\End_{K^b(\proj A)}(\Pn(X))$ is self-injective, which in turn implies that $\End_\t(X)$ is self-injective. Therefore, by Corollary~\ref{cor:cluster-tilting}, we have $\Hom_\t(\Sigma^n X,X)=0$. This completes the proof.
\end{proof}

\noindent {\bf Acknowledgement. }We would like to thank Xiaofa Chen for his interseting and helpful discussions.

\end{document}